\documentclass[12pt, reqno]{amsart}							

\usepackage[scale=0.75, centering, headheight=14pt]{geometry}

\usepackage{graphicx}								
\usepackage{amssymb}
\usepackage{amsmath,amssymb}
\usepackage{amsthm}
\usepackage{mathrsfs}
\usepackage{paralist}
\usepackage{dsfont}
\usepackage{mdframed}
\usepackage{wrapfig}
\usepackage{lipsum}
\usepackage[utf8x]{inputenc}
\usepackage{pifont}
\usepackage{bbding}
\usepackage{lmodern}
\usepackage{url}
\usepackage{nicefrac}
\usepackage{float}
\usepackage{multicol}
\usepackage{bigints}
\usepackage{scrextend}
\usepackage{verbatim}
\usepackage{ stmaryrd }
\usepackage{constants}
\usepackage{comment}
\usepackage{yfonts}

\newcommand\R{\mathbb{R}}

\newcommand\Z{\mathbb{Z}}
\newcommand\N{\mathbb{N}}

\newcommand\1{\mathds{1}}
\newcommand{\CC}{\mathcal{C}}
\newcommand{\HH}{\mathcal{H}}
\newcommand{\M}{\mathcal{M}}
\newcommand{\EE}{\textbf{E}}
\newcommand{\Aa}{\textbf{A}}
\newcommand{\Anu}{\textbf{A}_\nu}
\newcommand{\VV}{\textbf{V}}
\newcommand{\WW}{\textbf{W}}
\newcommand{\MM}{\mathbb{M}}
\newcommand{\UU}{\textbf{B}}
\newcommand{\BB}{\overline{\textbf{B}}}
\newcommand{\BL}{\textbf{L}}
\newcommand{\pp}{\textbf{p}}
\newcommand{\BC}{\textbf{C}}
\newcommand{\ee}{\textbf{e}}
\newcommand{\RR}{\mathcal{R}}
\newcommand{\TT}{\mathcal{T}}
\newcommand{\DD}{\mathcal{D}}
\newcommand{\Leb}{\mathcal{L}}

\newcommand{\Phii}{\boldsymbol{\Phi}}
\newcommand{\PhiM}{\boldsymbol{\Phi}_\mathcal{M}}
\newcommand{\mm}{\textswab{m}}
\newcommand{\al}{\boldsymbol{\alpha}}
\newcommand{\bet}{\boldsymbol{\beta}}
\newcommand{\del}{\boldsymbol{\delta}}
\newcommand{\gamm}{\boldsymbol{\gamma}}
\newcommand{\omm}{\boldsymbol{\omega}}
\newcommand{\muu}{\boldsymbol{\mu}}
\newcommand{\nuu}{\boldsymbol{\nu}}
\newcommand{\tauu}{\boldsymbol{\tau}}
\newcommand{\ze}{\boldsymbol{\zeta}}

\newcommand{\spt}{\textnormal{spt}}
\newcommand{\id}{\textnormal{id}}

\newcommand{\Div}{\textnormal{div}}
\newcommand{\sgn}{\textnormal{sgn}}
\newcommand{\dist}{\textnormal{dist}}
\newcommand{\graph}{\textnormal{graph}}
\newcommand{\dd}{\textnormal{d}}
\newcommand{\dT}{\textnormal{d} \lVert T \rVert}

\newcommand{\mT}{\lVert T \rVert} 
\newcommand{\mTnu}{\lVert T_\nu \rVert}

\newcommand{\vi}{{v_i^T}}
\newcommand{\wj}{{w_j^T}}
\newcommand{\epnu}{\varepsilon_\nu}
\newcommand{\vin}{v_i^{(\nu)}}
\newcommand{\wjn}{w_j^{(\nu)}}

\newcommand{\LM}[1]{\hbox{\vrule width.2pt \vbox to#1pt{\vfill \hrule width#1pt height.2pt}}}
\newcommand{\LL}{{\mathchoice{\,\LM7\,}{\,\LM7\,}{\,\LM5\,}{\,\LM{3.35}\,}}}

\makeatletter
\newcommand{\extp}{\@ifnextchar^\@extp{\@extp^{\,}}}
\def\@extp^#1{\mathop{\bigwedge\nolimits^{\!#1}}}
\makeatother

\newif\ifshow 
\showfalse 

\ifshow
\includecomment{wrap}
\else
\excludecomment{wrap} 
\fi

\numberwithin{equation}{section}

\usepackage[german, english]{babel}					

\include{biblio.bib}

\newconstantfamily{constCounting}{symbol=$C$}

	\title{Boundary Regularity of minimal oriented Hypersurfaces on a Manifold}
\author{Simone Steinbr\"uchel
}

\address{Simone Steinbr\"uchel \hfill\break Institute of Mathematics, University of Zurich, Winterthurerstrasse 190, 8057 Zurich, Switzerland.}

\email{simone.steinbruechel@math.uzh.ch}

\begin{document}

	\newtheorem{thm}{Theorem}[subsection]							
	\newtheorem{prop}[thm]{Proposition}
	\newtheorem{lem}[thm]{Lemma}
	\newtheorem{cor}[thm]{Corollary}
	\newtheorem{kor}[thm]{Korollar}
	\newtheorem{satz}[thm]{Satz}
	\newtheorem{defn}[thm]{Definition}
	\newtheorem{rmk}[thm]{Remark}
	\newtheorem{bmk}[thm]{Bemerkung}
	\newtheorem{exa}[thm]{Example}
	\newtheorem{bsp}[thm]{Beispiel}
	\newtheorem{nota}[thm]{Notation}
	\newtheorem{aufgabe}[thm]{Aufgabe}
	\newtheorem{exc}[thm]{Exercise}
	\newtheorem{schritt}[thm]{Schritt}
	\newtheorem{step}[thm]{Step}
	\newtheorem{claim}[thm]{Claim}
	\newtheorem{beh}[thm]{Behauptung}
	
	\newtheorem*{thm*}{Theorem}							
	\newtheorem*{cor*}{Corollary}
	\newtheorem*{prop*}{Proposition}
	\newtheorem*{lem*}{Lemma}
	\newtheorem*{defn*}{Definition}
	
	\newmdtheoremenv{thmr}[thm]{Theorem}					
	\newmdtheoremenv{propr}[thm]{Proposition}
	\newmdtheoremenv{lemr}[thm]{Lemma}
	\newmdtheoremenv{corr}[thm]{Corollary}
	\newmdtheoremenv{defnr}[thm]{Definition}
	\newmdtheoremenv{rmkr}[thm]{Remark}
	\newmdtheoremenv{bmkr}[thm]{Bemerkung}
	\newmdtheoremenv{korr}[thm]{Korollar}
	\newmdtheoremenv{satzr}[thm]{Satz}

	\begin{abstract}
		In this article we prove that all boundary points of a minimal oriented hypersurface in a Riemannian manifold are regular, that is, in a neighborhood of any boundary point, the minimal surface is a $\CC^{1, \frac14}$ submanifold with boundary.
	\end{abstract}

	\maketitle
	\thispagestyle{empty}
	
	\tableofcontents

	\newpage

\section{Introduction}
	Minimal surfaces have been studied in the last two centuries by various mathematicians. In the 1930's, T. Rad\'o \cite{Rado} and J. Douglas \cite{douglas} proved the existence of 2-dimensional minimal surfaces in $\R^3$ and for this work, Douglas was awarded the Fields medal. Since then a lot of progress has been made and moreover, a new language was invented in order to understand the higher dimensional case. The language that we use in this article is the one of Geometric Measure Theory, where we see surfaces as currents supported in a Riemannian manifold and area minimizing currents are those having least mass among all currents sharing the same boundary. The existence of such minimizers has been proven by H. Federer and W. Fleming \cite{federerFleming} in the 1960's. However, such a minimizing (integral) current is supported on a rectifiable set and thus a priori can have many singularities.
	
A posteriori singularities are rare. In his Ph.D. thesis \cite{allard_thesis_boundary}, W. Allard proved that, in case the boundary is contained in the boundary of a uniformly convex set and the ambient manifold is the euclidean space, then all boundary points are regular. This means, in a neighborhood of a boundary point, the support of the current is a regular manifold with boundary. Later,  R. Hardt and L. Simon came to the same conclusion in \cite{hardtSimon} when having replaced the assumption of the uniform convexity by the fact that the current is of codimension $1$. However, as the result of Hardt and Simon is stated and proved only in the euclidean space, in this paper, we provide an adaptation of the arguments to the case of general Riemannian manifolds. We show the following
	\begin{thm*}
		Let $U \subset \R^{n+k}$ be open and $T$ an $n$-dimensional locally rectifiable current in $U$ that is  area minimizing in some smooth $(n+1)$-manifold $\M$ and such that $\partial T$ is an oriented $\CC^{2}$ submanifold of $U$. Then for any point $a \in \spt(\partial T)$, there is a neighborhood $V$ of $a$ in $U$ satisfying that $V \cap \spt(T)$ is an embedded $\CC^{1,\frac14}$ submanifold with boundary.
	\end{thm*}
	The theorem of Hardt and Simon is then a case of the one stated above, however we follow their strategy of proof with a few modifications in order to deal with additional error terms coming from the ambient manifold in the main estimates.
	
Notice that the complete absence of singular points only happens at the boundary and only in codimension $1$. Indeed, in 2018, C. De Lellis, G. De Philippis, J. Hirsch and A. Massaccesi showed in \cite{DeLellisBoundary} that in the case of higher codimension and on a general Riemannian manifold, there can be singular boundary points but anyway, the set of regular ones is dense. Moreover, in the interior of an area minimizing current, we know thanks to the works of E. Bombieri, E. De Giorgi, E. Giusti \cite{bombieri}, W. Allard \cite{allard2, allard3} and J. Simons \cite{simonsCone}, that the set of singularities of an $n$-dimensional current in an $(n+1)$-dimensional manifold is of dimension at most $n-7$. In the case of higher codimension, the sharp dimension bound is $n-2$ which was first proven in Almgren's Big regularity paper \cite{almgrenBig} and then revisited and shortened by De Lellis and Spadaro in \cite{delellisMultivaluedFunctions, delellisSpadaroMultivaluedCurrents, delellisSpadaroOne, delellisSpadaroTwo, delellisSpadaroThree}.

\subsection{Overview of the proof}
	We would like to measure how flat a current $T$ is. Therefore we introduce its excess in a cylinder of radius $r$ and denote it by $\EE_C(T,r)$. It is the scaled version of the difference between the mass of the current in a cylinder and the mass of its projection. The main ingredient to deduce the boundary regularity is the fact that this excess scales (up to a small rotation) like $r$ assuming that the curvature of both the boundary of the current $\kappa_T$ and the ambient manifold $\mathbf{A}$ are small.
	
	\begin{thm*}[Excess decay]
		Let $\M$ be a smooth manifold and let $T$ be area minimizing in $\M$ such that $ \displaystyle \max\{ \EE_C(T,1), \mathbf{A}, \kappa_T \} \leq \frac{1}{C}$. Then there is a real number $\eta$ such that for all $0<r< R$ the following holds
		$$ \EE_C(\gamm_{\eta \#}T, r)\leq  C r.$$
	\end{thm*}
The precise statement can be found in \ref{thm:RotExcessDecay}. In order to prove it, we first analyze in section \ref{sec:interior} the current away from the boundary. There we can use results from the interior regularity theory and to find that the current is supported on a union of graphs of functions fulfilling the minimal surface equation. When zooming in (up to rescaling), the boundary (and the ambient manifold) become more flat and therefore, we can find the interior graphs closer to the boundary. The point is then to study what happens in the limit when the graphs on both sides of the boundary grow together. This limiting rescaled functions we call the harmonic blow-ups and they are introduced in section \ref{sec:blowup}.

After proving the uniform convergence of the harmonic blow-ups also at boundary points, we show in a first step that in case the harmonic blow-ups are linear, they coincide on both sides of the boundary, see the collapsing lemma \ref{lem:harmblowequal}. Having proven some technical estimates on the excess (Theorem \ref{thm:CylSpherExc}), the assumption of linearity then is dropped in Theorem \ref{thm:blowupsC2}. This follows by blowing up the harmonic blow-ups a second time. To do so we need to make some estimates on the harmonic blow-ups (Lemma \ref{lem:fDividedY}) to guarantee the existence of this second blow-up. 

Then knowing that the harmonic blow-ups coincide and in fact merge together in a smooth way, we prove the excess decay via a compactness argument: if the excess decay did not hold, there would be a sequence of currents whose blow-ups cannot coincide. Then this decay leads to a $\CC^{1, \frac14}$-continuation up to the boundary of the functions whose graphs describe the current (Corollary \ref{cor:graphsOfGammT}) assuming that the excess and the curvatures are sufficiently small. In section \ref{sec:final} we then collect everything together and deduce that either the current lies only on one side of the boundary or both sides merge together smoothly. In case of a one-sided boundary, Allard's boundary regularity theory \cite{allard3} covers the result.
	
\newpage

\section{Notation and preliminaries}
	\subsection{Notation}
		In this paper, $k$, $m$ and $n$ denote fixed natural numbers with $m \geq1$ and $n, k \geq 2$. $C_1, \dots, C_{80}$ are positive constants depending only on $n$, $k$ and $m$.
		
		\subsubsection{Notation associated with $\R^n$}
			We define the following sets for $y \in \R^n$, $j \in \{1,\dots,n\} $ and any real numbers $r>0$ and $0<\sigma<1$ 
			\begin{align*}
			\UU^n_r(y)&=\{x \in \R^n: |x-y|<r \}, \\
			\BB^n_r(y)&=\{x \in \R^n: |x-y| \leq r \},\\
			\omm_n&=\mathcal{L}^n(\UU^n_1(0)),\\
			\BL&=\{x=(x_1, \dots, x_n) \in \UU^n_1(0): x_n=0\},\\
			\VV&=\{x=(x_1, \dots, x_n) \in \UU^n_1(0): x_n>0\},\\
			\WW&=\{x=(x_1, \dots, x_n) \in \UU^n_1(0): x_n<0\},\\
			\VV_{\sigma}&= \{x\in \textbf{V}: \text{dist}(x, \partial \VV)> \sigma\},\\
			\WW_{\sigma}&= \{x\in \textbf{W}: \text{dist}(x, \partial \WW)> \sigma\},\\
			Y_j&: \R^n \to \R, Y_j(y)=y_j.
			\end{align*}
		
		\subsubsection{Notation associated with $\R^{n+k}$}
			We define the following sets for $a \in \R^{n+k}$, $j \in \{1,\dots,n+k \} $ and any real numbers $\omega$ and $r>0$
			\begin{align*}
			\UU_r&=\{x \in \R^{n+k}: |x|<r \}, \\
			\BB_r&=\{x \in \R^{n+k}: |x|\leq r \}, \\
			\BC_r&=\{x \in \R^{n+k}: |\textbf{p}(x)|\leq r \} \text{ where } \textbf{p}:\R^{n+k} \to \R^n, \textbf{p}(x_1,\dots,x_{n+k})=(x_1,\dots,x_n) ,\\
			\textbf{e}_j &=(0,\dots,0,1,0,\dots,0) \text{ where the 1 is at the $j$-th component},\\
			X_j&: \R^{n+k} \to \R, X_j(x)=x_j,\\
			X&:=(X_1,\dots,X_{n+k}),
			\end{align*}
			For the following maps, we identify $\R^{n+k}$ with $\R^{n+1} \times \R^{k-1}$.
			\begin{align*}
			\boldsymbol{\tau}_a&:\R^{n+k} \to \R^{n+k}, \boldsymbol{\tau}_a(x,y)=(x,y) +a,\\
			\boldsymbol{\mu}_r&:\R^{n+k} \to \R^{n+k}, \boldsymbol{\mu}_r(x,y)=r(x, y),\\
			\boldsymbol{\gamma}_\omega&:\R^{n+k} \to \R^{n+k}, \boldsymbol{\gamma}_\omega(x,y)=\big(x_1,\dots,x_{n-1}, x_n\cos(\omega)-x_{n+1}\sin(\omega), x_n\sin(\omega)+x_{n+1}\cos(\omega),y \big).
			\end{align*}
		
		\subsubsection{Notation associated with the current $T \in \mathcal{R}_n(\R^{n+k})$}	
		
			For any real number $r>0$, we define the cylindrical excess as
			$$\EE_C(T,r)=r^{-n}\MM(T\LL\BC_r)-r^{-n}\MM(\pp_{\#} (T\LL\BC_r))$$ and the spherical excess as
			$$\EE_S(T,r)= r^{-n}\MM(T \LL \BB_r)-\omm_n\Theta^n(\lVert T \rVert,0),$$
			whenever $\displaystyle \Theta^n(\lVert T \rVert,0) = \lim_{r \downarrow 0} \frac{\mT(\UU_r)}{\omm_n r^n}$ exists.
			
			In Chapter \ref{lastproof}, we will see that it suffices to consider only currents with compact support and whose boundary lies on a $(n-1)$-dimensional $\CC^{2}$-graph going through the origin. Namely, we define $\TT$ to be the collection of pairs $(T, \M)$ where $\M$ is an embedded $(n+1)$-manifold and $T\in \mathcal{R}_n(\R^{n+k})$ is an absolutely area minimizing integer rectifiable current for which there exist a positive integer $m$, $\varphi_T, \psi_T \in \CC^{2}\big(\{z \in \R^{n-1}: |z|\leq 2\}\big)$ and a smooth map $\PhiM: \UU^{n+1}_4(0) \to \R^{k-1}$, such that
			\begin{itemize}
				\item $ \displaystyle \{ z \in \BC_3: z \in \M \} = \big\{ (x, \PhiM(x)): x \in \UU^{n+1}_3(0) \big\},$
				\item $ \PhiM(0)=0$ and $D \PhiM(0)= 0$,
				\item $ \Aa \leq 1$,
				\item $ \displaystyle  \spt(T)  \subset \BB_3 \cap \M,$
				\item $ \displaystyle  \MM(T)  \leq 3^n\big(1+m\omm_n\big),$
				\item $ \displaystyle  \Theta^n(\lVert T \rVert,0) =m-\nicefrac{1}{2},$
				\item $ \displaystyle  \pp_{\#} (T\LL \BC _2) =m\big(\EE^n \LL \{y\in \UU^n_2(0): y_n > \varphi_T(y_1,\dots,y_{n-1})\} \big)\\ 
				\hspace*{2.5cm} + (m-1)\big(\EE^n \LL \{y\in \UU^n_2(0): y_n < \varphi_T(y_1,\dots,y_{n-1})\} \big)$
				\item $\varphi_T(0) =0=\psi_T(0),$
				\item $ \displaystyle  \varphi_T(0) =0=\psi_T(0),$
				\item $ \displaystyle  D\varphi_T(0) =0=D\psi_T(0),$
				\item $ \displaystyle  (\partial T) \LL \big\{x\in \R^{n+1}:  |(x_1,\dots,x_{n-1})|<2, |x_n| <2\big\}\\
				\hspace*{1cm} = (-1)^{n+k}{F_T}_{\#}\big(\EE^{n-1} \LL \{z\in \R^{n-1}: |z|<2\}\big),$
				\item $ \displaystyle  \kappa_T  \leq 1,$
			\end{itemize}
			where 
			\begin{align*}
			\Aa &:= \lVert D^2 \Phii_\M \rVert_{\CC^1(\BB_2)},\\
			\EE^j &:= \llbracket  \R^j \times \{0\} \rrbracket  \in \mathcal{R}_j(\R^{n+k}) \textnormal{ for all } j \leq n,\\
			F_T(z)&:=\big(z, \varphi_T(z), \psi_T(z),  \PhiM(z, \varphi_T(z), \psi_T(z)) \big),\\
			\kappa_T&:= \lVert D^2(\varphi_T,\psi_T) \rVert_{\CC^0}
			\end{align*}

		\subsection{First variation and monotonicity}
			We start this paper with the following monotonicity estimates. The first two can be read in \cite[Theorem 3.2]{DeLellisBoundary} and the third one, we prove in Chapter \ref{chapter:proofs}.
			\begin{lem}[Monotonicity Formula]\label{lem:monFormula}
				For $(T, \M) \in \TT$ and $0<r<s<2$, the following holds
				\begin{align*}
				\frac{ \lVert T \rVert (\BB_s)}{s^n} -\frac{ \lVert T \rVert (\BB_r)}{r^n} &-\int_{\BB_s \backslash \BB_r} |X^\perp|^2 |X|^{-n-2} \dT\\
				&= \int_r^s \rho^{-n-1} \left( \int_{\UU_\rho} X^\perp \cdot \overset{\rightarrow}{H} \dT
				+ \int_{\spt(\partial T) \cap \UU_\rho} X \cdot \overset{\rightarrow}{n} \dd \HH^{n-1} \right) \dd \rho,
				\end{align*}
				where $X^\perp$ denotes the component orthogonal to the tangent plane of $T$ and $\overset{\rightarrow}{H}$ the curvature vector of $\M$.
			\end{lem}
		
			\begin{rmk}
				There exists $\Cl{23}$ such that $|\overset{\rightarrow}{H}| \leq \Cr{23} \Aa_\M$.
			\end{rmk}
			
			\begin{lem}\label{lem:expMassMonoton}
				There is a dimensional constant $\Cl{1}>0$ such that for $(T, \M) \in \TT$ and $0<r<2$, the map
				$$ r \mapsto \exp \left( \Cr{1} \left( \Aa_\M + \kappa_T \right)r \right) \frac{\lVert T \rVert (\BB_r)}{r^n \omm_n}$$
				is monotonously increasing.
			\end{lem}
		
			\begin{cor}\label{cor:monotonicity}
				For $(T, \M) \in \TT$ and $0<r<s<2$, the following holds
				\begin{align*}
				\left| \frac{ \lVert T \rVert (\BB_s)}{s^n} -\frac{ \lVert T \rVert (\BB_r)}{r^n} -\int_{\BB_s \backslash \BB_r} |X^\perp|^2 |X|^{-n-2} \dT \right|
				\leq \Cl{2}(\Aa_\M + \kappa_T)(s - r).
				\end{align*}
			\end{cor}
	\	\linebreak
		
			Letting $r \downarrow 0$, we deduce also
			
			\begin{cor}\label{cor:sphericalExc}
				For $(T, \M) \in \TT$ and $0<r<2$, the following holds
				\begin{align*}
				\left| \EE_S(T,r) -\int_{ \BB_r} |X^\perp|^2 |X|^{-n-2} \dT \right|
				\leq \Cl{6}(\Aa_\M  + \kappa_T).
				\end{align*}
			\end{cor}			
	
	\newpage

\section{Interior sheeting and nonparametric estimates}\label{sec:interior}
	In this chapter we prove that the minimizing current is, away from the boundary, supported on graphs.

	\begin{minipage}[t]{\linewidth}
		\hspace{-1cm}
		\centering
		\begin{minipage}{0.68\linewidth}
			\vspace*{0.2cm} 
				\begin{defn}
			Let $u: U \subset \R^n \to \R$. Then we define
			$$\graph (u, \Phii):= \big\{(x, u(x), \Phii(x,u(x))): x \in U \big\}.$$			\end{defn}
			Away from the boundary, the interior regularity theory gives us functions whose graphs describe the current. Moreover they fulfill the Riemannian minimal surface equation (see Definition \ref{def:minEquation}) that is elliptic and therefore, we can deduce estimates on the gradient of these functions. These estimates are crucial as they guarantee the existence of the harmonic blow-ups introduced in section \ref{sec:blowup}.
		\end{minipage}
		\hspace{0.05\linewidth}
		\begin{minipage}[t]{0.25\linewidth}
			\vspace*{-2cm} 
			\includegraphics[scale=0.3]{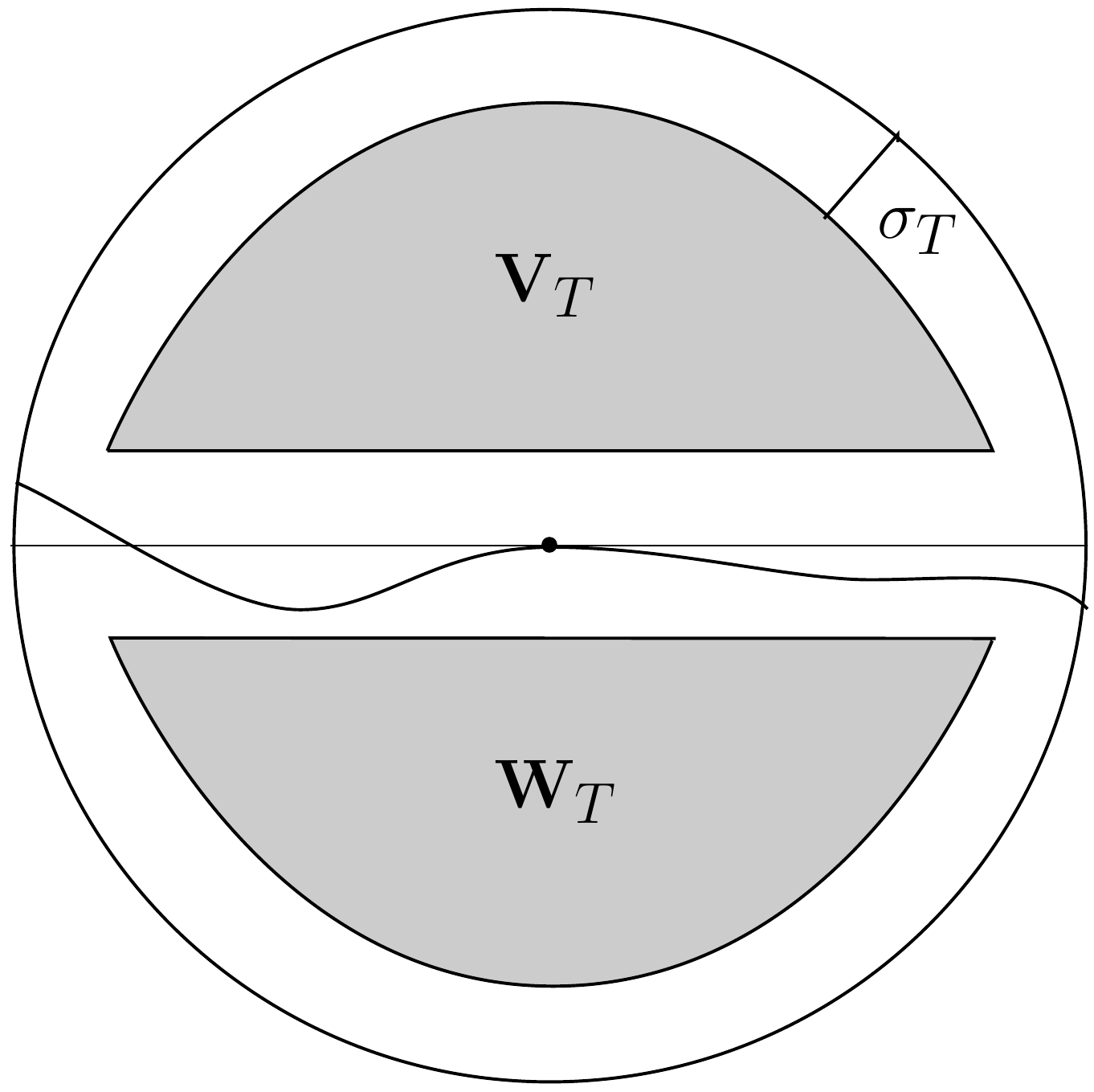}
		\end{minipage}
	\end{minipage}

	\begin{thm}\label{thm:graph}
		Let $(T, \M) \in \TT$ and assume $\Aa \leq 1/4$. Then there are constants $\Cl{30} \geq 12$, $\Cl{31} \geq 1$ such that if 
		$$ \EE_C(T,1) + \kappa_T + \Aa  \leq (4\Cr{30})^{-2n-3}$$
		and we denote $\sigma_T:=\Cr{30}\big(\EE_C(T,1) + \kappa_T + \Aa \big)^{1/(2n+3)}$, $\VV_T:= \VV_{\sigma_T}$ and $\WW_T:=\WW_{\sigma_T}$, then 
		for $i \in \{1,\dots,m\}$, $j \in \{1,\dots,m-1\}$ and $k \in \{1,2,3\}$ there are smooth functions $v_i^T:\VV_T \to \R$ and $w_j^T: \WW_T \to \R$ satisfying the Riemannian minimal surface equation and such that 
		\begin{enumerate}[(i.)]
			\item $ \displaystyle v_1^T \leq v_2^T \leq \cdots \leq v_m^T \qquad \textnormal{and} \qquad w_1^T \leq w_2^T \leq \cdots \leq w_{m-1}^T,$
			\item $ \displaystyle \pp^{-1}(\VV_T) \cap \spt(T) = \bigcup_{i=1}^m \textnormal{graph }( v_i^T, \Phii) \quad \textnormal{and}\quad \pp^{-1}(\WW_T) \cap \spt(T) = \bigcup_{i=1}^{m-1} \textnormal{graph }( w_i^T, \Phii),$
			\item $\displaystyle |D^k\vi(y)| \leq \Cl{33} \sqrt{\EE_C(T,1) + \kappa_T + \Aa} ~ \dist(y, \partial \VV)^{{{-k-n-1/2}}}$ for all $y \in \VV_T$,
			\item $\displaystyle |D^k\wj(y)| \leq \Cr{33} \sqrt{\EE_C(T,1) + \kappa_T +  \Aa} \dist(y, \partial \WW)^{{{-k-n-1/2}}}$ for all $y \in \WW_T$,
			\item $\displaystyle \int_{\VV_T} \left( \frac{\partial}{\partial r} \frac{\vi(y)}{|y|} \right)^2 |y|^{2-n} \dd \Leb^n(y)
			+  \int_{\WW_T} \left( \frac{\partial}{\partial r} \frac{\wj(y)}{|y|} \right)^2 |y|^{2-n} \dd \Leb^n(y) $
			\begin{flalign*}
&\qquad 	\leq 2^{n+7}\EE_S(T,1) + \Cl{34}(\Aa + \kappa_T)&\\
&\qquad \leq 2^{n+7}\EE_C(T,1) + \Cl{57}(\Aa+\kappa_T), \qquad \qquad \textnormal{ where } \frac{\partial}{\partial r} f(y):= \frac{y}{|y|} \cdot Df(y).& \end{flalign*}

		\end{enumerate}
		
	\end{thm}

	For the existence of these graphs, we need to split the current into pieces and show that each piece is supported on a graph. Then, once we have these graphs, we show the estimates by using the regularity theory of elliptic PDEs. This will be done more precisely in subsection \ref{ssection:proof of graph}.
	
	\subsection{Comparison between excess and height}
	
		To prove the estimate in Theorem \ref{thm:graph}$(iii.)$, $(iv.)$, we will deduce from the PDE theory an estimate on the values of the functions $v_i^T$, ($w_j^T$ respectively). This can be translated into the height of the current in the $(n+1)$-component. We wish to estimate the latter quantity with the excess of $T$ and hence, we need the following lemmata comparing the (cylindrical) excess with the height. The proofs are given in chapter \ref{chapter:proofs}.
		
		First notice that as in the original paper \cite{hardtSimon}, we infer that for $0 < r \leq s \leq 2$ the following holds
		\begin{align}\label{eq:excmon} 
		\EE_C(T,r) \leq \left( \frac{s}{r} \right)^n \EE_C(T,s)
		\end{align}
		and 
		\begin{align}\label{eq:spherCylindExc}
		\EE_S(T,r) \leq \EE_C(T,r) + mr \kappa_T.
		\end{align}

		\begin{lem}\label{lem:excLessX}
			There are positive constants $\Cl{18}$ and $\Cl{19}$ such that for all $0<\sigma<1$ and $(T, \M) \in \TT$, the following holds
			$$\frac{\sigma^2}{\Cr{18}} \EE_C(T,1) -\kappa_T - \Aa
			\leq \int_{\BC_{1+\sigma}} X_{n+1}^2 \dT
			\leq \Cr{19}\sup_{\BC_{1+\sigma} \cap \spt(T)} X_{n+1}^2.$$
		\end{lem}

		Not only it is true, that the height bounds the excess, but also the other way around. The following estimates rely on an area comparison lemma (Lemma \ref{lem:areacomp}). Its proof will give us a constant $\Cl{10}$ which we will use to prove the following
		
		\begin{lem}\label{lem:XLessExc}
			If $0<\sigma<1$, $\Aa^2 \leq \sigma/8$ and	$ \Aa \leq (7\Cr{23} + \Cr{10}+1)^{-1}$
			then there are positive constants $\Cl{21}$ and $\Cl{22} \geq 2$ such that for $(T, \M) \in \TT$, the following holds
			\begin{enumerate}[(i.)]
				\item $ \displaystyle \frac{\sigma^n}{\Cr{21}} \sup_{\BC_{1-\sigma} \cap \spt(T)}X_{n+1}^2  \leq \int_{\BC_{1-\sigma/2}} X_{n+1}^2 \dT + \kappa_T.$
				\item $ \displaystyle \int_{\BC_{1-\sigma/2}} X_{n+1}^2 \dT \
				\leq \frac{\Cr{22}-1}{\sigma^{n+1}} \left(\EE_C(T,1)+\kappa_T +\Aa \right).$
			\end{enumerate}
			In particular, we have
			$$ \sup_{\BC_{1-\sigma} \cap \spt(T)}X_{n+1}^2  \leq \frac{\Cr{21}\Cr{22}}{\sigma^{2n+1}} \big(\EE_C(T,1)+\kappa_T +\Aa \big).$$
		\end{lem}

	\subsection{Splitting of the minimizing current $T$}
	Here we prove the fact, that if a current has no boundary, its excess is not too large and the projection has multiplicity $j$, then it consists of $j$ many layers whose projection are of multiplicity $1$.
		\begin{lem}\label{lem:decomp}
			Let $j \in \N_+$, $V \subset \R^n$ be open and consider the cylinder $\Gamma:= \{x \in \R^{n+1}: (x_1, \dots, x_n) \in V\}$  and the modified version $\tilde{\Gamma} := \{(x, \Phii(x)) \in \M: \pp(x) \in V \}$. If $S \in \RR_n(\tilde{\Gamma})$ satisfies
			\begin{itemize}
				\item $(\partial S)\LL \tilde{\Gamma}=0$
				\item $\pp_\#S=j(\EE^n \LL V)$
				\item $\MM(S)-\MM(\pp_\# S) < \HH^n(V),$
			\end{itemize}
			then for all $i \in \{1,\dots,j\}$ there exists $S_i \in \RR_n(\tilde{\Gamma})$ such that
			\begin{equation*}
			\begin{aligned}
			\tilde{\Gamma} \cap \spt(\partial S_i) = \emptyset,\\
			S= \sum_{i=1}^j S_i,
			\end{aligned}
			\qquad \qquad
			\begin{aligned}[c]
			\pp_\#S_i = \EE^n \LL V,\\
			\lVert S \rVert = \sum_{i=1}^j \lVert S_i \rVert.
			\end{aligned}
			\end{equation*}
		\end{lem}
		
		\begin{proof}
			Denote by $\tilde{p}$ the projection to $\R^{n+1}$ and consider
			$ \tilde{S} := \tilde{p}_\# S.$ Then we have
			
			\begin{minipage}[t]{\linewidth}
				\centering
				\begin{minipage}{0.63\linewidth}
					\begin{itemize}
						\item $(\partial \tilde{S}) \LL \Gamma = \left( \tilde{p}_\#(\partial S) \right) \LL \Gamma = \tilde{p}_\#((\partial S) \LL \tilde{\Gamma})= 0$
						\item $\pp_\# \tilde{S} = \pp_\# S=j(\EE^n \LL V)$
						\item $\MM(\tilde{S}) - \MM(\pp_\# \tilde{S}) \leq \MM(S)-\MM(\pp_\#S) \leq \HH^n(V)$.
					\end{itemize}
					Therefore, we can argue as in the original paper \cite{hardtSimon} to deduce a decomposition for $\tilde{S}$: There are	$\tilde{S}_i \in \RR_n(\R^{n+1})$ such that
					\begin{equation*}
					\begin{aligned}
					{\Gamma} \cap \spt(\partial \tilde{S}_i) = \emptyset,\\
					\tilde{S}= \sum_{i=1}^j \tilde{S}_i,
					\end{aligned}
					\qquad \qquad
					\begin{aligned}[c]
					\pp_\#S_i = \EE^n \LL V,\\
					\lVert \tilde{S} \rVert = \sum_{i=1}^j \lVert \tilde{S}_i \rVert.
					\end{aligned}
					\end{equation*}
				\end{minipage}~
				\begin{minipage}[t]{0.35\linewidth}
					\vspace*{-3cm} 
					\includegraphics[scale=0.9]{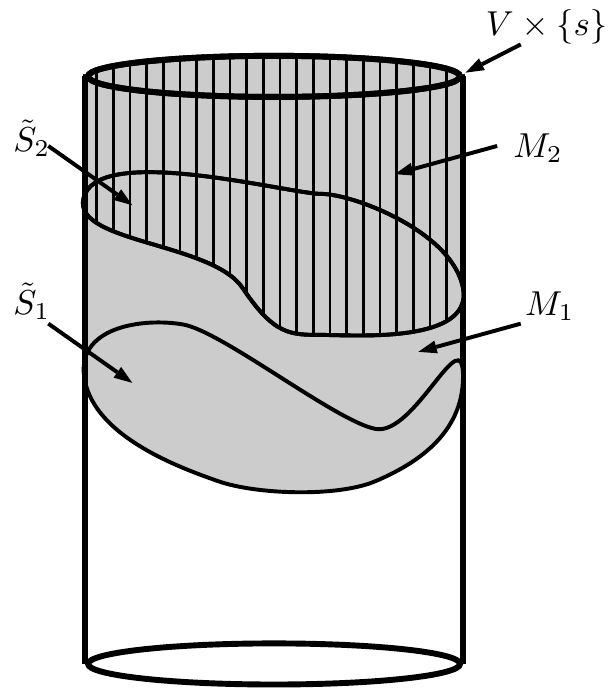}
				\end{minipage}
			\end{minipage}
			We conclude by putting $S_i:= (\id, \Phii)_\# \tilde{S}_i$.
		\end{proof}
	
		In the situation of Theorem \ref{thm:graph}, each of these $S_i$ is area minimizing in $\M$ and so the smallness of the excess implies that locally the  function, whose graph describe $\spt(S_i)$, fulfills an elliptic equation. Thus, we can deduce the following Schauder estimate:
		
		\begin{lem}\label{lem:minSurfEqu}
			Let $U$ be an open neighborhood of $0 \in \R^n$ and $u: U \to \R$ such that $u(0)=0$, $Du(0)=0$ and $\graph (u, \Phii) \subset \M$ is a minimal surface in $\M$. Then there is $r >0$ such that
			$$ r \lVert D u \rVert_{\CC^0(B_{r/2})} + r^2\lVert D^2 u \rVert_{\CC^0(B_{r/2})} + r^{2+ \alpha}\left[ D^2u \right]_{\CC^\alpha(B_{r/2})}
			\leq \Cl{32} \left( \lVert u \rVert_{\CC^0(B_r)} + \lVert D^2\Phii \rVert^*_{\CC^{\alpha}(B_r)} \right),$$
			where $$\lVert f\rVert^*_{\CC^{\alpha}(\Omega)} 
			:= \sup_{x \in \Omega} \dist(x, \partial \Omega)^2|f(x)|
			+ \sup_{\substack{x,y \in \Omega\\ x \neq y}} \max\big\{\dist(x, \partial \Omega), \dist(y, \partial \Omega)\big\}^{2+\alpha} \frac{|f(x)-f(y)|}{|x-y|^\alpha}.$$
		\end{lem}
		
		\begin{proof}
			We use the Euler-Lagrange equation in the form of Schoen-Simon in chapter 1 of \cite{SchoenSimon}. Then we use Gilbarg-Trudinger \cite[Theorem 6.2]{gilbarg}  to deduce the estimate. Indeed, we define
			\begin{align*}
			h&: \R^{n\times(2n+k)} \to \R,\ A \mapsto \sqrt{\det(A^t A)}\\
			\phi&: \graph(u) \times \R^n \to \R^{n\times(2n+k)},\ (z,v) \mapsto \big( \id, v^t, D\Phii(z)(\id, v)^t \big)\\
			g&: S^n \cap \left\{z_{n+1} \geq \frac{1}{\sqrt{1+R^2}} \right\} \to \BB_R, \ z \mapsto \frac{-1}{z_{n+1}}(z_1, \dots, z_n)
			\end{align*}
			and finally
			\begin{align*}
			F: \graph(u) &\times \big( \R^{n+1} \cap \{ p_{n+1} \sqrt{1+R^2} \geq |p|\} \big) \to \R \\ (z,p) &\longmapsto |p_{n+1}| \ h\big(\phi \big(z, g(p/|p|) \big) \big).
			\end{align*}
			Then $F$ is homogeneous in $p$ and moreover
			$$g \left( \frac{(-v,1)}{\sqrt{1+|v|^2}} \right) = v,$$
			$$\int_{\graph(u)} F(z, \nu(z)) \dd \mathcal{H}^n(z) = \int_{\UU_R} h\big(D \big(x, u(x), \phi(x, u(x)) \big) \big) \dd x = Vol(\graph(u,\Phii)).$$
			Then the Euler-Lagrange equation reads
			\begin{align}\label{eq:EulerLagrange}
			\Div \left( \frac{Du}{\sqrt{1+|Du|^2}} \right) = \sum_{i,j=1}^n a_{ij}(x) \partial_{ij} u(x) + b(x),
			\end{align}
			where 
			\begin{align*}
			a_{ij}(x) &= \int_0^1 \sum_{k=1}^{n+1} z_k\ \partial_{z_k p_i p_j} F(tz, -Du, 1) \dd t \qquad \textnormal{ evaluated in } \qquad z=(x, u(x)),\ p=(-Du, 1)\\
			b(x)&= \sum_{i=1}^{n+1} \partial_{z_i,p_i} F(z,p)  \qquad \textnormal{ evaluated in } \qquad z=(x, u(x)),\ p=(-Du, 1).
			\end{align*}
			In order to use elliptic estimates, we define
			$$A_{ij} := \frac{\delta_{ij} (1+ |Du|^2)-\partial_iu \partial_j u}{(1+ |Du|^2)^{3/2}} - a_{ij}$$
			and notice that for $r>0$ small enough, we have $|Du|+ \max_{ij}|a_{ij}| \leq 1/12$ in $\UU_r$ and therefore $\frac{1}{2} \id \leq A \leq 2 \id$ as a quadratic form. The only thing left to do is to notice that
			$$ \lVert b \rVert^*_{\CC^\alpha} \leq \Cl{5000} \lVert D^2\Phii \rVert^*_{\CC^{\alpha}}.$$ 
		\end{proof}
	
		\begin{defn}\label{def:minEquation}
			We define \eqref{eq:EulerLagrange} to be the \emph{Riemannian minimal surface equation}.
		\end{defn}

	\subsection{Proof of Theorem \ref{thm:graph}}\label{ssection:proof of graph}
		\begin{proof}[Proof of Theorem \ref{thm:graph}]
			~ \\
			\begin{minipage}[t]{\linewidth}
				\hspace{-0.2cm}
				\centering
				\begin{minipage}{0.68\linewidth}
					\vspace*{0cm} 
					Let $\varepsilon >0$ be as in \cite[Theorem 5.3.14]{federer} with $\lambda$, $\kappa$, $m$, $n$ replaced by $1$, $1$, $n$, $n+1$ respectively and we choose the parametric integrand to be the one associated to $\M$:
					\begin{align*}
					\Psi: \R^n \times \R \times \Lambda_n(\R^{n+k}) &\longrightarrow \R,\\
					( (x,y),\zeta) \longmapsto &|\zeta| \ h \left( \left(
					\begin{array}{c}
					\id\\
					D\Phii(x,y)\\
					\end{array}
					\right) \big( v_1 \cdots v_n \big) \right),
					\end{align*}
				\end{minipage}
				\hspace{0.05\linewidth}
				\begin{minipage}[t]{0.25\linewidth}
					\vspace*{-2.5cm} 
					\includegraphics[scale=0.3]{Pokeball}
				\end{minipage}
			\end{minipage}
			where $h$ is the map from Lemma \ref{lem:minSurfEqu} and $\{v_1, \dots, v_n\}$ are orthonormal and such that
			$$ v_1 \wedge \cdots \wedge v_n = \frac{\zeta}{|\zeta|}.$$
			We require $\Cr{30}$ to fulfil $(4\Cr{30})^{-2n-3} \leq \Leb^n(\VV_T)$ implying that $\spt(\partial T) \cap \pp^{-1}(\VV_{\sigma_T/3} \cup \WW_{\sigma_T/3}) = \emptyset$, because $\kappa_T < 9^{-n}\sigma_T$. Indeed,
			$$\kappa_T \leq \left( \frac{\sigma_T}{\Cr{30}} \right)^{2n+3}
			\leq\frac{\sigma_T}{9^n} \frac{9^n}{4^{2n+2} \Cr{30}^{2n+3}} \leq\frac{\sigma_T}{9^n}.$$
			Then, we have
			$$\pp_\#\big(T \LL \pp^{-1}(\VV_{\sigma_T/3})\big)=m (\EE^n \LL \VV_{\sigma_T/3}) \quad \textnormal{and} \quad \pp_\#\big(T \LL \pp^{-1}(\WW_{\sigma_T/3})\big)=(m-1) (\EE^n \LL \WW_{\sigma_T/3})$$ 
			and we can apply Lemma \ref{lem:decomp}. We obtain for $i \in \{1,\dots,m\}$, $j \in \{1,\dots,m-1\}$ on $\M$-area minimizing currents $S_i$ and $\tilde{S}_j$ satisfying 
			$$\sum_{i=1}^m S_i= T \LL \pp^{-1}(\VV_{\sigma_T/3}) \quad \textnormal{and} \quad \sum_{j=1}^{m-1} \tilde{S}_j= T \LL \pp^{-1}(\WW_{\sigma_T/3}),$$
			$$\pp_\#S_i= \EE^n \LL \VV_{\sigma_T/3} \quad \textnormal{and} \quad \pp_\#\tilde{S}_j = \EE^n \LL \WW_{\sigma_T/3}.$$
			Denote again by $\tilde p$ the projection to $\R^{n+1}$. Then $\tilde p_\# S_i$ and $\tilde p_\# \tilde{S}_j$ are absolutely $\Psi$-minimizing. Now, we cover $\pp^{-1}(\VV_{\sigma_T/3})$ with cylinders $\BC_{\sigma_T/3}(x)$ for all $x \in \VV_{2 \sigma_T/3} \cap \spt(\pp_\#(S_i))$. In each of these cylinders, we want to use \cite[Corollary 5.3.15]{federer} replacing $\lambda$, $\kappa$, $m$, $n$, $r$, $S$ by $1$, $1$, $n$, $n+1$, $\sigma_T/3$, $\tauu_{-x\#} \tilde p_\#S_i$ respectively. To do so, we must have $(4\Cr{30})^{-2n-3} \leq (\varepsilon/2)^{n+1}$. As a result, we get in each cylinder a solution $u$ of the Riemannian minimal surface equation whose graph forms $\spt(\tilde p_\#S_i) \cap \BB_{\sigma_T/3}(x)$ and hence, $\graph(u, \Phii) \cap \BB_{\sigma_T/3}(x)=\spt(S_i) \cap \BB_{\sigma_T/3}(x)$ . These solutions yield a unique function $\vi$ whose graph on $\M$ is $\spt(S_i) \cap \pp^{-1}(\VV_{2\sigma_T/3}).$ As the integrand is smooth in $(x,y)$, so are the solutions. In $\pp^{-1}(\WW_{\sigma_T/3})$ we argue analogously. By construction of the splitting $\{S_i\}_i$, there is a numbering such that $(i.)$ holds.

			Now, we prove $(iii.)$. We want to use Lemma \ref{lem:minSurfEqu} and Lemma \ref{lem:XLessExc} with $\sigma = 2\sigma_T/3$. To do so, we notice that as $\Cr{30} \geq 12$, we have
			$$\Aa^2 \leq \frac{\Cr{30}}{12} \big( \EE_C(T,1)+\kappa_T + \Aa \big)^{1/(2n+3)}= \frac{\sigma_T}{12}
			= \frac{1}{8} \left( \frac{2}{3} \sigma_T \right).$$
			Thus,
			\begin{align}\label{eq:heightBound}
			\sup_{\BC_{1-2 \sigma_T /3}\cap \spt(T)} X_{n+1}^2
			\leq  \frac{\Cr{21}\Cr{22}}{\left( \frac{2}{3} \right)^{2n+1} \sigma_T^{2n+1}}\big(\EE_C(T,1) + \kappa_T + \Aa \big)
			\leq  \left( \frac{3}{2} \right)^{2n+1} \frac{\Cr{21}\Cr{22}}{ \Cr{30}^{2n+3}}  \sigma_T^2. 
			\end{align}
			Let $y \in \VV_T$. We differ between two cases. Either $y$ is near the boundary having distance to $\partial \VV$ which is comparable with $\sigma_T$, or $y$ lies more in the inner of $\VV$, then $\sigma_T$ is much smaller than the distance, but on the other hand, we can choose larger balls. More formally:\\
			\textit{Case 1: } $ \sigma_T < \dist(y, \partial \VV) < 2\sigma_T$.\\
			We define $\delta:= \dist(y, \partial \VV_{2\sigma_T /3})$. Notice that
			\begin{align*}
			\BB^n_\delta(y) \subset \VV_{2\sigma_T /3} \qquad \text{ and } \qquad 
			\delta = \dist(y, \partial \VV)-\frac{2}{3} \sigma_T 
			\geq \frac{1}{3}\dist(y, \partial \VV)
			\geq \frac{1}{3} \sigma_T.
			\end{align*}
			Lemma \ref{lem:minSurfEqu}, \eqref{eq:heightBound} and Lemma \ref{lem:XLessExc} (with $\sigma$ replaced by $2\sigma_T/3$)  then yield for $k \in \{1,2,3\}$
			\begin{align*}
			\big|D^k\vi(y) \big| &\leq \frac{2^k\Cr{32}}{\delta^k} \big(\sup_{\BB^n_\delta(y)} |\vi| +\lVert D^2\Phii \rVert^*_{\CC^{1}(B_\delta)} \big)\\
			&\leq 24\Cr{32} \left( \frac{1}{\sigma_T^k} \sup_{\BC_{1-2 \sigma_T /3}\cap \spt(T)} |X_{n+1}| +\frac{1}{\delta^k}\lVert D^2 \Phii \rVert^*_{\CC^{1}(B_\delta)} \right)\\
			&\leq 24\Cr{32}\left( \sqrt{ \frac{\Cr{21}\Cr{22}}{\left( \frac{2}{3} \right)^{2n+2k+1} \sigma_T^{2n+1}}\big(\EE_C(T,1) + \kappa_T + \Aa \big)} + \frac{1}{\delta^3}\lVert D^2\Phii \rVert^*_{\CC^{1}(B_\delta)}  \right)\\
			&\leq \frac{\Cr{33}}{\dist(y, \partial \VV)^{n+k+1/2}} \sqrt{\EE_C(T,1) + \kappa_T + \Aa }.
			\end{align*}
			
			\textit{Case 2: } $ \dist(y, \partial \VV) \geq 2\sigma_T$.\\
			We choose $\delta := \dist(y, \partial \VV)/2 \geq \sigma_T $. Notice that also in this case $\BB^n_\delta(y) \subset \VV_{2\sigma_T /3}$. Indeed, the following holds
			$$ \dist(y, \partial \VV_{2\sigma_T /3})
			= \dist(y, \partial \VV) - \frac{2}{3}\sigma_T 
			\geq 2\delta -\frac{2}{3}\delta > \delta.$$
			also  $\BB^n_\delta(y) \subset \overline{\VV_{\delta}}$. Therefore, Lemma \ref{lem:minSurfEqu} and Lemma \ref{lem:XLessExc} (with $\sigma$ replaced by $\delta$) imply
			\begin{align*}
			\big|D^k\vi(y) \big| &\leq \frac{2^k\Cr{32}}{\delta^k} \big(\sup_{\BB^n_\delta(y)} |\vi| + \lVert D^2\Phii \rVert^*_{\CC^{1}(B_\delta)} \big)\\
			&\leq  16\Cr{32}\left( \frac{1}{\sigma_T^k} \sup_{\BC_{1-2 \sigma_T /3}\cap \spt(T)} |X_{n+1}| + \frac{1}{\delta^k}\lVert D^2\Phii \rVert^*_{\CC^{1}(B_\delta)} \right)\\
			&\leq 16\Cr{32} \left( \sqrt{ \frac{\Cr{21}\Cr{22}}{\left( \frac{2}{3} \right)^{2n+2k+1} \sigma_T^{2n+1}}\big(\EE_C(T,1) + \kappa_T + \Aa \big)} + \frac{1}{\delta^3} \lVert D^2\Phii \rVert^*_{\CC^{1}(B_\delta)}  \right)\\
			&\leq \frac{\Cr{33}}{\dist(y, \partial \VV)^{n+k+1/2}} \sqrt{\EE_C(T,1) + \kappa_T + \Aa}.
			\end{align*}
			This shows $(iii.)$.
			
			$(iv.)$ is done as $(iii.)$.
			
			For $(v)$, we fix $i \in \{1,\dots,m\}$, $j \in \{1,\dots,m-1\}$ and abbreviate $v:= \vi$, $w:=\wj$.
			Additionally to the conditions before, we now require for $\Cr{30}$ to fulfil $\Cr{30}^{2n+3} \geq \Cr{21}\Cr{22}(2^{2/(n+2)}-1)^{-1}$.
			Then Lemma \ref{lem:XLessExc} implies that 
			\begin{align*}
			\sup_{\BC_{1-\sigma_T} \cap \spt(T)}X_{n+1}^2 
			&\leq \sigma_T^2 \frac{\Cr{21}\Cr{22}}{ \Cr{30}^{2n+3}} 
			\leq \sigma_T^2 (2^{2/(n+2)}-1).
			\end{align*}
			In the following, let $y \in \VV_T$ and $\delta := \dist(y, \VV_{2\sigma_T/3})$. Then we have
			$$\big|(y,v(y), \Phii(y, v(y)))\big|^2 = |y|^2 +v(y)^2 + |\Phii(y, v(y))|^2 \leq (1-\sigma_T)^2 + \sigma_T^2 + |D\Phii|^2 \leq 1+ |D\Phii|^2.$$ 
			Denote by $K:=\lVert D\Phii \rVert_{\CC^0(\BB_1)}$. Then $\pp^{-1}(\VV_T) \cap \spt(T) \subset \UU_{1+K}$.\hfill \refstepcounter{equation}(\theequation)\label{subsetSptT}\\
			Last, we let $\Cr{30}$ also fulfil $\Cr{30}^{2n+3} \geq 144 \left( \frac{3}{2} \right)^{2n+1}\Cr{32}^2 \Cr{21}\Cr{22}$. By Lemma \ref{lem:minSurfEqu}, the following holds
			\begin{align}
			\big|(y,v(y), \Phii(y,v(y)))\big|^{n+2} &= \big(|y|^2 +v(y)^2 + |\Phii(y,v(y))|^2 \big) ^{(n+2)/2} \notag \\
			&\leq \big(|y|^2 + (2^{2/(n+2)}-1)|y|^2 |D\Phii|^2(1+ |Dv|^2)|y|^2\big)^{(n+2)/2} \notag \\
			&\leq 2^{2+n/2} |y|^{n+2}, \label{eq:scalingDimension}
			\end{align}
			\begin{align}
			|Dv(y)|^2 &\leq  \left(\frac{2\Cr{32}}{\delta}\right)^2 \big(\sup_{\BB^n_\delta(y)} |\vi| +\lVert D^2\Phii \rVert^*_{\CC^{1}(B_\delta)} \big)^2 \notag \\
			&\leq \frac{8\Cr{32}^2}{(\frac{\sigma_T}{3} )^2} \frac{\Cr{21}\Cr{22}}{\left( \frac{2}{3} \right)^{2n+1} \sigma_T^{2n+1}}\big(\EE_C(T,1) + \kappa_T + \Aa \big) +\frac{8\Cr{32}^2}{\delta^2} \left( \lVert D^2\Phii \rVert^*_{\CC^1(B_\delta)} \right)^2 \notag \\
			&\leq 72\Cr{32}^2 \left( \left( \frac{3}{2} \right)^{2n+1}\frac{\Cr{21}\Cr{22}}{\Cr{30}^{2n+3}} +  \lVert D^2\Phii \rVert_{\CC^{\alpha}} \right)
			\leq 1. \label{eq:normDv(y)}
			\end{align}
			
			Now, we compute
			\begin{align}\label{eq:radialDeriv}
			\frac{\partial}{\partial r} \frac{v(y)}{|y|} 
			= \frac{y}{|y|} \left( \frac{Dv(y)}{|y|}-v(y) \frac{y}{|y|^3} \right) = \frac{yDv(y)-v(y)}{|y|^2}.
			\end{align}
			We notice that this is similar to the projection on span$\{(Dv(y),-1,0)\}$. Let $\zeta(y):= \frac{1}{\sqrt{1+|Dv(y)|^2}}(Dv(y),-1,0) \in \R^{n+k}$. Then
			\begin{align}\label{eq:projNormalVector}
			\langle \big(y, v(y), \Phii(y, v(y)) \big) , \zeta(y) \rangle
			= \frac{\langle y, Dv(y) \rangle -v(y)}{\sqrt{1+|Dv(y)|^2}}.
			\end{align}
			Moreover, the approximate tangent space of $\spt(T)$ at $(y, v(y), \Phii(y, v(y)) )$ is spanned by the vectors $\partial_iG(y)$ for $i \leq n$ and $G(y)=\big(y, v(y), \Phii(y,v(y)) \big)$. As $(Dv(y),-1,0)$ is normal to all of the $\partial_iG(y)$, we have that $\zeta(y)$ is normal to the approximate tangent space of $\spt(T)$ at $(y, v(y), \Phii(y, v(y)) )$. In particular,
			\begin{align}\label{eq:perp}
			\big| \langle \big(y, v(y), \Phii(y, v(y)) \big) , \zeta(y) \rangle \big| \leq \big| \big(y, v(y), \Phii(y, v(y)) \big)^\perp \big|,
			\end{align}
			where $X^\perp$ denotes the component orthogonal to the approximate tangent space of $T$.
			Therefore, we deduce by using \eqref{eq:radialDeriv}, \eqref{eq:projNormalVector}, \eqref{eq:normDv(y)}, \eqref{eq:scalingDimension}, \eqref{eq:perp} and \eqref{subsetSptT}
			\begin{align*}
			\int_{\VV_T} &\left( \frac{\partial}{\partial r} \frac{v(y)}{|y|} \right)^2 |y|^{2-n} \dd \Leb^n(y)\\
			&= \int_{\VV_T} \langle \big(y, v(y), \Phii(y, v(y)) \big) , \zeta(y) \rangle^2 \ \frac{1+|Dv(y)|^2}{|y|^{n+2}} \dd \Leb^n(y)\\
			&\leq \int_{\VV_T} \big| \big(y, v(y), \Phii(y, v(y)) \big)^\perp \big|^2 \frac{2^{2+n/2} \sqrt{2}}{\big|(y,v(y), \Phii(y,v(y)) \big|^{n+2}}  \sqrt{1+|Dv(y)|^2}\ \dd \Leb^n(y)\\
			&\leq 2^{(n+5)/2} \int_{\BB_{1+K} \cap \pp^{-1}(\VV_T)} |X^\perp|^2 |X|^{-n-2} \dT.
			\end{align*}
			We argue in the same manner to extract
			$$\int_{\WW_T} \left( \frac{\partial}{\partial r} \frac{w(y)}{|y|} \right)^2 |y|^{2-n} \dd \Leb^n(y) \leq 2^{(n+5)/2}  \int_{\BB_{1+K} \cap \pp^{-1}(\WW_T)} |X^\perp|^2 |X|^{-n-2} \dT.$$
			By \eqref{eq:massOfBall}, we have
			\begin{align*}
			\int_{\BB_{1+K} \setminus\BB_1} |X \cdot \ze^T|^2 ~ |X|^{-n-2} \dT
			&\leq 4 \mT(\BB_{1+K} \setminus\BB_1)\\
			&\leq 4 \left(\Cr{3}(1+K)^n - \frac{1}{\Cr{3}} \right)
			\leq \Cl{35}K \leq \Cr{35} \Aa.
			\end{align*}
			In total, we conclude by Corollary \ref{cor:sphericalExc} and \eqref{eq:spherCylindExc} that
			\begin{align*}
			\int_{\VV_T} \left( \frac{\partial}{\partial r} \frac{\vi(y)}{|y|} \right)^2 |y|^{2-n} \dd \Leb^n(y)
			&+  \int_{\WW_T} \left( \frac{\partial}{\partial r} \frac{\wj(y)}{|y|} \right)^2 |y|^{2-n} \dd \Leb^n(y)\\
			&\leq 2^{(n+5)/2} \big(\EE_S(T,1) + \Cr{6}(\Aa+ \kappa_T ) +\Cr{35} \Aa \big) \\
			&\leq  2^{(n+5)/2} \big(\EE_C(T,1) + (\Cr{6} +m +\Cr{35})(\Aa+ \kappa_T ) \big).
			\end{align*}
		\end{proof}

	\newpage
\section{Blow-up sequence and statement of the excess decay}\label{sec:blowup}
	We now know that, away from the boundary, our minimizing current $T$ is supported on graphs. We would like to extend that fact up to the boundary. To do so, we use that the functions describing the current are bounded by the square root of the excess such that we can introduce a blow-up procedure by rescaling by the latter quantity. Notice that the domain of the functions converges to the half ball as the excess tends to zero.
	
	We aim to extend the graphs up to the boundary of $T$ and such that they are merging together smoothly. To do so, we will show that the harmonic blow-ups on $\VV$ (or $\WW$ respectively) are all identical (see Theorem \ref{thm:blowupsC2}), which will lead to an excess decay (Theorem \ref{thm:RotExcessDecay}) which will then lead to the extension of the graphs (Corollary \ref{cor:graphsOfGammT}).
	
	First we describe the blow-up procedure.
	\begin{defn}\label{def:harmblow}
		For $\nu \in \N$, $\nu \geq 1$, $i \in \{1,\dots,m\}$, $j \in \{1,\dots,m-1\}$ and $(T_\nu, \M_\nu) \in \TT$, we define $\Anu := \Aa_{\M_\nu}$, $\varepsilon_\nu :=  \sqrt{\EE_C(T_\nu,1)}$, $\kappa_\nu:= \kappa_{T_\nu}$,
		$v_i^{(\nu)} :=v_i^{T_\nu} \1_{\VV_{T_\nu}} : \VV \to \R$  and $w_j^{(\nu)} :=w_j^{T_\nu} \1_{\WW_{T_\nu}} : \WW \to \R$. We call $\{(T_\nu, \M_\nu)\}_{\nu \geq 1}$ a blowup sequence with associated harmonic blow-ups $f_i$, $g_j$ if the following holds as $\nu \to \infty$,
		\begin{enumerate}[(i.)]
			\item $\epnu \to 0$,
			\item $\epnu^{-2} \kappa_\nu \to 0$,
			\item $\Anu \to 0$,
			\item $ \displaystyle \frac{ \vin }{\max\{\epnu, \Anu^{1/4}\}} \longrightarrow f_i$ uniformly on compact subsets of $\VV$,
			\item $\displaystyle \frac{ \wjn}{\max\{\epnu, \Anu^{1/4} \}} \longrightarrow g_i$ uniformly on compact subsets of $\WW$.
		\end{enumerate}
	\end{defn}
	
	Notice that by the estimates of Theorem \ref{thm:graph} and the Riemannian minimal surface equation, it follows that $f_i, g_j$ are harmonic. Furthermore, by Lemma \ref{lem:XLessExc}, we have for $0< \rho <1$
	\begin{align}\label{eq:fPlusg}
	\sup_{\VV \cap \BB^n_\rho (0)} |f_i|^2 + \sup_{\WW \cap \BB^n_\rho (0)} |g_j|^2
	\leq \limsup_{\nu \to \infty} \left( \frac{2}{\max\{\epnu, \Anu^{1/4}\}^2} \sup_{\BC_\rho \cap \spt(T_\nu)} X_{n+1}^2 \right)
	\leq \frac{4\Cr{21}\Cr{22}}{(1-\rho)^{2n+1}}.
	\end{align}
	Notice that by the Arzel\`{a}-Ascoli Theorem and Theorem \ref{thm:graph}, every sequence $\{(S_\nu, \M_\nu) \}_{\nu \geq 1} \subset \TT$ satisfying
	\begin{align}\label{eq:existenceHarmBlowup}
	\lim_{\nu \to \infty} \left( \EE_C(S_\nu,1) + \frac{\kappa_{S_\nu}}{\EE_C(S_\nu,1)} + \Anu \right) =0
	\end{align}
	contains a blow-up subsequence.
	
	The main result of this section is the following excess decay: We define $\Cl{54}$, $\Cl{55}$ and $\theta$ later (in Remark \ref{rmk:DfBounds}, Remark \ref{thm:excessThetaMax} and Theorem \ref{thm:excessThetaMax}) and claim
	\begin{thm}\label{thm:RotExcessDecay}
		Let $(T, \M) \in \TT$ and assume that $ \displaystyle \max\{ \EE_C(T,1), \mathbf{A}, \Cr{55}\kappa_T \} \leq \frac{\theta}{\Cr{55}}$. Then there is a real number $\displaystyle  |\eta| \leq 2\Cr{54} \sqrt{\frac{\theta}{\Cr{55}}}$ such that for all $0<r< \theta/4$ the following holds
		$$ \EE_C(\gamm_{\eta \#}T, r)\leq  \frac{\theta^{-n-1}}{\Cr{55}} ~ r.$$
	\end{thm}

	A direct consequence of the Theorem \ref{thm:RotExcessDecay} is the following
	\begin{cor}\label{cor:graphsOfGammT}
		Let $T$, $\M$, $\eta$, $\Cr{30}$ and $\theta$ be as in Theorem \ref{thm:RotExcessDecay} and Theorem \ref{thm:graph}. If we define the real numbers $\beta:= \frac{1}{4n+10}$ and $\delta := \theta^{2(1+n)} (4\Cr{30})^{-(4n+6)}$ and the sets
		$$\tilde{V} := \big\{ y \in \UU^n_\delta(0): y_n > |y|^{1+\beta} \big\} 
		\qquad \text{ and } \qquad
		\tilde{W} := \big\{ y \in \UU^n_\delta(0): y_n < -|y|^{1+\beta} \big\},$$
		then there are functions $\tilde{v}_i \in \CC^{1,\frac{1}{4}}(\overline{\tilde{V}})$, $\tilde{w}_j \in \CC^{1,\frac{1}{4}}(\overline{\tilde{W}})$ such that
		\begin{enumerate}[(i.)]
			\item $ \displaystyle \pp^{-1}(\tilde{V}) \cap \spt(\gamm_{\eta \#} T) = \bigcup_{i=1}^m \graph (\tilde{v}_i, \gamm_{\eta} \circ \Phii)$ \quad and\\ $ \displaystyle \pp^{-1}(\tilde{W}) \cap \spt(\gamm_{\eta \#} T) = \bigcup_{j=1}^{m-1} \graph (\tilde{w}_j, \gamm_{\eta} \circ \Phii).$
			\item $\left. \tilde{v}_i \right|_{\tilde{V}}$, $\left. \tilde{w}_j \right|_{\tilde{W}}$ satisfy the Riemannian minimal surface equation.
			\item $D\tilde{v}_i(0)=0= D\tilde{w}_j(0)$.
			\item $\tilde{v}_1 \leq \tilde{v}_2 \leq \cdots \leq \tilde{v}_m$ \quad and \quad $\tilde{w}_1 \leq \tilde{w}_2 \leq \cdots \leq \tilde{w}_{m-1}$.
		\end{enumerate}
		
	\end{cor}
	
		In order to handle the rotations and scalings of $T$, we state the following computations that we will prove in chapter \ref{chapter:proofs}.
	\begin{rmk}\label{rmk:scale3}
		For $\Cl{28}:= \Cr{2}+ 6^n(1+m\omm_n)$, $(T, \M) \in \TT$ and $r \geq 3$ the following holds: if $\EE_C(T,1) + \kappa_T +\Aa \leq \frac{1}{\Cr{2}}$, then
		$$ \big((\muu_{r\#}T) \LL \UU_3, \muu_r(\M) \big) \in \TT,
		\qquad \Aa_{\muu_r(\M)} \leq \frac{\Aa_\M}{r} 
		\qquad \textnormal{ and } \qquad
		\kappa_{(\muu_{r\#}T) \LL \UU_3} \leq \frac{\kappa_T}{r}.$$
	\end{rmk}
	
	Indeed, we apply Corollary \ref{cor:monotonicity} with $r$, $s$ replaced by $3/r$, $1$:
	\begin{align*}
	\left( \frac{r}{3} \right)^n \mT (\BB_{3/r}) + \int_{\BB_1 \setminus \BB_{3/r}} |X^\perp|^2 |X|^{-n-2} \dT
	\leq \mT (\BB_1) + \Cr{2} \left( \Aa + \kappa_T \right) \left( 1- \frac{3}{r} \right).
	\end{align*}
	Therefore, we have
	\begin{align*}
	\MM \left( \left( \muu_{r\#}T \right) \LL \BB_3 \right)
	&\leq r^n \MM(T \LL \BB_{3/r})\\
	&\leq 3^n \left( \mT (\BB_1) + \Cr{2} \left( \Aa + \kappa_T \right) \right) \\
	&\leq 3^n \big( \EE_C(T,1) + m \omm_n + \Cr{2} \left( \Aa+ \kappa_T \right) \big)\\
	&\leq 3^n(1+m \omm_n).
	\end{align*}
	For the bound of $\kappa_{(\muu_{r\#}T) \LL \UU_3}$, we refer to the original paper \cite{hardtSimon}.
	
	\begin{rmk}\label{rmk:mugammT}
		Let $(T, \M) \in \TT$ and $|\omega| \leq 1/8$ and assume that 
		$$ \Aa \leq \max\left\{ \frac18, (7\Cr{23} + \Cr{10}+1)^{-1} \right\} $$
		Then, we have
		\begin{enumerate}[(i.)]
			\item if $\EE_C(T,1) + \kappa_T + \Aa \leq \big(\Cr{21}\Cr{22}4^{2n+4}
			\big)^{-1}$, then $~\sup\limits_{\BC_{3/4} \cap \spt(T)} |X_{n+1}| \leq \frac{1}{8}$.
			\item if $\EE_C(T,1) + \kappa_T + \Aa \leq \min\big\{\Cr{28}^{-1},
			\big(\Cr{21}\Cr{22}4^{2n+4} \big)^{-1}\big\}$, then
			$$(\muu_{4\#} \gamm_{\omega\#} T) \LL \UU_3 =  (\gamm_{\omega\#} \muu_{4\#} T) \LL \UU_3 \in \TT.$$
			\item if $12 \leq r < \infty$ and $\omega^2 + \EE_C(T,1) +
			\kappa_T +|D\Phii|^2+ |D^2\Phii| \leq \Cl{29}^{-1}$, where $\Cr{29} = 4^{2n+4} \Cr{28} \Cr{18}(1+\Cr{19})(1+\Cr{21})\Cr{22}$, then $(\muu_{r\#} \gamm_{\omega\#} T) \LL \UU_3 \in \TT$ and 
			$$ \kappa_{(\muu_{r\#} \gamm_{\omega\#} T) \LL \UU_3 } \leq \frac{\kappa_T}{r}.$$
		\end{enumerate}
	\end{rmk}	
	
	\begin{proof}[Proof of Corollary \ref{cor:graphsOfGammT}]
		We only show it for $\tilde{v}_i$, the argument for $\tilde{w}_j$ goes analogously.
		
		Let $0< \rho <\delta$ and define $S_\rho := (\muu_{1/\rho\#} \gamm_{\eta \#}T) \LL \UU_3$, $\M_\rho:= \muu_{1/\rho}(\M)$. As in Theorem \ref{thm:RotExcessDecay}, $(S_\rho, \M_\rho) \in \TT$. Moreover, notice that by Theorem \ref{thm:graph}, Theorem \ref{thm:RotExcessDecay} and Remark \ref{rmk:mugammT} the following holds
		\begin{align*}
		\sigma_{S_\rho} 
		&=\Cr{30}\big(\EE_C(S_\rho,1) + \kappa_{S_\rho} + \Aa_{\M_\rho} \big)^{1/(2n+3)}\\
		&= \Cr{30}\big(\EE_C(\gamm_{\eta \#}T,\rho) + \rho( \kappa_T + \Aa )\big)^{1/(2n+3)}\\
		&\leq  \Cr{30}\Big( \theta^{-n-1} \frac{\rho}{\Cr{55}} + \rho \frac{2\theta}{\Cr{55}} \Big)^{1/(2n+3)}\\
		&= \Cr{30} \rho^{1/(4n+6)} \left( \rho^{1/2} \frac{3\theta^{-n-1}}{\Cr{55}} \right)^{1/(2n+3)}\\
		&\leq \Cr{30} \rho^\beta \left( \delta^{1/2} \frac{4}{\Cr{55}} \theta^{-n-1} \right)^{1/(2n+3)}\\
		&= \Cr{30} \rho^\beta \left( \frac{4^{2n+5}}{\Cr{55} \Cr{30}^{2n+3}} \right)^{1/(2n+3)}\\
		&\leq \frac{\rho^\beta}{4}.
		\end{align*}
		Now, we estimate using Theorem \ref{thm:graph}$(iii.)$ (with $T$, $\M$, $k$ replaced by $S_\rho$, $\M_\rho$, $1$ and $2$) for all $i \in \{1, \dots, m\}$
		\begin{equation}\label{eq:DvS}
		\begin{split}
		\sup_{\VV_{\frac{1}{4}\rho^\beta}} \big| D v_i^{S_\rho} \big|
		&\leq \Cr{33} \sqrt{\EE_C(S_\rho,1) + \kappa_{S_\rho} + \Aa_{\M_\rho} } \sup_{y \in \VV_{\frac{1}{4}\rho^\beta}} \dist(y, \partial \VV)^{-1-n-1/2}\\
		&\leq \Cr{33} \sqrt{3\rho\ \theta^{-n-1}} \left( \frac{4}{\rho^\beta} \right)^{n+3/2}\\
		&\leq \Cl{56} \rho^{1/4},\\
		\sup_{\VV_{\frac{1}{4}\rho^\beta}} \big| D^2 v_i^{S_\rho} \big|
		&\leq \Cr{33} \sqrt{\EE_C(S_\rho,1) + \kappa_{S_\rho} +  \Aa_{\M_\rho} } \sup_{y \in \VV_{\frac{1}{4}\rho^\beta}} \dist(y, \partial \VV)^{-2-n-1/2}\\
		&\leq \Cr{33} \sqrt{3\rho\ \theta^{-n-1}} \left( \frac{4}{\rho^\beta} \right)^{n+5/2}\\
		&\leq \Cr{56} \rho^{1/4}.
		\end{split}
		\end{equation}
		Now, we look for functions whose graph contain $\spt(\gamm_{\eta \#}T)$. For a fixed $\rho$, we apply Theorem \ref{thm:graph} to $(S_\rho,  \M_\rho)$ and get $v_1^{S_\rho} \leq v_2^{S_\rho} \leq \cdots \leq v_m^{S_\rho}$ whose $\Phii_{\M_\rho}$-graph from the $\spt({S_\rho})$. 
		Define $\rho_k := \left( \frac{7}{8} \right)^k$, $k \in \Z$, and look at the annuli 
		$$A_k := \left\{ y \in \tilde{V}: \frac{1}{2} p_k \leq |y| \leq \frac{2}{3} \rho_k \right\}.$$
		These annuli are overlapping as $\frac{1}{2} \rho_k < \frac{2}{3} \rho_{k+1}$ and moreover they cover all of $\tilde{V}$. Notice that for $y \in A_k$ the following holds
		$$ \frac{y_n}{\rho_k} > \frac{|y|^{1+\beta}}{\rho_k} \geq \left( \frac{\rho_k}{2} \right)^{1+\beta} \frac{1}{\rho_k} \geq \frac{\rho_k^{\beta}}{4} \geq \sigma_{S_{\rho_k}}.$$
		Therefore, $y_k/\rho_k \in \VV_{S_{\rho_k}}$ and we can define for $y \in A_k$ 
		$$ \tilde{v}_i(y) = \rho_k v_i^{S_{\rho_k}} \Big( \frac{y}{\rho_k} \Big).$$
		Then
		$$ \pp^{-1}(\tilde{V}) \cap \spt(\gamm_{\eta \#}T) = \bigcup_{i=1}^m \graph (\tilde{v}_i, \gamm_{\eta} \circ \Phii),$$
		because $S_\rho := (\muu_{1/\rho\#} \gamm_{\eta \#}T) \LL \UU_3$. Moreover, all $\tilde{v}_i$ fulfil the Riemannian minimal surface equation on $\tilde{V}$ and
		$\tilde{v}_1 \leq \tilde{v}_2 \leq \cdots \leq \tilde{v}_m$. The only thing we still have to prove is the $\CC^{1, \frac{1}{4}}$-regularity. By using the bounds in \eqref{eq:DvS}, we estimate for each $y \in \tilde{V}$
		\begin{align}
		|D \tilde{v}_i(y) | &\leq \Cr{56}\rho_k^{1/4}
		\leq 2 \Cr{56}  |y|^{1/4}, \label{eq:Dvtilde}\\
		|D^2 \tilde{v}_i(y) | &\leq \frac{1}{\rho_k} \Cr{56} \rho_k^{1/4}
		\leq  \Cr{56} |y|^{{-3/4}}. \label{eq:D2vtilde}
		\end{align}
		Let $y$, $z \in \tilde{V}$ be arbitrary. We want to deduce that $|D\tilde{v}_i(y) -D\tilde{v}_i(z)| \leq 4 \Cr{56}|y-z|$. We differ between the following cases:\\
		\textit{Case 1: } $\max \big\{ |y|, |z| \big\} \leq 2|y-z|$.\\
		Then the following holds by \eqref{eq:Dvtilde}
		\begin{align*}
		|D\tilde{v}_i(y) -D\tilde{v}_i(z)|
		&\leq |D\tilde{v}_i(y)| + |D\tilde{v}_i(z)|\\
		&\leq 2 \Cr{56} |y|^{1/4} + 2 \Cr{56} \theta^{-n/2} |z|^{1/4}\\
		&\leq 4 \Cr{56} |y-z|^{1/4}.
		\end{align*}
		\textit{Case 2: } $\max \big\{ |y|, |z| \big\} > 2|y-z|$.\\
		Wlog $\max \big\{ |y|, |z| \big\} = |y|$. We claim that also the path between these two points fulfils this inequality. Indeed, for every $t \in [0,1]$ we have
		$$|y +t(y-z)| \geq \big| |y| -t |z-y| \big|
		\geq 2|y-z| -t |y-z| \geq |y-z|$$
		and $$|y +t(y-z)|^{-3/4} \leq |y-z|^{-3/4}.$$
		We use this together with \eqref{eq:D2vtilde} to infer
		\begin{align*}
		|D\tilde{v}_i(y) -D\tilde{v}_i(z)| 
		\leq |y-z| \int_0^1 \big|D^2 \tilde{v}_i(y +t(y-z)) \big| \dd t
		\leq \Cr{56}|y-z|.
		\end{align*}
		Hence, we can extend each $\tilde{v}_i$ on $\overline{\tilde{V}}$ with the required regularity. Moreover, the following holds for all $y \in \tilde{V}$
		$$|D\tilde{v}_i(0)| \leq |D\tilde{v}_i(y)| + |D\tilde{v}_i(0)-D\tilde{v}_i(y)|
		\leq 4 \Cr{56}|y|^{1/4}.$$
		Letting $|y| \downarrow 0$ yields $(iii.)$.
		
	\end{proof}

	\section{Glueing of harmonic blow-ups and first collapsing lemma}
		We aim to prove that under certain conditions, the harmonic blow-ups agree in order to deduce later that the graphs are equal on $\VV$ and $\WW$ respectively. The first step in this direction is to show that if we glue them together, the result is weakly differentiable.
		\begin{lem}\label{lem:harmblowupgradient}
			Let $\{(T_\nu, \M_\nu)\}_{\nu \geq 1} \subset \TT$ be a blow-up sequence with associated harmonic blow-ups $f_i$, $g_j$. Define $h, \mu: \UU^n_1(0) \to \R$ by
			$$h(x)=\begin{cases}
			\sum_{i=1}^m f_i(x), &\text{ if } x \in \VV\\
			\sum_{j=1}^{m-1} g_j(x), &\text{ if } x \in \WW\\
			0, &\text{ if } x \in \BL
			\end{cases}$$
			and $$\mu(x)=\begin{cases}
			\min\big\{|f_1(x)|, \dots, |f_m(x)|\big\}, &\text{ if } x \in \VV\\
			0, &\text{ if } x \in \WW \cup \BL.
			\end{cases}$$
			Then $h$ and $\mu$ are in $W^{1,2}_{\mathrm{loc}}(\UU_1^n(0))$.
		\end{lem}
		
		\begin{rmk}\label{rmk:traceVanish}
			Consider the notion of trace as in \cite[Chapter 26]{treves}. The previous lemma implies that $\left. \mu \right|_\VV$ has zero trace on $\BL$.
		\end{rmk}
		
		\begin{proof}[Proof of Leamm \ref{lem:harmblowupgradient}]
			Let $\{(T_\nu, \M_\nu)\}_{\nu \geq 1} \subset \TT$ be a blow-up sequence with associated harmonic blow-ups $f_i$, $g_j$ and denote $\Anu$, $\epnu$, $\kappa_\nu$ as in the Definition \ref{def:harmblow} and $\Phii_\nu := \Phii_{\M_\nu}$. Let $\Cl{36}>0$ be such that $\sqrt{1+t} \geq 1+ t -\Cr{36}t^2$ for all $0 \leq t \leq 1$. We use Theorem \ref{thm:graph}$(iii.)$ to estimate for any $i \in \{1, \dots, m\}$
			\begin{align*}
			\epnu^2&= \MM(T_\nu \LL \BC_1) - \MM\big( \pp_\#(T_\nu \LL \BC_1) \big)\\
			&\geq \MM \big(T_\nu \LL \pp^{-1}(\VV_{\sqrt{\sigma_\nu}}) \big) - \MM\big( \pp_\#(T_\nu \LL  \pp^{-1}(\VV_{\sqrt{\sigma_\nu}})) \big)\\
			&\geq \int_{\VV_{\sqrt{\sigma_\nu}}} \left( \sqrt{1+ |D\vin|^2} -1 \right) \dd \Leb^n\\
			&\geq \int_{\VV_{\sqrt{\sigma_\nu}}} \left( |D\vin|^2 - \Cr{36} |D\vin|^4 \right) \dd \Leb^n\\
			&\geq \int_{\VV_{\sqrt{\sigma_\nu}}} |D\vin|^2 \left( 1- \Cr{33}^2 \Cr{36}(\epnu^2 + \kappa_\nu + \Anu) {\sigma_\nu}^{-n-3/2} \right) \dd \Leb^n\\
			&=\int_{\VV_{\sqrt{\sigma_\nu}}} |D\vin|^2 \left( 1- \Cr{30}^{-n-3/2}\Cr{33}^2 \Cr{36} \sqrt{\epnu^2 + \kappa_\nu + \Anu}  \right) \dd \Leb^n.
			\end{align*}
			Hence, for $\nu$ large enough, it follows that
			$$\int_{\VV_{\sqrt{\sigma_\nu}}} |D\vin|^2 \dd \Leb^n \leq 2\epnu^2.$$
			Moreover, fix $\delta >0$. For all $\nu$ such that $\sigma_\nu \leq \delta^2$ the following holds 
			$$\int_{\VV_{\delta}} \frac{|D\vin|^2}{\max\{\epnu^2, \Anu^{1/2}\}} \dd \Leb^n \leq
			\int_{\VV_{\delta}} \frac{|D\vin|^2}{\epnu^2} \dd \Leb^n \leq 2$$
			and by locally uniform convergence, we deduce
			$$\int_{\VV_{\delta}} |Df_i|^2 \dd \Leb^n \leq 2.$$
			As $\delta$ was arbitrary, we can conclude the integrability of the weak derivative of $f_i$ in all of $\VV$ and analogously for $g_j$ in $\WW$, thus also of $h$ and $\mu$.
			
		\end{proof}
	
		As a next step, we see that also around boundary points, we have local uniform convergence. In fact, the proof of the original paper \cite{hardtSimon} carries over and thus, we omit the details here.
		
		\begin{lem}\label{lem:harmblowbound}
			Let $0< \sigma <1/2$, $a \in \BL \cap \UU^n_{1-2\sigma}(0)$, $U:= \UU^n_\sigma (a)$, $B:= \partial U$, $C \subset \pp^{-1}(U)$ compact and $\{(T_\nu, \M_\nu)\}_{\nu \geq 1} \subset \TT$ a blowup sequence with associated harmonic blowups $f_i$ and $g_j$. Denote $\epnu := \sqrt{\EE_C(T_\nu,1)}$ and $\mm_\nu := \max\{\epnu, \Anu^{1/4} \}$. Then, the following holds
			\begin{align*}
			\limsup_{\nu \to \infty} \sup_{C\cap \spt(T_\nu)} \frac{ X_{n+1}}{\mm_\nu}
			&\leq \max \Big\{ \sup_{B\cap \VV} f_m, \sup_{B\cap \WW} g_{m-1},0 \Big\},\\
			\liminf_{\nu \to \infty} \inf_{C\cap \spt(T_\nu)} \frac{ X_{n+1}}{\mm_\nu}
			&\geq \min \left\{ \inf_{B\cap \VV} f_1, \inf_{B\cap \WW} g_1,0 \right\}.
			\end{align*}
		\end{lem}
		
		As a first step to the fact, that the harmonic blow-ups coincide, we prove it under the strong assumptions that they are linear. This will be useful, as for the excess decay we will use a blow-up argument in which the inequality of Theorem \ref{thm:graph}$(v.)$ forces them to be linear. The argument for the equality of the blow-ups relies on the fact, that in case they are not equal, we find a better competitor for the minimization problem.
		
		\begin{lem}[Collapsing lemma]\label{lem:harmblowequal}
			Let $\{(T_\nu, \M_\nu)\}_{\nu \geq 1} \subset \TT$ be a blowup sequence and\\
			\begin{minipage}[t]{\linewidth}
				\hspace{-0.2cm}
				\centering
				\begin{minipage}{0.58\linewidth}
					\vspace*{0cm} 
					 denote $\epnu := \sqrt{\EE_C(T_\nu,1)}$ and $\mm_\nu := \max\{\epnu, \Anu^{1/4} \}$. Assume the
					harmonic blowups are of the form
					$$ f_i = \beta_i \left. Y_n \right|_\VV,  \qquad
					g_j = \gamma_j \left. Y_n \right|_\WW,$$
					for some real numbers $\beta_1 \leq \cdots \leq \beta_m$, $\gamma_1 \geq \cdots \geq \gamma_{m-1}$. Then the following holds $$\beta_1 = \cdots = \beta_m = \gamma_1 = \cdots = \gamma_{m-1}$$
					and for every $0< \rho <1$
					$$ \lim_{\nu \to \infty} \sup_{\BC_\rho \cap \spt(T_\nu)} \left| \frac{X_{n+1}}{\mm_\nu} - \beta_1X_n \right|=0.$$
				\end{minipage}
				\hspace{0.05\linewidth}
				\begin{minipage}[t]{0.35\linewidth}
					\vspace*{-2.8cm} 
					\includegraphics[scale=0.4]{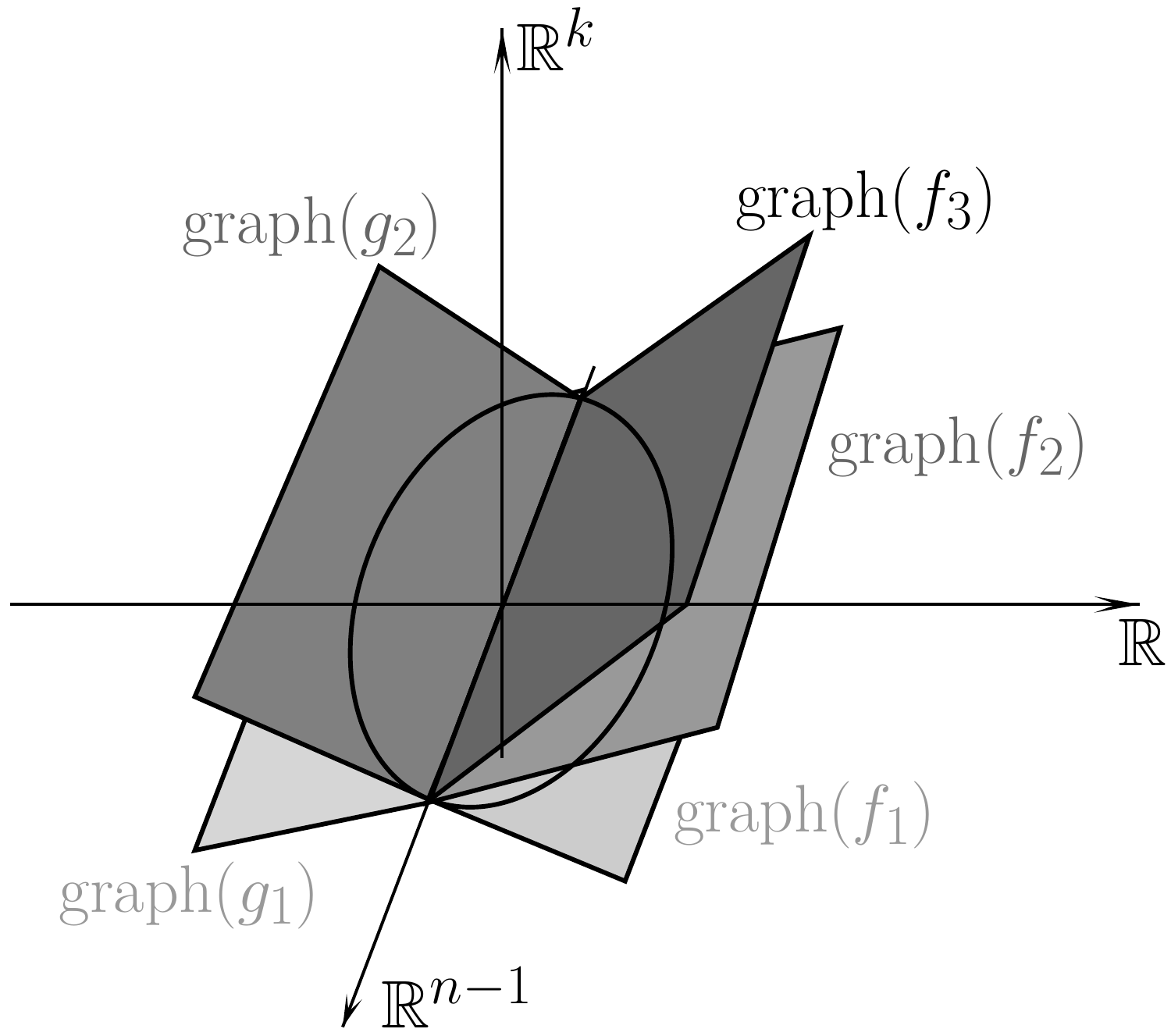}
				\end{minipage}
			\end{minipage}
			
		\end{lem}
		
		\begin{proof}
			Let $\vin$ and $\wjn$ be as in Definition \ref{def:harmblow}, define $\zeta := \max \big\{|\beta1|, |\beta_m|, |\gamma_1|, |\gamma_{m-1}| \big\}$,
			$\delta:=\min\big\{\{1\} \cup \{\beta_{i+1}-\beta_i: \beta_{i+1} \neq \beta_i\} \cup \{\gamma_i - \gamma_{i+1}: \gamma_i \neq \gamma_{i+1} \} \big\}$ and let $0<\sigma < \min\{\delta/2, 1/16\}$.
			By Theorem \ref{thm:graph}$(iii.)$, $(iv.)$, Definition \ref{def:harmblow}$(i.)$-$(v.)$ and the previous Lemma \ref{lem:harmblowbound}, we can choose $N_\sigma >0$ such that for all $\nu \geq N_\sigma$ the following holds
			
			\begin{minipage}[t]{\linewidth}
				\hspace{-0.8cm}
				\centering
				\begin{minipage}{0.12\linewidth}
					\vspace{0.5cm} 
					\includegraphics[scale=0.4]{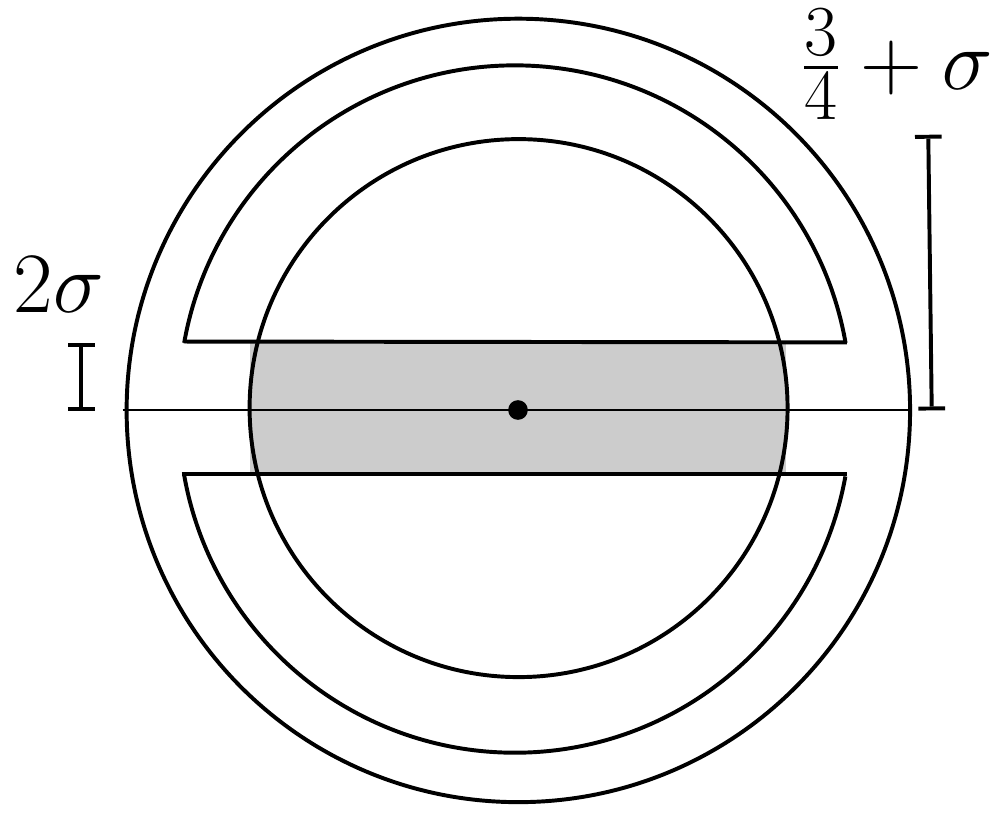}
				\end{minipage}
				\hspace{0.15\linewidth}
				\begin{minipage}[t]{0.7\linewidth}
					\vspace*{-2cm} 
					\begin{align}
					\sigma_{T_\nu} < \frac{\sigma}{4}, \quad \mm_\nu^2 < \sigma, \quad \kappa_{T_\nu} &< \sigma^3 \mm_\nu^2 \label{eq:epsigmkap}\\
					\sup_{\VV_{\sigma/2}} \big| \vin - \mm_\nu \beta_i Y_n \big|^2 \leq \sigma^{n+4}\mm_\nu^2 \quad &\text{ for all } 0 \leq i \leq m \label{eq:vminuseby}\\
					\sup_{\WW_{\sigma/2}} \big| \wjn - \mm_\nu \gamma_j Y_n \big|^2 \leq \sigma^{n+4}\mm_\nu^2 \quad &\text{ for all } 0 \leq j \leq m-1 \\
					\sup_{\BC_{3/4+\sigma} \cap \spt(T_\nu) \setminus \pp^{-1}(\VV_{2\sigma} \cap \WW_{2\sigma})} |X_{n+1}| 
					&\leq 2 \zeta \sigma \mm_\nu + \sigma\mm_\nu. \label{eq:supx}
					\end{align}
				\end{minipage}
			\end{minipage}
			\hfill \\
			The grey area in the sketch stands for the set where the supremum in \eqref{eq:supx} is taken.
			We divide the proof into several steps.\\
			\textit{Step 1:} For all $i \in \{1,\dots,m\}$, $j \in \{1, \dots, m-1\}$ the following holds
			\begin{align}\label{eq:firstStep}
			\sup_{\VV_\sigma} \big| D(\vin- \mm_\nu\beta_i Y_n) \big|^2, \
			\sup_{\WW_\sigma} \big| D(\wjn- \mm_\nu\gamma_j Y_n) \big|^2
			\leq \Cl{38} \sigma^2 \mm_\nu^2.
			\end{align}
			\textit{Step 2:} There is a Lipschitzian map $F_\nu^\sigma$ such that 
			$$\MM({F_\nu^\sigma}_\# T_\nu)- \MM(T_\nu ) \leq \Cl{41}(1+\zeta)^2\sigma \mm_\nu^2.$$
			\begin{addmargin}[1.5cm]{0pt}
				The maps  $F_\nu^\sigma$  are constructed by performing the blowup process backwards: we multiply the harmonic blowups with $\epnu$ and move it by $\sigma$ to the origin. These compressed sheets then almost recreate the original currents.
			\end{addmargin}
			\textit{Step 3:} With the help of $F_\nu^\sigma$, we show that $$\eta: \BB^n_{1/2}(0) \to \R,\ \eta(y)=
			\begin{cases}
			\beta_m Y_n(y), &\text{ if } y \in \BB^n_{1/2}(0) \cap \overline{\VV}\\
			\gamma_{m-1}Y_n(y), &\text{ if } y \in \BB^n_{1/2}(0) \cap \overline{\WW}
			\end{cases}$$
			is harmonic in $\UU^n_{1/2}(0)$. In particular, $\eta$ is differentiable in $0$ and hence, $\beta_m=\gamma_{m-1}$. We argue similarly to deduce that also $\beta_1 = \gamma_1$.\\
			\textit{Step 4:} 
			$ \displaystyle \lim_{\nu \to \infty} \sup_{\BC_\rho \cap \spt(T_\nu)} \left| \frac{X_{n+1}}{\mm_\nu} - \beta_1X_n \right|=0.$\\
			\hfill\\
			\textit{Proof of step 1:} 
			
			Away from the boundary, we want to use \cite[Corollary 6.3]{gilbarg} on the function $u:= \vin -\mm_\nu \beta_i Y_n$. Recall the coefficients $a_{ij}$ and $b$ of \eqref{eq:EulerLagrange} and define $a_{ij}^{(\nu)}$, $b^{(\nu)}$ accordingly. Then for
			$$A_{kl}:= \frac{ \delta_{k,l}}{\sqrt{1+|D\vin|^2}} -\frac{D_k\vin D_l \vin}{(1+|D\vin|^2)^{3/2}}-a_{kl}.$$
			we have
			$\displaystyle \sum_{k,l=1}^n A_{kl} \partial_{kl} u = \sum_{k,l=1}^n A_{kl} \partial_{kl}\vin =b^{(\nu)}$ and for $\nu$ large enough, $A_{kl}$ are elliptic in $\VV_{\sigma/3}$. Hence, we have
			\begin{align*}
				\sup_{\VV_\sigma} \big| D\big( \vin - \mm_\nu \beta_i Y_n\big) \big|^2 
				&\leq \frac{\C}{\sigma^2} \left( \sup_{\VV_{\sigma/3}} \big| \vin - \mm_\nu \beta_i Y_n \big|^2  + \lVert b^{(\nu)} \rVert^2_{\CC^1(\VV_{\sigma/3})} \right)
				\leq \frac{\Cr{38}}{2} \left( \sigma^2 \mm_\nu^2 + \mm_\nu^8 \right)\\
				&\leq \Cr{38} \sigma^2 \mm_\nu^2.
			\end{align*}
			In the same manner we show that
			$$\sup_{\WW_\sigma} \big| D(\wjn- \mm_\nu\gamma_j Y_n) \big|^2
			\leq \Cr{38} \sigma^2 \mm_\nu^2.$$
			
			\textit{Proof of step 2:}  Fix $i \in \{1,\dots,m\}$, $j \in \{1, \dots, m-1\}$ and define the following subsets of $\R^{n+1}$:
			\begin{align*}
			H^{\sigma} &:=  \{x \in \R^{n+1}: |x_n| \leq \sigma\},\\
			I_i^{\sigma} &:=  \left\{ x\in \R^{n+1}: (x_1, \dots, x_n) \in \VV_\sigma \text{ and } |x_{n+1} - \beta_ix_n| < \frac{\delta\sigma}{2} \right\},\\
			J_j^{\sigma} &:= \left\{ x\in \R^{n+1}: (x_1, \dots, x_n) \in \WW_\sigma \text{ and } |x_{n+1} - \gamma_jx_n| < \frac{\delta\sigma}{2} \right\}.
			\end{align*}
			\begin{minipage}[t]{\linewidth}
				\hspace{-0.4cm}
				\centering
				\begin{minipage}{0.66\linewidth}
					Notice that $ I_i^\sigma \cap I_k^\sigma = \emptyset \text{ for all } \beta_i \neq \beta_k$ and $ J_j^\sigma \cap J_l^\sigma = \emptyset \text{ for all } \gamma_j \neq \gamma_l$ by the definition of $\delta$.	Additionally,
					define the maps $\bet_r:\R^{n+1} \to \R^{n+1}, (x,y) \mapsto \left(x_1,\dots,x_n,rx_{n+1} \right)$ for $r>0$. We define
					\begin{align*}
					G_\nu^\sigma &:= H^{\sigma} \cup \bet_{\mm_\nu} \Big( \bigcup_{i=1}^m I_i^{\sigma} \Big) \cup \bet_{\mm_\nu} \Big( \bigcup_{j=1}^{m-1} J_j^{\sigma} \Big),\\
					\lambda_\nu^\sigma&: G_\nu^\sigma \to \R^{n+1}\\
					&\quad x \mapsto \begin{cases}
					(x_1,\dots,x_n,0), &\text{ if } x \in H^\sigma\\
					\big(x_1,\dots,x_n,\mm_\nu \beta_i(x_n-\sigma) \big), &\text{ if } x \in \bet_{\mm_\nu} ( I_i^\sigma )\\
					\big(x_1,\dots,x_n,\mm_\nu \gamma_j(x_n-\sigma) \big), &\text{ if } x \in \bet_{\mm_\nu} ( J_j^\sigma ),
					\end{cases}
					\end{align*}
				\end{minipage}~
				\begin{minipage}{0.33\linewidth}
					\hspace*{-1cm}
					\includegraphics[scale=0.5]{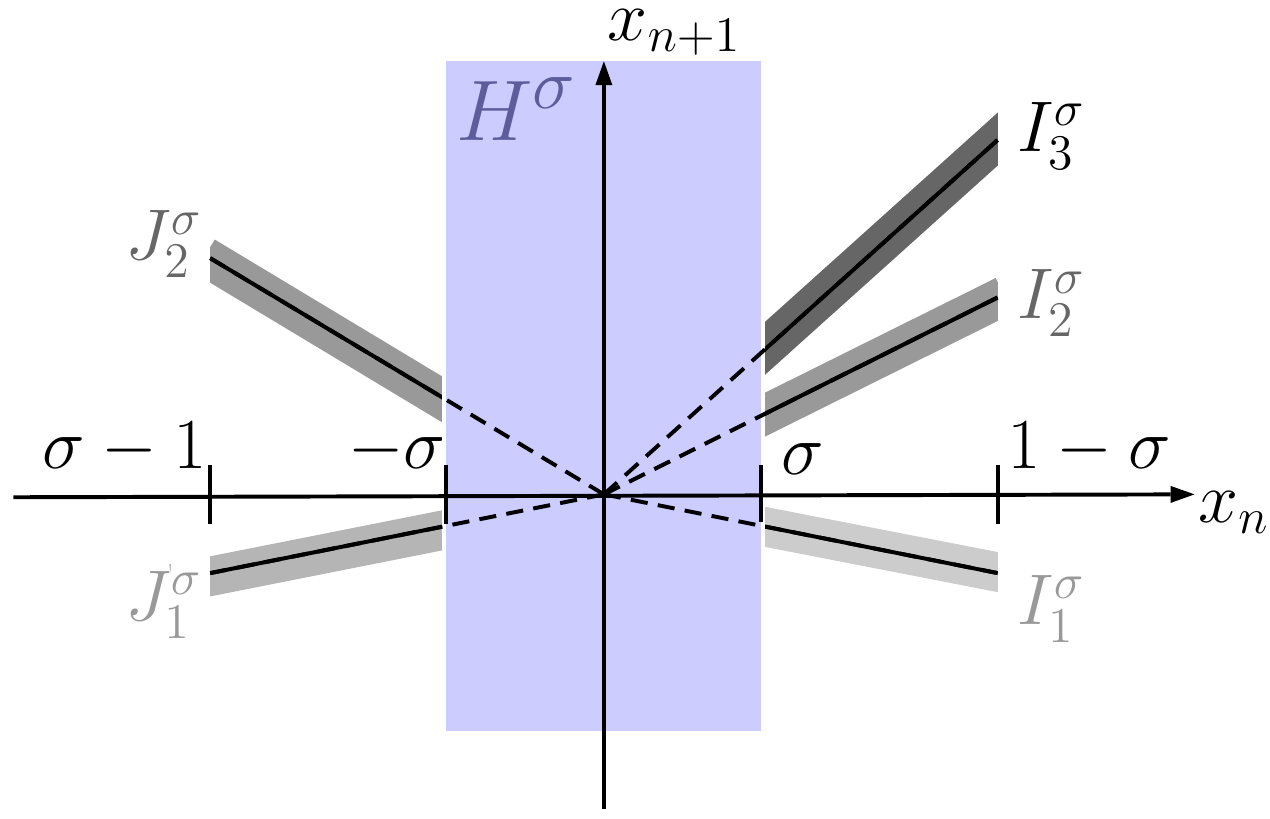}
				\end{minipage}
			\end{minipage}
			\begin{minipage}[t]{\linewidth}
				\vspace{0.4cm}
				\centering
				\begin{minipage}{0.49\linewidth}
					\includegraphics[scale=0.5]{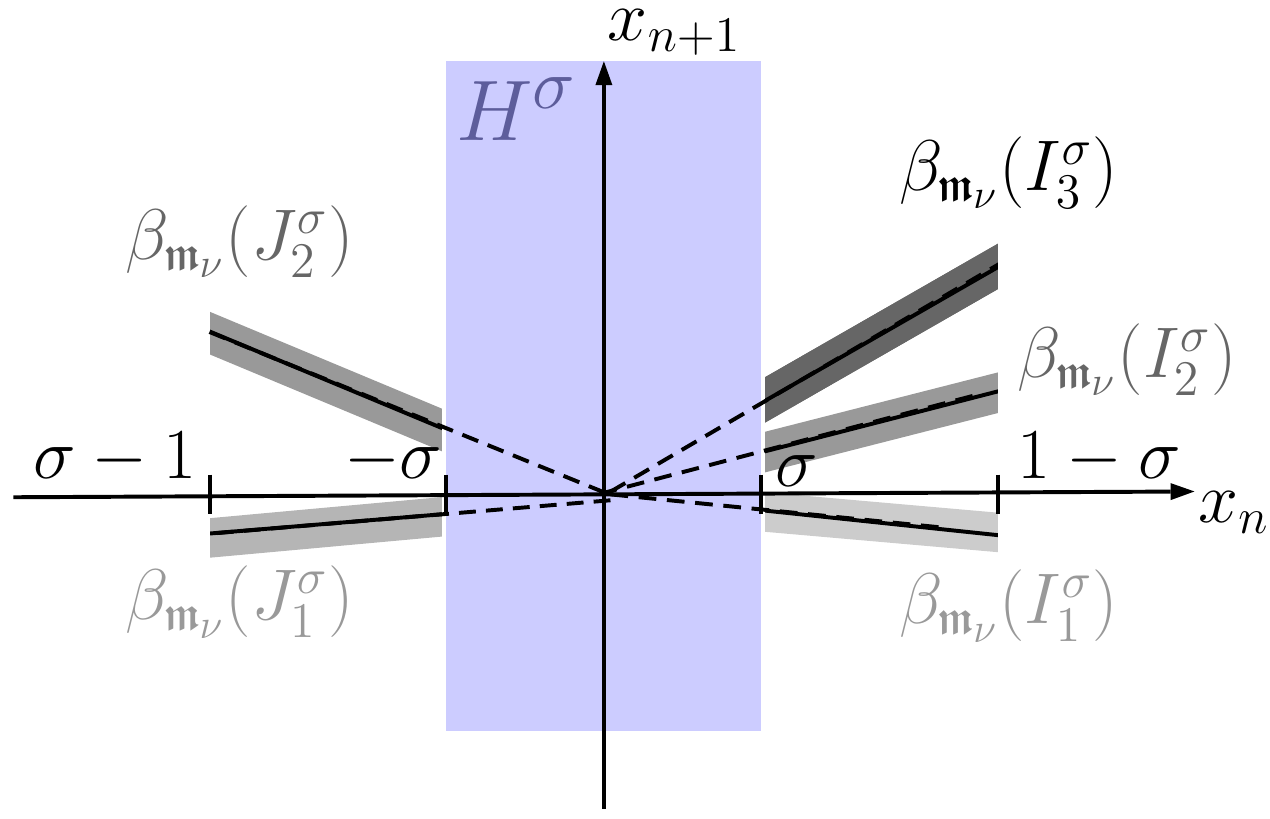}
				\end{minipage}~
				\begin{minipage}{0.49\linewidth}
					\includegraphics[scale=0.5]{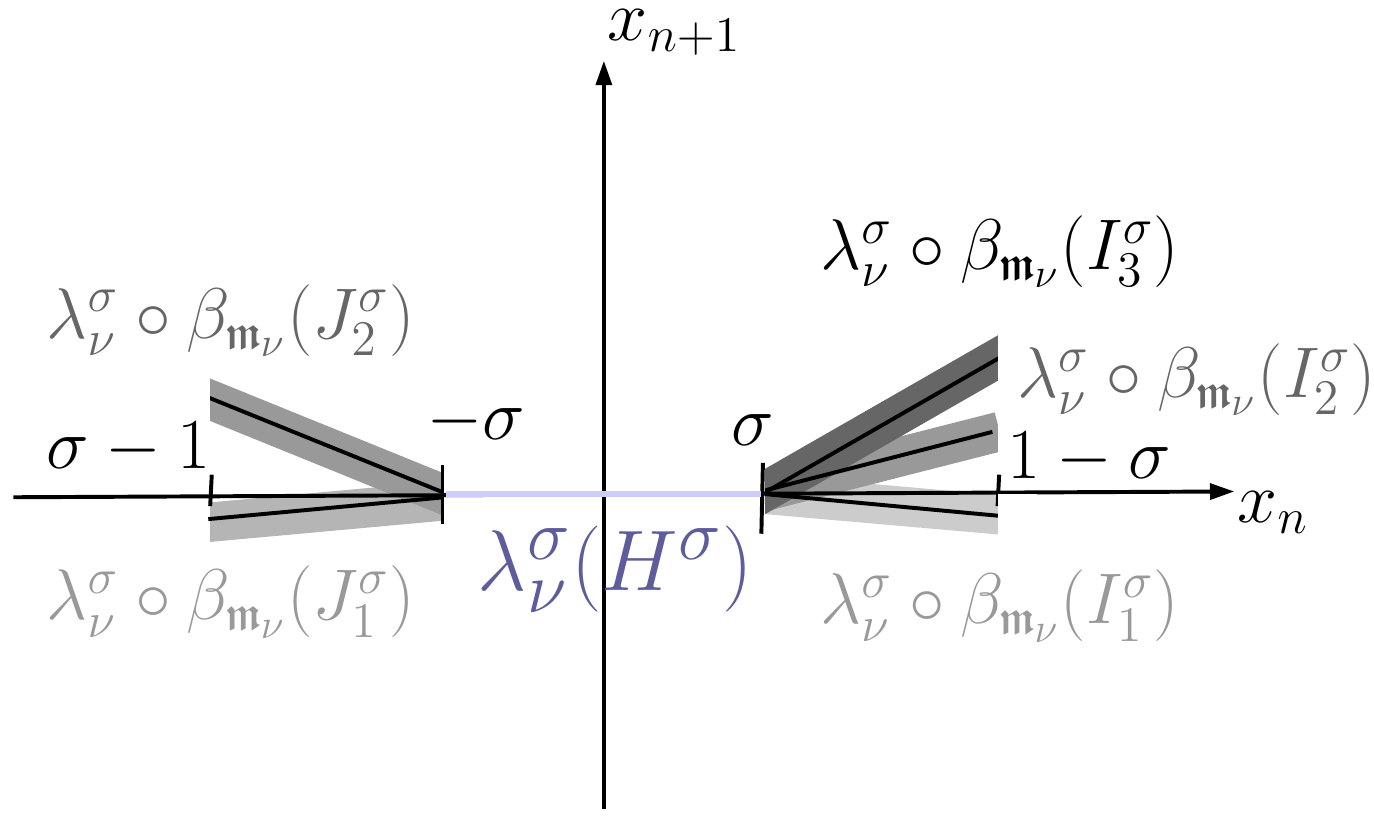}
				\end{minipage}
			\end{minipage}
			
			Now, we want to construct a homotopy between $\lambda_\nu^\sigma$ and the identity map. For this we take a $\CC^1$ function $\mu: \BB^n_1(0) \to [0,1]$ satisfying $\left. \mu \right|_{\BB^n_{1/2}(0)} \equiv 0$, $\left. \mu \right|_{\BB^n_{1}(0) \setminus \UU^n_{3/4}(0)} \equiv 1$ and $\sup\limits_{\BB^n_1(0)} |D\mu| \leq 5$. Then, we define
			\begin{align*}
			\Lambda_\nu^\sigma& := G_\nu^\sigma \cup (\R^{n+1} \setminus \BC_{3/4}) \longrightarrow \R^{n+1}\\
			&\quad x \mapsto \begin{cases}
			x, &\text{ if } x \in \R^{n+1} \setminus \BC_{3/4}\\
			\big(1-\mu\circ\pp(x)\big)\lambda_\nu^\sigma(x) + \big( \mu\circ \pp(x) \big) x, &\text{ if } x \in G_\nu^\sigma
			\end{cases}
			\end{align*}
			and finally map everything to $\M_\nu$ with
			\begin{align*}
			F_\nu^\sigma:  (G_\nu^\sigma \times \R^{k-1} ) \cup (\R^{n+k} \setminus \BC_{3/4}) &\longrightarrow \M_\nu \subset \R^{n+k}	\\
			(x,y) &\mapsto \big( \Lambda_\nu^\sigma (x), \Phii_\nu(\Lambda_\nu^\sigma (x)) \big).
			\end{align*}
			We know that in $\pp^{-1}(\VV_\sigma)$, $\spt(T_\nu)$ lives on the $\Phii_\nu$-graphs of $\vin$. As $\vin \mm_\nu^{-1}$ converges to $\beta_iY_n$, for $\nu$ big enough, graph$(\vin, \Phii_\nu) \subset (\id, \Phii_\nu) \circ \bet_{\mm_\nu}(I_i^\sigma)$. Therefore 
			$$\pp^{-1}(\VV_\sigma) \cap \spt(T_\nu) = \bigcup_{i=1}^m \graph(\vin, \Phii_\nu) \subset (\id, \Phii_\nu)(G_\nu^\sigma).$$
			Now, we compute the functions whose $\Phii_\nu$-graph describes $\spt({F_\nu^\sigma}_\#T_\nu) \cap \pp^{-1}(\VV_\sigma)$:
			\begin{align*}
			u_i^{(\nu)}&= (1-\mu)\mm_\nu \beta_i(Y_n-\sigma) + \mu\vin\\
			&= (1-\mu)\mm_\nu \beta_i Y_n + \mu \vin -(1-\mu)\mm_\nu\beta_i\sigma\\
			&= \mu(\vin-\mm_\nu \beta_iY_n)+ \mm_\nu \beta_iY_n - (1-\mu)\mm_\nu \beta_i\sigma.
			\end{align*}
			Then the following holds
			$$ u_i^{(\nu)}-\vin = \mu(\vin-\mm_\nu \beta_iY_n)-(\vin-\mm_\nu \beta_iY_n) - (1-\mu)\mm_\nu \beta_i\sigma.$$
			Recall $\zeta := \max \big\{|\beta_1|, |\beta_m|, |\gamma_1|, |\gamma_{m-1}| \big\}$. We bound by step 1 and \eqref{eq:vminuseby}
			\begin{align*}
			\sup_{\VV_\sigma} |Du_i^{(\nu)}|
			&\leq \sup_{\VV_\sigma} \left( |D\mu||v-\mm_\nu \beta_i Y_n| + |D(\vin-\mm_\nu \beta_i Y_n)| + \mm_\nu|\beta_i| + \mm_\nu\sigma |\beta_i D\mu| \right)\\
			&\leq 5 \sigma\mm_\nu + \sqrt{\Cr{38}}\sigma \mm_\nu  +\mm_\nu \zeta +5\mm_\nu\zeta\sigma\\
			&\leq \Cl{37}\mm_\nu(1+\zeta),\\
			\sup_{\VV_\sigma} |D\vin| 
			&\leq \sup_{\VV_\sigma} \left( |D(\vin-\mm_\nu \beta_i Y_n)| + |D(\mm_\nu \beta_i Y_n)| \right) \\
			&\leq \Cr{37}\mm_\nu(1+\zeta),\\
			\sup_{\VV_\sigma} |Du_i^{(\nu)}-D\vin| 
			&\leq \sup_{\VV_\sigma}\left( |D\mu||v-\mm_\nu \beta_i Y_n| + |1+\mu||D(\vin-\mm_\nu \beta_i Y_n)| + \mm_\nu\sigma |\beta_i D\mu| \right)\\
			&\leq 5 \sigma + 2\sqrt{\Cr{38}}\sigma \mm_\nu +5\mm_\nu\zeta\sigma\\
			&\leq \Cr{37}\sigma \mm_\nu(1+\zeta).
			\end{align*}
			With this we can estimate
			\begin{equation}\label{eq:massFTV}
			\begin{split}
			\MM \big( {F_\nu^\sigma}_\# &\big(T_\nu \LL \pp^{-1}(\VV_\sigma) \big) \big) - \MM\big(T_\nu \LL \pp^{-1}(\VV_\sigma) \big)\\
			&\leq \sum_{i=1}^m \sqrt{ 1+ |D\Phii_\nu|^2}\int_{\VV_\sigma} \left( \sqrt{1+|Du_i^{(\nu)}|^2} - \sqrt{1+|D\vin|^2} \right) ~ \dd \Leb^n\\
			&\leq  2\sum_{i=1}^m \int_{\VV_\sigma} \left( 1+|Du_i^{(\nu)}|^2 - 1-|D\vin|^2 \right)  ~ \dd \Leb^n \\
			&\leq  2\sum_{i=1}^m \int_{\VV_\sigma} | Du_i^{(\nu)}-D\vin| \big(|Du_i^{(\nu)}|+|D\vin| \big)~ \dd \Leb^n\\
			&\leq \Cl{40}(1+\zeta)^2\mm_\nu^2 \sigma.
			\end{split}
			\end{equation}
			In the same manner, we deduce
			\begin{align}\label{eq:massFTW}
			\MM \big( {F_\nu^\sigma}_\# \big(T_\nu \LL \pp^{-1}(\WW_\sigma) \big) \big) - \MM\big(T_\nu \LL \pp^{-1}(\WW_\sigma) \big)
			\leq \Cr{40}(1+\zeta)^2\mm_\nu^2 \sigma.
			\end{align}
			
			Outside of $\pp^{-1}(\VV_\sigma \cup \WW_\sigma)$ we notice that $F_\nu^\sigma$ is the identity in
			$\M_\nu \cap \big( (H^{\sigma} \times \R^{k-1}) \setminus \BC_{3/4} \big)$ and hence
			\begin{align}\label{eq:FisIdentity}
			\MM \big( {F_\nu^\sigma}_\# \big(T_\nu \LL ((H^{\sigma} \times \R^{k-1}) \setminus \BC_{3/4}) \big) \big) = \MM\big(T_\nu \LL ((H^{\sigma} \times \R^{k-1}) \setminus \BC_{3/4}) \big).
			\end{align}
			In $(H^{\sigma} \times \R^{k-1})\cap \BC_{3/4}$, the following holds
			$$F_\nu^\sigma (x,y)=\big(x_1, \dots, x_n, \mu(x_1,\dots,x_n)x_{n+1}, \Phii_\nu(x_1, \dots, x_n, \mu(x_1,\dots,x_n)x_{n+1})\big).$$
			Hence, we can use Lemma \ref{lem:areacomp} (with $A=(H^\sigma \times \R^{k-1})\cap \BC_{3/4}$, $\tau=\sigma$, $\rho=5\sigma$) to bound
			\begin{align*}
			\MM \big( {F_\nu^\sigma}_\# &(T \LL (H^\sigma \times \R^{k-1})) \big) - \MM (T \LL (H^\sigma \times \R^{k-1}))\\
			&\overset{\eqref{eq:FisIdentity}}{=} \MM \big( {F_\nu^\sigma}_\# \big(T \LL ((H^\sigma \times \R^{k-1})\cap \BC_{3/4} )\big) \big) - \MM \big(T \LL ((H^\sigma \times \R^{k-1})\cap \BC_{3/4}) \big)\\
			&\ \ \leq \frac{\Cr{7}}{\sigma^2} \left( \kappa_{T_\nu}^2 +2\int_{(H^{2\sigma} \times \R^{k-1}) \cap \BC_{3/4 +\sigma}} X_{n+1}^2 \dd \lVert T_\nu \rVert + 27\Anu \right)\\
			&\ \overset{\eqref{eq:supx}}{\leq} \frac{\Cr{7}}{\sigma^2} \left( \kappa_{T_\nu}^2 + 27\Anu + 2\lVert T_\nu \rVert \big( (H^{2\sigma} \times \R^{k-1})\cap \BC_{3/4 + \sigma} \big)  (2\zeta \sigma \mm_\nu + \sigma\mm_\nu)^2 \right).
			\end{align*}
				Further, we see that by the monotonicity property \eqref{eq:excmon} and the projection property of currents in $\TT$, the following holds
			\begin{align*}
			\lVert T_\nu \rVert &\big( (H^{2\sigma} \times \R^{k-1})\cap \BC_{3/4 + \sigma} \big)\\
			&= \left( \frac{3}{4}+\sigma \right)^n \EE_C \big(T_\nu, \frac{3}{4}+\sigma \big) + \MM \left( \pp_\# \big(T_\nu \LL ( (H^{2\sigma} \times \R^{k-1})\cap \BC_{3/4 + \sigma})\big) \right)\\
			&\leq \epnu^2 + m\sigma\left( \frac{3}{4} + 
			\sigma \right)^{n-1} \\
			&\leq \Cl{39} \sigma,
			\end{align*}
			where we used \eqref{eq:epsigmkap} in the last inequality.
			
			\begin{minipage}[t]{\linewidth}
				\centering
				\begin{minipage}{0.7\linewidth}
					Therefore,
					\begin{align*}
					\MM &\big( {F_\nu^\sigma}_\# (T \LL (H^{2\sigma} \times \R^{k-1})) \big) - \MM (T \LL (H^{2\sigma} \times \R^{k-1}))\\
					&\leq \frac{\Cr{7}}{\sigma^2} \left( \kappa_{T_\nu}^2 + 27\Anu + 2\Cr{39}\sigma (2\zeta \sigma \mm_\nu + \sigma\mm_\nu)^2 \right) \\
					&\overset{\eqref{eq:epsigmkap}}{\leq} \C(1+\zeta)^2\mm_\nu^2 \sigma.
					\end{align*}
					
					Putting this toghether with \eqref{eq:massFTV} and  \eqref{eq:massFTW} yields
					$$\MM({F_\nu^\sigma}_\# T_\nu)- \MM(T_\nu ) \leq \Cr{41}(1+\zeta)^2\mm_\nu^2 \sigma $$
					for all $\nu \geq N_\sigma$.
				\end{minipage}~
				\begin{minipage}[t]{0.25\linewidth}
					\vspace*{-2.5cm} 
					\hspace*{-0cm}
					\includegraphics[scale=0.35]{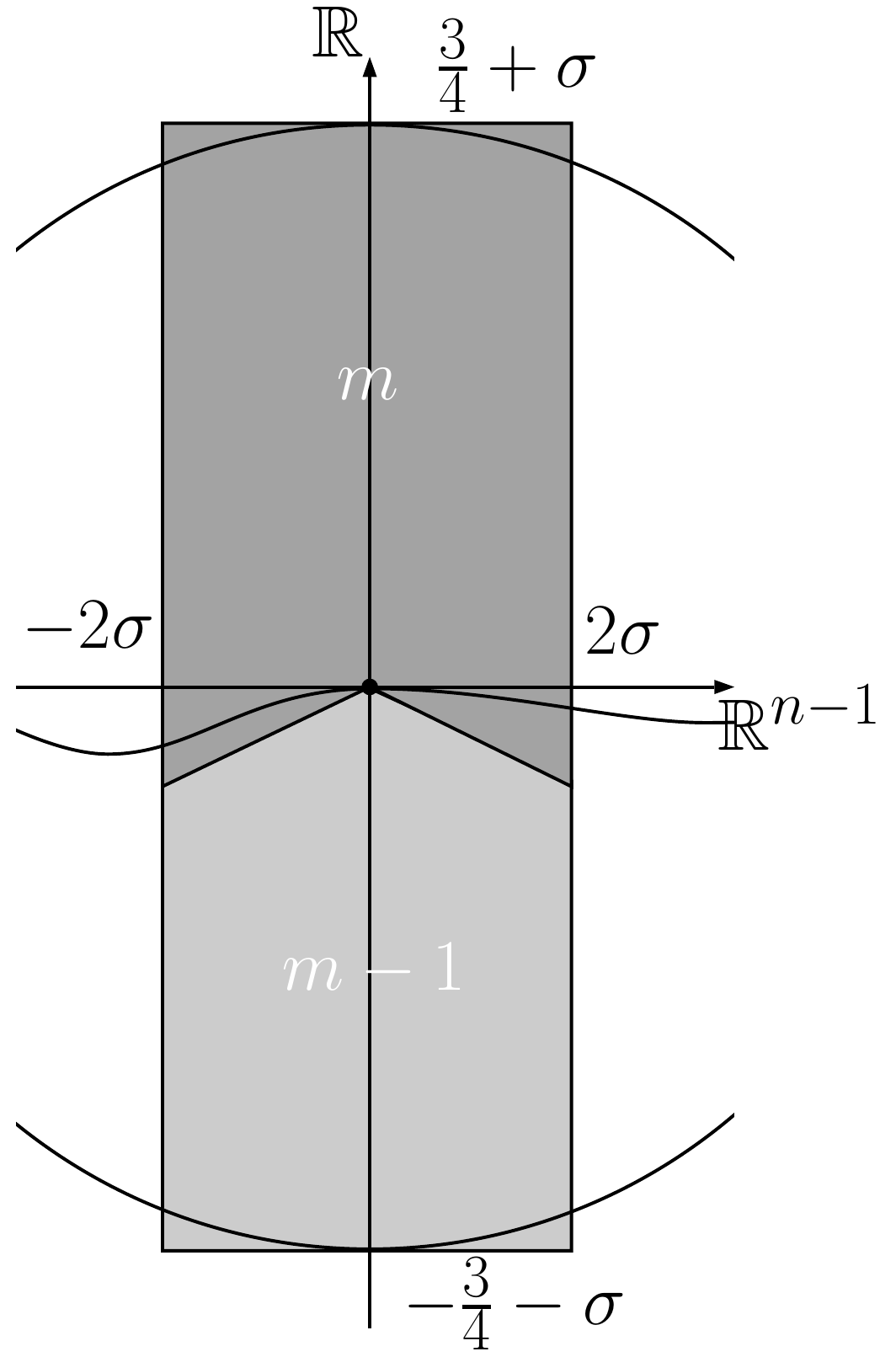}
				\end{minipage}
			\end{minipage}
			
			\textit{Proof of step 3:} We define
			$$\eta: \BB^n_{1/2}(0) \to \R, \eta(y)=
			\begin{cases}
			\beta_m Y_n(y), &\text{ if } y \in \BB^n_{1/2}(0) \cap \overline{\VV}\\
			\gamma_{m-1}Y_n(y), &\text{ if } y \in \BB^n_{1/2}(0) \cap \overline{\WW}.
			\end{cases}$$
			To show that $\eta$ is harmonic, we prove that it minimizes the Dirichlet integral. To do so, we take some arbitrary Lipschitz function $\theta: \BB^n_{1/2}(0) \to \R$ satisfying $\left. \theta \right|_{\partial \BB^n_{1/2}(0)} = \left. \eta \right|_{\partial \BB^n_{1/2}(0)}$. Then we notice that $\int |D\eta|^2 - \int |D\theta|^2$ is comparable to the difference of the Hausdorff measure of the graphs of $\eta$ and $\theta$. These graphs, we express as currents and use the minimality of $T_\nu$ to deduce that $\int |D\eta|^2 - \int |D\theta|^2 \leq 0$. To make this precise,  we approximate both of these functions. Indeed, let $\{\sigma_k\}_{k\geq 1}$ be a monotonously decreasing null sequence with $\sigma_1 < \min\{\delta/2, 1/16\}$. For each $k \geq 1$, let $\nu_k =N_{\sigma_k}$, 
			$$\eta_k: \BB^n_{1/2}(0) \to \R, \eta_k(y)=
			\begin{cases}
			\beta_m (Y_n(y)-\sigma_k), &\text{ if } y \in \BB^n_{1/2}(0) \cap \VV_{\sigma_k}\\
			\gamma_{m-1}(Y_n(y)+\sigma_k), &\text{ if } y \in \BB^n_{1/2}(0) \cap \WW_{\sigma_k}\\
			0, &\text{ if } y \in \BB^n_{1/2}(0) \setminus (\VV_{\sigma_k} \cup \WW_{\sigma_k}),
			\end{cases}$$
			and choose some $\CC^1$ function $\theta_k:\BB^n_{1/2}(0) \to \R$ with $\left. \theta_k \right|_{\partial \BB^n_{1/2}(0)} = \left. \eta_k \right|_{\partial \BB^n_{1/2}(0)}$, \newline
			$\limsup\limits_{k \to \infty} \sup\limits_{\BB^n_{1/2}(0)} |D\theta_k| < \infty$ and 
			$ \displaystyle \lim_{k\to \infty} \int_{\BB^n_{1/2}(0)} |D\theta_k -D\theta|^2 \dd \Leb^n =0.$\\
			With this, we define two auxiliary currents associated to the $\Phii_\nu$-graphs of $\mm_\nu \eta_k$ and $\mm_\nu\theta_k$ respectively:
			\begin{align*}
			R_k &:=  \left( \big(\id_n, \mm_{\nu_k} \theta_k, \Phii_{\nu_k}(\id_n, \mm_{\nu_k} \eta_k)\big) _\# (\EE^n \LL \BB^n_{1/2}) \right) \LL \overset{\circ}{\BC}_{1/2},\\
			S_k &:= \left( \big(\id_n, \mm_{\nu_k} \theta_k, \Phii_{\nu_k} (\id_n, \mm_{\nu_k} \theta_k)\big) _\# (\EE^n \LL \BB^n_{1/2}) \right) \LL \overset{\circ}{\BC}_{1/2}.
			\end{align*}
			Notice that $R_k$, $S_k$ are supported in $\M_{\nu_k}$ and moreover, in $\BC_{1/2} \cap G_\nu^\sigma$ the following holds $F_{\nu_k}^{\sigma_k} =(\id, \Phii_\nu)\circ \Lambda_{\nu_k}^{\sigma_k}=(\id, \Phii_\nu)\circ \lambda_{\nu_k}^{\sigma_k}$ and hence,
			\begin{align}\label{eq:massFTR}
			\MM \big( {F_{\nu_k}^{\sigma_k}}_\# (T_{\nu_k} \LL \BC_1) \big)
			=\MM \big( {F_{\nu_k}^{\sigma_k}}_\# (T_{\nu_k} \LL \BC_1)-R_k \big) +\MM(R_k).
			\end{align}
			In addition, we define $q(t,x)= (\id, \Phii)(x_1, \dots, x_{n-1},tx_n, tx_{n+1})$ and  $Q_{\nu_k}:= q_\#\big( [0,1] \times ((\partial T_{\nu_k})\LL\BC_2)\big)\LL \BC_1$. This is the filling between $\BB_1^{n-1}\times\{0\}$ and $\spt(\partial T_\nu) \cap \BC_1$ mapped onto $\M_{\nu_k}$.
			Then we consider $P_k := Q_{\nu_k}-({F_{\nu_k}^{\sigma_k}})_\#Q_{\nu_k}$. Because $ \left. F_{\nu_k}^{\sigma_k} \right|_{\partial \BC_1}= \left. (\id, \Phii) \right|_{\partial \BC_1}$, $\left. \theta_k \right|_{\partial \BB^n_{1/2}(0)} = \left. \eta_k \right|_{\partial \BB^n_{1/2}(0)}$ and the homotopy formula \cite[4.1.9]{federer}, the following holds
			\begin{align*}
			\partial R_k
			&= \partial S_k,\\
			\partial P_k &= \partial Q_{\nu_k}-\partial ({F_{\nu_k}^{\sigma_k}}_\#Q_{\nu_k})\\
			&= (\partial T_{\nu_k})\LL\BC_1-(\id, \Phii_{\nu_k})_\# \big((\EE^{n-1}\times \{0\})\LL \BC_1\big)\\
			& \qquad -{F_{\nu_k}^{\sigma_k}}_\#\big((\partial T_{\nu_k}) \LL \BC_1 \big)+ (\id, \Phii_{\nu_k})_\# \big((\EE^{n-1}\times \{0\})\LL \BC_1\big)\\
			&=(\partial T_{\nu_k})\LL\BC_1 - {F_{\nu_k}^{\sigma_k}}_\#\big((\partial T_{\nu_k}) \LL \BC_1\big)\\
			&= \partial (T_{\nu_k} \LL \BC_1) -
			\partial \big({F_{\nu_k}^{\sigma_k}}_\#(T_{\nu_k} \LL \BC_1) \big).
			\end{align*}
			Moreover, the area minimality of $T_{\nu_k}$ in $\M_{\nu_k}$ implies
			\begin{align*}
			\MM(T_{\nu_k} \LL \BC_1) 
			&\leq \MM \big({F_{\nu_k}^{\sigma_k}}_\#(T_{\nu_k} \LL \BC_1) +P_k -R_k+S_k \big)\\
			&\leq  \MM \big({F_{\nu_k}^{\sigma_k}}_\#(T_{\nu_k} \LL \BC_1)  -R_k \big)  +\MM(P_k) +\MM(S_k).
			\end{align*}
			Together with step 2 and \eqref{eq:massFTR}, we deduce
			\begin{align*}
			\MM(R_k)-\MM(S_k) 
			&=\MM \big( {F_{\nu_k}^{\sigma_k}}_\# (T_{\nu_k} \LL \BC_1) \big) - \MM \big( {F_{\nu_k}^{\sigma_k}}_\# (T_{\nu_k} \LL \BC_1)-R_k \big)-\MM(S_k)\\
			&\leq \MM \big( {F_{\nu_k}^{\sigma_k}}_\# (T_{\nu_k} \LL \BC_1) \big) -\MM(T_{\nu_k} \LL \BC_1)+\MM(P_k)\\
			&\leq \MM(P_k) + \Cr{41}(1+\zeta)^2\mm_{\nu_k}^2 \sigma_{\nu_k}.
			\end{align*}
			Notice that again by the homotopy formula \cite[4.1.9]{federer}, $\MM(Q_{\nu_k})\leq \C( \kappa_{T_{\nu_k}}+ \mm_{\nu_k}^4)$. Then the condition $(ii.)$ in Definition \ref{def:harmblow} yields
			$$ \limsup_{k \to \infty} \frac{\MM(P_k)}{\mm_{\nu_k}^2}
			\leq  \limsup_{k \to \infty} \big(1+\text{Lip}(F_{\nu_k}^{\sigma_k})^n \big) \frac{\MM(Q_{\nu_k})}{\mm_{\nu_k}^2} =0.$$
			Thus,
			\begin{align*}
			0 &\geq \limsup_{k \to \infty} \frac{\MM(R_k)-\MM(S_k)}{\mm_{\nu_k}^2}\\
			&= \limsup_{k \to \infty} \left( \int_{\BB^n_{1/2}(0)} \frac{\sqrt{1+\mm_{\nu_k}^2 |D\eta_k|^2}}{\mm_{\nu_k}^2} \dd \Leb^n 
			-\int_{\BB^n_{1/2}(0)} \frac{\sqrt{1+\mm_{\nu_k}^2 |D\theta_k|^2}}{\mm_{\nu_k}^2} \dd \Leb^n -\C\frac{|D\Phii_{\nu_k}|}{\mm_{\nu_k}^2}\right)\\
			&=\limsup_{k \to \infty} \int_{\BB^n_{1/2}(0)} \frac{\big(1+\mm_{\nu_k}^2 |D\eta_k|^2 \big)-(1+\mm_{\nu_k}^2 |D\theta_k|^2)}{\mm_{\nu_k}^2 \big(\sqrt{1+\mm_{\nu_k}^2 |D\eta_k|^2} +\sqrt{1+\mm_{\nu_k}^2 |D\theta_k|^2} \big)} \dd \Leb^n \\
			&= \frac{1}{2}  \int_{\BB^n_{1/2}(0)} |D\eta|^2-|D\theta|^2\dd \Leb^n.
			\end{align*}
			As $\theta$ was arbitrary, $\eta$ minimizes the Dirichlet integral and hence, is a harmonic function. In particular, $\eta$ is differentiable in $0$ and thus, $\beta_m=\gamma_{m-1}$. We argue similarly to deduce that also $\beta_1 = \gamma_1$. 
			
			\textit{Step 4:} Let $0< \rho <1$ and assume $0<\sigma < (1-\rho)/2.$ Then by Definition \ref{def:harmblow}$(iii.)$,$(iv.)$, it follows that
			$$\limsup_{\nu \to \infty} \sup_{\spt(T_\nu) \setminus H^{\sigma/2}} \left| \frac{X_{n+1}}{\mm_\nu}- \beta_1 X_n \right| =0$$
			and by Lemma \ref{lem:harmblowbound}
			\begin{align*}
			\limsup_{\nu \to \infty} \sup_{\spt(T_\nu) \cap H^{\sigma/2} \cap \BC_\rho} \left| \frac{X_{n+1}}{\mm_\nu}- \beta_1 X_n \right| 
			\leq \limsup_{\nu \to \infty} \sup_{\spt(T_\nu) \cap H^{\sigma/2} \cap \BC_\rho} \left| \frac{X_{n+1}}{\mm_\nu} \right| + |\beta_1| \frac{\sigma}{2}
			~\leq ~ |\beta_1| \sigma.
			\end{align*} 
			Letting $\sigma \downarrow 0$ concludes the proof.
			
		\end{proof}
	
	\section{Comparison between spherical and cylindrical excess}
		In some situations it is more convenient to work with the spherical excess rather than with the cylindrical one. However, in the context of blow-ups, we see that they are in fact comparable.
		\begin{lem}\label{lem:sphercylindr}
			There exist positive constants $\Cl{43}$, $\Cl{46}$, $\Cl{47}$ such that if $(T, \M) \in \TT$ satisfies
			$$ \EE_C(T,1) + \kappa_T + \Aa \leq \frac{1}{\Cr{43}} 
			\qquad \text{ and } \qquad 
			\sup_{\BC_{1/4} \cap \spt(T)} X_{n+1}^2 \leq \frac{\EE_C(T,\frac{1}{3})}{\Cr{46}},$$
			then $$ \EE_C(T, \frac{1}{3}) \leq \Cr{47} \big(\EE_S(T,1) + \kappa_T + \Aa \big).$$
		\end{lem}
		We will give the very technical proof for this in chapter \ref{chapter:proofs}. It follows by computing the first variation of a  suitable vectorfield.
		
		Instead of asking $X_{n+1}^2$ to be small, we now only assume that $T$ is optimal with respect to rotations. We will argue by contradiction, finding a suitable blow-up sequence and then we will reduce it to the case when the harmonic blow-ups are linear (in order to use Lemma \ref{lem:harmblowequal}). Here, we give a sufficient condition for this to happen.

		\begin{rmk}\label{rmk:traceEVlaplace}
				Let $h: \VV \to \R$ be a harmonic function such that for all $y \in \VV$ and $0< \rho <1$ the following holds $h(\rho y) = \rho h(y)$. Then it follows
				\begin{enumerate}[(i.)]
					\item If $h \geq 0$, then $h$ has zero trace on $\BL$.
					\item If $h$ has zero trace on $\BL$, then there is some $\beta \in \R$ satisfying $h= \beta \left. Y_n \right|_\VV$.
				\end{enumerate}
		\end{rmk}
		The proof of this fact can be read in the original paper \cite{hardtSimon}.
		
		\begin{thm}\label{thm:CylSpherExc}
			Let $(T, \M) \in \TT$ and recall $\Cr{43}$ and $\Cr{47}$ from Lemma \ref{lem:sphercylindr}. 
			Then there is a positive constant $\Cl{48}$ such that if for all real numbers $|\eta | < 1/8$ the following holds
			\begin{itemize}
				\item $\displaystyle \EE_C(T,1) + \kappa_T + \Aa \leq \frac{1}{2\Cr{43}}$,
				\item $\displaystyle \EE_C(T,\frac{1}{3}) + \frac{\kappa_T}{\EE_C(T, \frac{1}{3})} \leq \frac{1}{\Cr{48}}$,
				\item $\displaystyle \EE_C(T, \frac{1}{4}) \leq 2 \EE_C(\gamm_{\eta \#}T, \frac{1}{4}),$
			\end{itemize} 	
			then 
			$$ \EE_C(T, \frac{1}{4}) \leq \Cr{48}\big( \EE_S(T,1) + \kappa_T + \Aa \big).$$
		\end{thm}
	
		\begin{proof}
			We argue by contradiction. Assume that no matter how large $\Cr{48}$ is, there is a current satisfying the four conditions but not the fifth one. This means, there is a sequence $\{(T_\nu, \M_\nu)\}_{\nu \geq 1} \subset \TT$ such that for every $\nu \geq 1$ and $|\eta | < 1/8$ the following holds
			\begin{align}
			\EE_C(T_{\nu},1) + \kappa_{T_{\nu}} + \Anu &\leq \frac{1}{2\Cr{43}},\notag \\
			\EE_C(T_{\nu}, \frac{1}{4}) &\leq 2 \EE_C(\gamm_{\eta \#}T_{\nu}, \frac{1}{4}), \label{eq:gammexc} \\
			\lim_{\nu \to \infty} \left( \EE_C(T_{\nu},\frac{1}{3}) + \frac{\kappa_{T_{\nu}}}{\EE_C(T_{\nu}, \frac{1}{3})} \right) &=0,\label{eq:kdividedE} \\
			\lim_{\nu \to \infty} \left( \frac{ \EE_S(T_{\nu},1) + \kappa_{T_{\nu}} + \Anu}{\EE_C(T_{\nu}, \frac{1}{4})} \right) &=0. \label{eq:EsKappaDividedEc}
			\end{align}
			We define $S_\nu := (\muu_{3\#}T_\nu) \LL \UU_3 $. By Remark \ref{rmk:scale3} $(S_\nu, \muu_3(\M_\nu)) \in \TT$ and moreover
			$$\epnu := \sqrt{\EE_C(S_\nu,1)} =  \sqrt{\EE_C(T_\nu,\frac{1}{3})},
			\quad \kappa_\nu := \kappa_{S_\nu} 
			{\leq \kappa_{T_\nu}}
			\quad \text{and} \quad
			\mm_\nu := \max\left\{\epnu, \left(\frac{1}{3} \Anu \right)^{1/4}\right\} .$$
			Up to subsequence (which we do not relabel) is $\{(S_\nu, \muu_3(\M_\nu)) \}_{\nu \geq 1}$ a blowup sequence (see \eqref{eq:existenceHarmBlowup}) with harmonic blowups $f_i$ and $g_j$. We want to show that they are of the form $\beta Y_n$. Then we will be able to deduce that $\beta \neq 0$ which will make it impossible for $\EE_C(T_\nu, \frac{1}{4})\epnu^{-2}$ to converge to zero. This then leads to a contradiction to \eqref{eq:gammexc}.
			Notice that by Lemma \ref{lem:expMassMonoton}, the following holds
			$$e^{\frac{\Cr{1}}{3} ( \Anu + \kappa_{T_\nu} )}3^{n} \lVert T_\nu \rVert (\BB_{1/3})
			\leq e^{ \Cr{1} ( \Anu + \kappa_{T_\nu} ) }\lVert T_\nu \rVert (\BB_1).$$ 
			From this, it follows
			\begin{align*}
			\EE_S(S_\nu,1)
			&= \EE_S \left(T_\nu,\frac{1}{3} \right)\\
			&\leq 3^n \mTnu (\BB_{1/3})-\omm_\nu \left(m-\frac{1}{2} \right)\\
			&\leq e^{\frac{2}{3} \Cr{1} ( \Anu + \kappa_{T_\nu} ) } \lVert T_\nu \rVert (\BB_1)-\omm_\nu\left(m-\frac{1}{2} \right)\\
			&\leq e^{ \frac{2}{3} \Cr{1} ( \Anu + \kappa_{T_\nu} ) } \EE_S(T_\nu,1)+ \left(e^{ \frac{2}{3} \Cr{1} ( \Anu + \kappa_{T_\nu} )}-1 +\kappa_{T_\nu} \right)\omm_\nu \left(m-\frac{1}{2} \right)\\
			&\leq \left(e^{\Cr{1}/\Cr{43}} +2\frac{\Cr{1}}{\Cr{43}}\right) \big(\EE_S(T_{\nu},1) + \kappa_{T_\nu} \big)
			\end{align*}
			and hence,
			\begin{align}\label{eq:sphericalExcessVanishes}
			\limsup_{\nu \to \infty} \frac{ \EE_S(S_\nu,1)}{\epnu^2}
			\le \left(\frac{4}{3} \right)^n \left(e^{\Cr{1}/\Cr{43}} +2\frac{\Cr{1}}{\Cr{43}}\right) \limsup_{\nu \to \infty} \frac{ \EE_S(T_{\nu},1) + \kappa_{T_{\nu}}}{\EE_C(T_{\nu}, \frac{1}{4})}
			=0,
			\end{align}
			where we used \eqref{eq:EsKappaDividedEc}.
			
			We can apply Theorem \ref{thm:graph}$(v.)$ (with $T$ replaced by $S_\nu$) combined with Definition \ref{def:harmblow}$(iv.)$,$(v.)$  (with $T_\nu$ replaced by $S_\nu$), \eqref{eq:kdividedE} and \eqref{eq:sphericalExcessVanishes} to infer
			\begin{align*}
			\int_{\VV_T} \left( \frac{\partial}{\partial r} \frac{f_i(y)}{|y|} \right)^2 &|y|^{2-n} \dd \Leb^n(y)
			+  \int_{\WW_T} \left( \frac{\partial}{\partial r} \frac{g_j(y)}{|y|} \right)^2 |y|^{2-n} \dd \Leb^n(y)\\
			&\leq 2^{n+7}\limsup_{\nu \to \infty}\frac{ \EE_S(S_\nu,1) + \Cr{34}(\Anu + \kappa_{T_\nu})}{\mm_\nu^2}\\
			&=0.
			\end{align*} 
			Hence, both terms must vanish and therefore the following holds for all $0< \rho <1$
			$$f_i(\rho y)=\rho f_i(y) \quad \text{ for } y \in \VV \qquad \textnormal{ and }
			\qquad g_j(\rho y)=\rho g_j(y) \quad \text{ for } y \in \WW.$$
			This allows us to use Remark \ref{rmk:traceEVlaplace}$(i.)$ to the nonnegative functions $f_m-f_1$, $g_{m-1}-g_1$ having vanishing trace on $\BL$. We notice that
			\begin{align*}
			|f_i| &= \big( |f_i| - \min \{|f_1|, \cdots, |f_m|\} \big) + \min \{|f_1|, \cdots, |f_m|\}\\
			&\leq (f_m-f_1)+ \min \{|f_1|, \cdots, |f_m|\}
			\end{align*}
			and so, also each $f_i$ has zero trace on $\BL$ by Lemma \ref{lem:harmblowupgradient}. Remark \ref{rmk:traceEVlaplace}$(ii.)$ gives that $f_i = \beta_i \left. Y_n \right|_\VV$ for some $\beta_i \in \R$. The analogues statement holds for $g_j$ because Lemma \ref{lem:harmblowupgradient} implies that also $\sum_{l=1}^{m-1} g_l$ has zero trace on $\BL$ and we can bound
			\begin{align*}
			(m-1)|g_j| &= \left| \sum_{l=1}^{m-1} (g_j-g_l) + \sum_{l=1}^{m-1} g_l \right|
			\leq (m-1) (g_{m-1}-g_1) + \left|\sum_{l=1}^{m-1} g_l \right|.
			\end{align*}
			Then we can apply Lemma \ref{lem:harmblowequal} to deduce 
			\begin{align}
			\beta_1 = \cdots = \beta_m = \gamma_1 = \cdots &= \gamma_{m-1}, \notag \\
			\lim_{\nu \to \infty} \sup_{\BC_{7/8} \cap \spt(S_\nu)} \left| \frac{X_{n+1}}{\mm_\nu} - \beta_1X_n \right| &=0.
			\end{align}
			Next, we infer $\beta_1 \neq 0$. Indeed, if this were not the case, then Lemma \ref{lem:sphercylindr} would imply that
			\begin{align*}
			0&= \limsup_{\nu \to \infty} \left( \frac{ \EE_S(T_{\nu},1) + \kappa_{T_{\nu}}}{\EE_C(T_{\nu}, \frac{1}{4})} \right)\\
			&\geq \limsup_{\nu \to \infty} \left( \frac{ \frac{1}{\Cr{47}}\EE_C(T_\nu, \frac{1}{3})-\Anu }{\EE_C(T_{\nu}, \frac{1}{4})} \right)\\
			&\geq \frac{3^n}{4^n \Cr{47}} > 0,
			\end{align*}
			where we used \eqref{eq:EsKappaDividedEc} for the last inequality.
			
			Now, we rotate $T_\nu$ such that the new blowup sequence has a vanishing harmonic blowups. To do so, let $\eta_\nu := \arctan(\beta_1 \mm_\nu)$ and consider $R_\nu := (\muu_{4\#} \gamm_{\eta_\nu \#} T_\nu) \LL \UU_3$. From Remark \ref{rmk:mugammT}$(ii.)$, we know that $(R_\nu, (\muu_{4} \circ \gamm_{\eta_\nu})(\M_\nu)) \in \TT$ for $\nu$ large enough. We use again Lemma \ref{lem:excLessX} (with $T$, $\sigma$ replaced by $R_\nu$, $1/6$) and Lemma \ref{lem:harmblowbound} to obtain
			\begin{equation}\label{eq:excessGammT}
			\begin{split}
			\limsup_{\nu \to \infty} \frac{\EE_C(\gamm_{\eta_\nu \#} T_\nu, \frac{1}{4})}{\mm_\nu^2}
			&= \limsup_{\nu \to \infty} \frac{\EE_C(R_\nu, 1)}{\mm_\nu^2}\\
			&\leq \limsup_{\nu \to \infty}  36 \Cr{18} \left(\Cr{19} \sup_{\BC_{7/6} \cap \spt(R_\nu)} \frac{X_{n+1}^2}{\mm_\nu^2} + \frac{\kappa_{T_\nu}+ \Anu}{\mm_\nu^2} \right)\\ 
			&=0.
			\end{split}
			\end{equation} 
			But by Lemma \ref{lem:XLessExc} (with $T$, $\sigma$ replaced by $R_\nu$, $7/8$)
			\begin{align*}
			\liminf_{\nu \to \infty} \frac{\EE_C(T_\nu, \frac{1}{4})}{\mm_\nu^2}
			&= \liminf_{\nu \to \infty} \frac{\EE_C(\muu_{4\#}T_\nu, 1)}{\mm_\nu^2}\\
			&\geq \liminf_{\nu \to \infty} \Big( \frac{7}{8} \Big)^{2n+1} \frac{1}{\Cr{21}\Cr{22}} \left( \sup_{\BC_{1/8} \cap \spt(\muu_{4\#}T_\nu)} \frac{X_{n+1}^2}{\mm_\nu^2} - \frac{\kappa_{T_\nu}+\Anu}{\mm_\nu^2} \right)\\
			&= \frac{7^{2n+1}}{8^{2n+1}\Cr{21}\Cr{22}} \left( \frac{\beta_1}{8} \right)^2 >0.
			\end{align*}
			For $\nu$ large enough, together with \eqref{eq:excessGammT}, this contradicts \eqref{eq:gammexc}.
			
		\end{proof}
	
		\section{Coincidence of the harmonic sheets}
	
		As mentioned before, the excess decay will follow from the fact, that the harmonic blow-ups coincide on $\VV$ and $\WW$ respectively. To see this, we want to blow-up the harmonic blow-ups in a homogeneous way. Thus, we need to make sure that the limit exists, i.e. we prove that the harmonic blow-ups are $\CC^{0,1}$ up to the boundary. The proof uses suitable rotations of $T_\nu$ and the uniform convergence of the blow-up sequence at the boundary.
		\begin{lem}\label{lem:fDividedY}
			Let $\{(T_\nu, \M_\nu)\}_{\nu \geq 1} \in \TT$ be a blow-up sequence with harmonic blow-ups $f_i$ and $g_j$. Then for all $0< \rho <1$, $i \in \{1, \dots, m\}$ and $j \in \{1,\dots, m-1\}$ the following holds
			$$ \sup_{\VV \cap \BB^n_\rho(0)} \frac{|f_i(y)|}{|y|} < \infty
			\qquad \text{ and } \qquad
			\sup_{\WW \cap \BB^n_\rho(0)} \frac{|g_j(y)|}{|y|} < \infty.$$
		\end{lem}
		\begin{proof}
			For $\nu \in \N$ with $\nu \geq1$, we define $\epnu := \sqrt{\EE_C(T_\nu,1)}$ and $\kappa_\nu := \kappa_{T_\nu}$. Let $0 < \sigma \leq 1/12$ and $\omega(\nu, \sigma) \in \R$ such that for all $|\eta | \leq 1/8$
			\begin{align}\label{eq:gamexcess2excess}
			\EE_C(\gamm_{\omega(\nu, \sigma) \#} T_\nu, \frac{\sigma}{4}) \leq 2 \EE_C(\gamm_{\eta \#} T_\nu, \frac{\sigma}{4}).
			\end{align}
			Notice that by the monotonicity of the excess \eqref{eq:excmon} and Definition \ref{def:harmblow}$(i.)$, it follows $\lim_{\nu \to \infty} \EE_C(T_\nu, \sigma)=0.$ As \eqref{eq:gamexcess2excess} also must hold for $\eta=0$, it follows by Lemma \ref{lem:excLessX} that also
			\begin{align}\label{eq:omegatozero}
			\lim_{\nu \to \infty} \omega(\nu, \sigma)=0.
			\end{align}
			This implies that
			\begin{align}\label{eq:excgammT}
			\lim_{\nu \to \infty} \EE_C(\gamm_{\omega(\nu, \sigma) \#} T_\nu, \sigma)=0.
			\end{align}
			In a first step, we show that there is a constant $\Cl{50}$ such that for infinitely many $\nu$ the following holds
			\begin{align*}
			\sup_{\BC_{\sigma/5} \cap \spt(\gamm_{\omega(\nu, \sigma) \#}T_\nu)} |X_{n+1}| \leq \Cr{50}\mm_\nu \sigma.
			\end{align*}
			To do so, we first bound $\EE_C(\gamm_{\omega(\nu, \sigma) \#} T_\nu, \frac{\sigma}{4})$ by looking at two different cases:\\
			\textit{Case 1: } $\EE_C(\gamm_{\omega(\nu, \sigma) \#} T_\nu, \frac{\sigma}{3}) < \epnu^2$ for infintely many $\nu$.\\
			We use the monotonicity of the excess \eqref{eq:excmon} to deduce
			$$\EE_C(\gamm_{\omega(\nu, \sigma) \#} T_\nu, \frac{\sigma}{4}) 
			\leq \left( \frac{4}{3} \right)^n \EE_C(\gamm_{\omega(\nu, \sigma) \#} T_\nu, \frac{\sigma}{3})
			\leq \left( \frac{4}{3} \right)^n \epnu^2$$
			for infinitely many $\nu$.
			
			\textit{Case 2: } $\EE_C(\gamm_{\omega(\nu, \sigma) \#} T_\nu, \frac{\sigma}{3}) \geq \epnu^2$ for all $\nu \geq N$ for some $N$ large enough.
			
			We define $S_\nu := ( \gamm_{\omega(\nu, \sigma) \#} \muu_{\frac{1}{\sigma} \#} T_\nu ) \LL \UU_3$ and $\tilde{M}_\nu:= \gamm_{\omega(\nu, \sigma)}  \circ \muu_{\frac{1}{\sigma}} (\M_\nu) $. By Remark \ref{rmk:mugammT}$(iii.)$ is $(S_\nu, \tilde{M}_\nu) \in \TT$. Recall the constants $\Cr{43}$ and $\Cr{48}$ of Theorem \ref{thm:CylSpherExc}. By \eqref{eq:excmon}, \eqref{eq:excRotSmall}, \eqref{eq:excgammT}, \eqref{eq:gamexcess2excess} and Definition \ref{def:harmblow}, there is an integer $N_\sigma$ such that for all $\nu \geq N_\sigma$ the following holds
			\begin{itemize}
				\item $\displaystyle \kappa_\nu \leq \epnu^2,$
				\item $\displaystyle \EE_C(S_\nu,1) + \kappa_{S_\nu} + \Aa_{\tilde{M}_\nu}
				\leq \EE_C \big((\gamm_{\omega(\nu, \sigma) \#} T_\nu) \LL \UU_3, \sigma \big) +  \sigma (\kappa_\nu + \Anu)
				\leq \frac{1}{2\Cr{43}},$
				\item $\displaystyle \EE_C(S_\nu, \frac{1}{3}) + \frac{\kappa_{S_\nu}}{\EE_C(S_\nu, \frac{1}{3})}
				\leq 3^n \EE_C \big( ( \gamm_{\omega(\nu, \sigma) \#} T_\nu) \LL \UU_3, \sigma \big) + \sigma \frac{\kappa_\nu}{\epnu^2}
				\leq \frac{1}{\Cr{48}},$
				\item $\displaystyle \EE_C(S_\nu, \frac{1}{4}) \leq 2 \EE_C(\gamm_{\eta \#}S_\nu, \frac{1}{4}) \qquad \text{ for all } |\eta| \leq \frac{1}{8}.$	
			\end{itemize}
			Therefore, we can apply Theorem \ref{thm:CylSpherExc} (with $T$ replaced by $S_\nu$ for $\nu \geq N_\sigma$) to deduce
			\begin{equation}\label{eq:excessgammT}
			\begin{split}
			\EE_C(\gamm_{\omega(\nu, \sigma) \#}T_\nu, \frac{\sigma}{4})
			&= \EE_C(S_\nu, \frac{1}{4})
			\leq \Cr{48} \big(\EE_S(S_\nu, 1) + \kappa_{S_\nu} + \Aa_{\tilde{M}_\nu} \big)\\
			&\leq  \Cr{48} \big(\EE_S(T_\nu, \sigma) + \sigma (\kappa_{\nu} + \Anu) \big).
			\end{split}
			\end{equation}
			Notice that by Lemma \ref{lem:expMassMonoton}, the following holds
			$$e^{\Cr{1} \left( \Anu + \kappa_\nu \right)\sigma }\sigma^{-n} \mTnu (\BB_\sigma)
			\leq e^{\Cr{1} \left( \Anu + \kappa_\nu \right) } \mTnu (\BB_1).$$
			Therefore, 
			\begin{align*}
			\EE_S(T_\nu, \sigma) &= \sigma^{-n} \mTnu (\BB_\sigma) - (m-\frac{1}{2})\al(n)\\
			&\leq e^{\Cr{1} \left( \Anu + \kappa_\nu \right) } \big( \mTnu (\BB_1) - (m-\frac{1}{2})\omm_n \big) +\left( e^{\Cr{1} \left( \Anu + \kappa_\nu \right) }-1 \right)(m-\frac{1}{2})\omm_n.
			\end{align*}
			With this and \eqref{eq:spherCylindExc}, we can continue to estimate \eqref{eq:excessgammT} with
			\begin{align*}
			\Cr{48} \big( \EE_S(T_\nu, \sigma) + \sigma (\kappa_\nu + \Anu) \big)
			\leq \Cr{48} \big(\EE_C(T_\nu,1) + \kappa_\nu + \Anu \big) 
			\leq \C \mm_\nu^2.
			\end{align*}
			Hence, in both cases we have infinitely many $\nu$ satisfying 
			$$\EE_C(\gamm_{\omega(\nu, \sigma) \#}T_\nu, \frac{\sigma}{4})
			\leq  \C \mm_\nu^2.$$
			For these $\nu$ we apply Lemma \ref{lem:XLessExc} (with $\sigma$, $T$ replaced by $1/5$, $(\gamm_{\omega(\nu, \sigma) \#}\muu_{4/\sigma \#}T_\nu) \LL \UU_3$) and infer
			\begin{align*}
			\sup_{\BC_{\sigma/5} \cap \spt(\gamm_{\omega(\nu, \sigma) \#}T_\nu)} &|X_{n+1}|
			= \sup_{\BC_{4/5} \cap \spt(\gamm_{\omega(\nu, \sigma) \#} \muu_{4/\sigma \#} T_\nu)} \frac{\sigma}{4} |X_{n+1}|\\
			&\leq \frac{\sigma}{4} \sqrt{\Cr{21}\Cr{22} 5^{2n+1} \big( \EE_C(\gamm_{\omega(\nu, \sigma) \#} T_\nu, \frac{\sigma}{4}) + \frac{\sigma}{4}\left( \kappa_\nu + \Anu \right) \big) }\\
			&\leq \Cl{52} \mm_\nu \sigma.
			\end{align*}
			\begin{minipage}[t]{\linewidth}
				\begin{minipage}{0.55 \linewidth}
					With this, we now prove the bound on $f_i$ and $g_j$.
					
					To be able to jump between $\VV$ and $\WW$, we define for $y \in \R^n$ the map $y \mapsto \bar{y}:=(y_1, \dots, y_{n-1}, -y_n)$. Denote by $\vin$ and $\wjn$ the maps whose $\Phii_\nu$-graphs form the $\spt(T_\nu)$ as in Definition \ref{def:harmblow}. By the previous inequality and \eqref{eq:omegatozero}, we can bound for infintely many $\nu$, arbitrary $0<\tau<1$, $i \in \{1, \dots, m\}$ and  $j \in \{1, \dots, m-1\}$ 
				\end{minipage}
				\hspace{0.02\linewidth}
				\begin{minipage}[t]{0.35 \linewidth}
					\vspace*{-2cm}
					\includegraphics[scale=0.5]{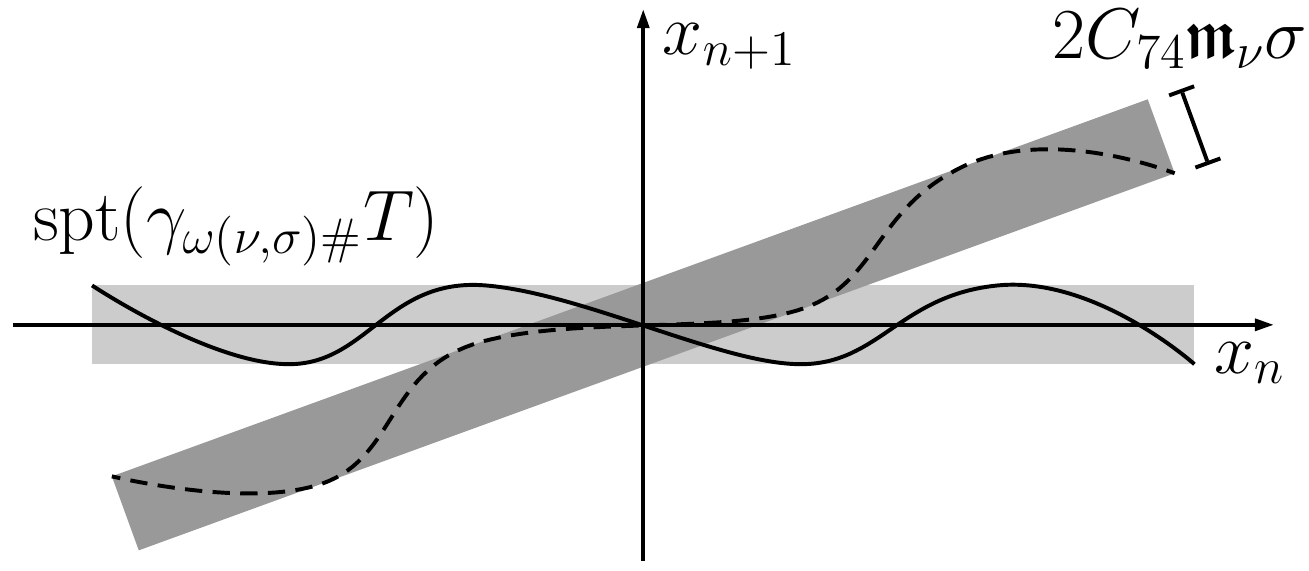}
				\end{minipage}
			\end{minipage}
			\begin{align*}
			\big|\vin(y)+\wjn(\bar{y}) \big| \leq 2\Cr{52}\mm_\nu \sigma \qquad \text{ for } y \in \VV_\tau \cap \BB^n_{\sigma/5}(0).
			\end{align*}
			Consider now any $0 \neq y \in \VV \cap \BB^n_{1/60}(0)$. Then let $\sigma := 5|y| \leq 1/12$. The previous bounds imply that
			$$\left| \frac{\vin(y)}{\mm_\nu} + \frac{ \wjn(\bar{y})}{\mm_\nu} \right| \leq 2 \Cr{52}\sigma = 10\Cr{52}|y|$$
			for infintely many $\nu$. Hence, by local uniform convergence,
			\begin{align}\label{eq:fiMinusfk1}
			|f_i(y) + g_j(\bar{y})| \leq 10\Cr{52}|y| \qquad \text{ for } y \in \VV \cap \BB^n_{1/60}(0).
			\end{align}
			
			Moreover, by \eqref{eq:fPlusg}, for $y \in \VV \cap \big( \BB^n_\rho(0) \setminus \BB^n_{1/60}(0) \big)$, $i \in \{1, \dots, m\}$ and $j \in \{1, \dots, m-1\}$, the following holds
			\begin{align}\label{eq:fiMinusfk4}
			|f_i(y)|^2 + |g_j(\bar{y})|^2
			\leq \frac{4\Cr{21}\Cr{22}}{(1-\rho)^{2n+1}} (60|y|)^2.
			\end{align}
			Now, we define the following auxiliary functions
			\begin{align*}
			h: \UU^n_1(0) \to \R,\ h(y)&= 
			\begin{cases}
			\sum_{i=1}^m f_i(y), &\text{ for } y \in \VV\\
			\sum_{j=1}^{m-1} g_j(y), &\text{ for } y \in \WW\\
			0, &\text{ for } y \in \BL
			\end{cases},\\
			H: \UU^n_1(0) \to \R,\ H(y)&=h(y) - h(\bar{y}).
			\end{align*}
			By Lemma \ref{lem:harmblowupgradient}, these two functions have locally square integrable weak gradients. Moreover, $H$ is odd in the $n$-th variable and $\left. H \right|_{\VV\cup\WW}$ is harmonic. The weak version of the Schwarz reflection principle implies that $H$ is harmonic on all $\UU^n_1(0)$. Therefore, the following holds for all $0<\rho<1$
			\begin{align}\label{eq:Hnormy}
			\sup_{\BB^n_\rho(0)} \frac{|H(y)|}{|y|} < \infty.
			\end{align}
			Notice that for $y \in \VV$, we can write
			\begin{align*}
			f_i(y)&= H(y) - \sum_{k=1}^{i-1} \big(f_k(y) + g_k(\bar{y})\big) - \sum_{k=i+1}^m \big(f_k(y) + g_{k-1}(\bar{y}) \big),\\
			g_j(\bar{y})&= \big( f_1(y) +g_j(\bar{y}) \big) -f_1(y).
			\end{align*}
			\eqref{eq:fiMinusfk1}, \eqref{eq:fiMinusfk4} and \eqref{eq:Hnormy} then imply the lemma.
			
		\end{proof}
	
		Now, we are ready to prove that all harmonic blowups coincide even if they are not linear. The definition of the homogeneous blow-up of the harmonic blow-ups and the estimate in Theorem \ref{thm:graph}$(v.)$ will imply that they are linear, and hence, coincide with each other. Then we will use the E.Hopf boundary point Lemma for harmonic functions to deduce that also the harmonic blow-ups need to coincide themselves.
		
		\begin{thm}\label{thm:blowupsC2}
			Let $\{(T_\nu, \M_\nu)\}_{\nu \geq 1} \subset \TT$ be a blowup sequence with harmonic blowups $f_i$, $g_j$. Then
			\begin{enumerate}[(i.)]
				\item $f_1 = \cdots = f_m$ and $g_1 = \cdots = g_{m-1}.$
				\item The functions 
					\begin{align*}
					f: \VV \cup \BL \to \R,~ y \mapsto &\begin{cases}
					f_1(y), &\textnormal{ for } y \in \VV\\
					0, &\textnormal{ for } y \in \BL
					\end{cases}
					\end{align*}
					\begin{align*}
					g: \WW \cup \BL \to \R,\ y \mapsto &\begin{cases}
					g_1(y), &\textnormal{ for } y \in \WW\\
					0, &\textnormal{ for } y \in \BL
					\end{cases}
					\end{align*}
%
%
				are $\CC^2$.
				\item $Df(0)=Dg(0)$.
			\end{enumerate}
		\end{thm}
		
		\begin{proof}
			We first blow $f_i$, $g_j$ up and show the equality of these limiting functions. Then we deduce that also the $f_i$, $g_j$ coincide.
			
			Let  $i \in \{1, \dots, m\}$, $j \in \{1, \dots, m-1\}$, $4 \leq \rho < \infty$ and define the functions $f_i^{(\rho)}:= \rho f_i ( \frac{ \cdot}{\rho})$ and $g_j^{(\rho)}:= \rho g_j ( \frac{ \cdot}{\rho})$. Then $f_i^{(\rho)}$ and $g_j^{(\rho)}$ are harmonic and by Lemma \ref{lem:fDividedY} uniformly bounded. 
			
			Indeed, for all $4 \leq \rho < \infty$
			$$\sup_\VV |f_i^{(\rho)} |= \rho \sup_\VV \left|f_i \Big( \frac{ y}{\rho} \Big) \right|
			= \rho \sup_{\VV \cap \BB^n_{1/\rho}(0)} |f_i|
			\leq \sup_{\VV \cap \BB^n_{1/\rho}(0)} \frac{ |f_i(y)|}{|y|} 
			\leq \sup_{\VV \cap \BB^n_{1/4}(0)} \frac{ |f_i(y)|}{|y|}< \infty .$$
			Then \cite[Theorem 2.11]{gilbarg} implies that, up to subsequence, they converge pointwise to a harmonic function. This means, there exist a strictly increasing sequence $\rho_k \to \infty$ as $k \to \infty$ and harmonic functions $f_1^*, \dots, f_m^*$ on $\VV$, $g_1^*, \dots, g_{m-1}^*$ on $\WW$ such that for all $i \in \{1, \dots, m\}$, $j \in \{1, \dots, m-1\}$
			\begin{align*}
			\lim_{k \to \infty} f_i^{(\rho_k)}(y) = f_i^*(y) \qquad &\text{ and } \qquad  \lim_{k \to \infty} Df_i^{(\rho_k)}(y) = Df_i^*(y) \qquad \text{ for } y \in \VV,\\
			\lim_{k \to \infty} g_j^{(\rho_k)}(y) = g_j^*(y) \qquad &\text{ and } \qquad  \lim_{k \to \infty} Dg_j^{(\rho_k)}(y) = Dg_j^*(y)
			\qquad \text{ for } y \in \WW.
			\end{align*}
			We want to deduce their equality by using Lemma \ref{lem:harmblowequal}. To do so, we first must show that $f_i^*$, $g_j^*$ are of the form $\beta Y_n$ for some $\beta \in \R$. A sufficient condition for this is the following identity $ \displaystyle \frac{\partial}{\partial r} \frac{f_i^*(y)}{|y|} = 0 = \frac{\partial}{\partial r} \frac{g_i^*(\bar{y})}{|\bar{y}|} $, as we have seen in the proof of Theorem \ref{thm:CylSpherExc}. By Theorem \ref{thm:graph}$(v.)$, we have
			$$\int_{\VV} \left( \frac{\partial}{\partial r} \frac{f_i(y)}{|y|} \right)^2 |y|^{2-n} \dd \Leb^n(y)
			+  \int_{\WW} \left( \frac{\partial}{\partial r} \frac{g_j(y)}{|y|} \right)^2 |y|^{2-n} \dd \Leb^n(y) \leq 2^{n+5}\C < \infty,$$
			and hence, Fatou's Lemma implies that
			\begin{align*}
			\int_{\VV} &\left( \frac{\partial}{\partial r} \frac{f_i^*(y)}{|y|} \right)^2 |y|^{2-n} \dd \Leb^n(y)
			+  \int_{\WW} \left( \frac{\partial}{\partial r} \frac{g_j^*(y)}{|y|} \right)^2 |y|^{2-n} \dd \Leb^n(y)\\
			&\leq \liminf_{k \to \infty} \left( \int_{\VV} \left( \frac{\partial}{\partial r} \frac{f_i^{(\rho_k)}(y)}{|y|} \right)^2 |y|^{2-n} \dd \Leb^n(y)
			+  \int_{\WW} \left( \frac{\partial}{\partial r} \frac{g_j^{(\rho_k)}(y)}{|y|} \right)^2 |y|^{2-n} \dd \Leb^n(y) \right) \\
			&\leq \liminf_{k \to \infty} \left( \int_{\VV \cap \BB^n_{1/\rho_k}(0)} \left( \frac{\partial}{\partial r} \frac{f_i(y)}{|y|} \right)^2 |y|^{2-n} \dd \Leb^n(y)
			+  \int_{\WW \cap \BB^n_{1/\rho_k}(0)} \left( \frac{\partial}{\partial r} \frac{g_j(y)}{|y|} \right)^2 |y|^{2-n} \dd \Leb^n(y) \right) \\
			&=0.
			\end{align*}
			Therefore, there exist real numbers $\beta_1 \leq \cdots \leq \beta_m$, $\gamma_1 \geq \cdots \geq \gamma_{m-1}$ such that 
			$ f_i^* = \beta_i \left. Y_n \right|_\VV$, $
			g_j^* = \gamma_j \left. Y_n \right|_\WW.$ Now, we show that all these numbers coincide.\\
			This must hold by Lemma \ref{lem:harmblowequal}, if we find a blowup sequence whose associated harmonic blowups are exactly $ f_i^*$, $g_j^* $.
			For $k \in \N$, $k \geq 1$, we define $$S_\nu^k := (\muu_{\rho_k} T_\nu) \LL \UU_3.$$ Then there is an $N>0$ such that for $\nu \geq N$ the following holds $\EE_C(T_\nu, 1) + \kappa_\nu + \Anu \leq \frac{1}{\Cr{28}}$ and hence, by Remark \ref{rmk:scale3}, $(S_\nu^k, \muu_{\rho_k}(\M_\nu)) \in \TT$. Moreover, by Definition \ref{def:harmblow}$(iv.)$, $(v.)$ for all $i \in \{1, \dots, m\}$, $j \in \{1, \dots, m-1\}$ we have
			\begin{align*}
			\lim_{\nu \to \infty} \frac{v_i^{S_\nu^k}}{\mm_\nu} &= f_i^{(\rho_k)} \qquad \text{ on compact subsets of } \VV,\\
			\lim_{\nu \to \infty} \frac{w_j^{S_\nu^k}}{\mm_\nu} &= g_j^{(\rho_k)} \qquad \text{ on compact subsets of } \WW.
			\end{align*}
			We choose now for every $k$ an $\nu_k \geq \max\{N, k\}$ satisfying the following three properties:
			\begin{enumerate}
				\item $$
				\max \Big\{ \sup_{\VV \cap \BB_{1/2}} \big|f_1^{(\rho_k)}\big|,
				\sup_{\VV \cap \BB_{1/2}} \big|f_m^{(\rho_k)}\big|,
				\sup_{\WW \cap \BB_{1/2}} \big|g_1^{(\rho_k)}\big|,
				\sup_{\WW \cap \BB_{1/2}} \big|g_{m-1}^{(\rho_k)}\big| \Big\}
				\leq \sup_{\BC_{1/2} \cap \spt{S_{\nu_k}^k}} \frac{|X_{n+1}|}{\mm_{\nu_k}} + \frac{1}{k}.
				$$
				\item $$ \sup_{\BC_{3/2} \cap \spt(S_{\nu_k}^k)} \frac{|X_{n+1}|}{\mm_{\nu_k}}
				\leq 3  \max \Big\{ \sup_{\VV} \big|f_1^{(\rho_k/3)}\big|,
				\sup_{\VV } \big|f_m^{(\rho_k/3)}\big|,
				\sup_{\WW } \big|g_1^{(\rho_k/3)}\big|,
				\sup_{\WW } \big|g_{m-1}^{(\rho_k/3)}\big| \Big\} + \frac{1}{k}.$$
				This is possible by Lemma \ref{lem:harmblowbound}, where $\{(T_\nu, \M_\nu)\}_{\nu \geq 1}$, $a$, $\sigma$ are replaced by\\ $\{(\muu_{\rho_k/3 \#}T_{\nu_k}, \muu_{\rho_k/3}(\M_{\nu_k}) \}_{k \geq 1}$,  $0$, $ 1/2$) and because
				\begin{equation*}
				\begin{split}
				\sup_{\BC_{3/2} \cap \spt(S_{\nu_k}^k)} \frac{|X_{n+1}|}{\mm_{\nu_k}}
				&= \sup_{\BC_{1/2} \cap \spt(\muu_{1/3 \#} S_{\nu_k}^k)} 3 \frac{|X_{n+1}|}{\mm_{\nu_k}}
				= 3\sup_{\BC_{1/2} \cap \spt(\muu_{\rho_k/3 \#} T_{\nu_k}^k)}  \frac{|X_{n+1}|}{\mm_{\nu_k}}. 
				\end{split}
				\end{equation*}
				\item We define the (blowup) sequence $\{(S^*_k, \M^*_k)\}_{k \geq 1}$ by $S^*_k := S^k_{\nu_k}$ and $\M^*_k:= \muu_{\rho_k} (\M_{\nu_k})$ and notice
				\begin{align}\label{eq:vDividedEps}
				\lim_{k \to \infty} \frac{v_i^{S_k^*}}{\mm_{\nu_k}} = f_i^* \qquad \text{ and } \qquad
				\lim_{k \to \infty} \frac{w_j^{S_k^*}}{\mm_{\nu_k}} = g_j^* .
				\end{align}
			\end{enumerate}
			
			If all $f_i^*$, $g_j^*$ vanish, then also $0=\beta_1 = \cdots = \beta_m= \gamma_1 = \cdots = \gamma_{m-1}$.
			If not, we want to see whether $\{S_k^*\}_{k\geq1}$ is a blowup sequence to $f_i^*$, $g_j^*$. Hence, we aim for \eqref{eq:vDividedEps} with $\mm_{\nu_k}$ replaced by $\mm_{S^*_k}$. Therefore, we shall compare these two quantities. First, we notice that by Remark \ref{rmk:scale3},
			$$0 \leq \frac{\kappa_{S_k^*}+\Aa_{\M^*_k}}{\mm_{\nu_k}^2}
			\leq \frac{\kappa_{\nu_k} +\Aa_{\nu_k}}{\rho_k \mm_{\nu_k}^2}
			\rightarrow 0 \qquad \text{as } k \to \infty.$$
			Then by Lemmas \ref{lem:excLessX} and \ref{lem:XLessExc} (with $T$, $\M$, $\sigma$ replaced by $S_k^*$, $\M^*_k$ $1/2$) and the conditions 1. and 2., it follows that
			\begin{align*}
			\limsup_{k \to \infty} \frac{\EE_C(S_k^*, 1)}{ \mm_{\nu_k}^2}
			&\leq \limsup_{k \to \infty} 4\Cr{18} \left( \frac{\Cr{19}}{\mm_{\nu_k}^2} \sup_{\BC_{3/2} \cap \spt(S_k^*)} X_{n+1}^2 + \frac{\kappa_{S_k^*}+\Aa_{\M^*_k}} {\mm_{\nu_k}^2} \right) \\
			&\leq  36\Cr{18}\Cr{19} \max\big\{\sup_\VV (f_i^*)^2, \sup_\WW(g_j^*)^2: i,j \big\}, \\
			\liminf_{k\to \infty} \frac{\EE_C(S_k^*, 1)}{ \mm_{\nu_k}^2} 
			&\geq \liminf_{k\to \infty} \left( \frac{1}{2^{2n+1}\Cr{21}\Cr{22}}  \sup_{\BC_{1/2} \cap \spt(S_k^*)}\frac{X_{n+1}^2}{\mm_{\nu_k}^2} - \frac{\kappa_{S_k^*}+\Aa_{\M^*_k}}{\mm_{\nu_k}^2} \right) \\
			&\geq \frac{1}{2^{2n+1}\Cr{21}\Cr{22}} \max\big\{\sup_\VV (f_i^*)^2, \sup_\WW(g_j^*)^2: i,j \big\}.
			\end{align*}
			Hence, 
			$$0 < \liminf_{k\to \infty} \frac{\max\big\{\EE_C(S_k^*, 1), \textbf{A}_{S_k^*}^{1/2} \big\}}{ \mm_{\nu_k}^2}  
			\leq \limsup_{k \to \infty} \frac{\max\big\{\EE_C(S_k^*, 1), \textbf{A}_{S_k^*}^{1/2} \big\}}{ \mm_{\nu_k}^2} < \infty,$$
			and we can find a subsequence $\{(S_{k_l}^*, \M^*_{k_l})\}_{l \geq 1}$ which is a blowup sequence and whose associated harmonic blowups are $\gamma f_i^*$, $\gamma g_j^*$ for some fixed $\gamma \in \R$ by \eqref{eq:vDividedEps}. As they are of the form as in Lemma \ref{lem:harmblowequal} it follows that there is a $\beta \in \R$ satisfying
			$$f_1^*= \cdots = f_m^* = \beta \left. Y_n \right|_\VV 
			\qquad \text{ and } \qquad
			g_1^* = \cdots = g_{m-1}^* = \beta \left. Y_n \right|_\WW.$$
			From this, we want to deduce that also $f_1= \cdots = f_m$ and $g_1 = \cdots = g_{m-1}$. Notice that $f_1 -f_m$ and $g_1-g_{m-1}$ are nonpositive and harmonic functions. By Lemma \ref{lem:harmblowupgradient}, $f_i$ and $g_j$ have zero trace on $\BL$. Hence,
			$$ \sup_\VV (f_1 -f_m)= 0 = \sup_\WW (g_1-g_{m-1}).$$
			Moreover, the E. Hopf boundary point Lemma \cite[Lemma 3.4]{gilbarg} implies that
			if $y_0 \in \BL$ is a strict maximum point, then the outer normal derivative at $y_0$ (if it exists) must be positive. But at zero, the following holds
			\begin{align*}
			\frac{ \partial (f_1-f_m) }{\partial \nu}(0)
			= \lim_{t \downarrow 0} \frac{ (f_1-f_m)(0, \dots, 0,t)}{t}
			&= (f_1^*-f_m^*)(0, \dots, 0,1) =0,\\
			\frac{ \partial (g_1-g_{m-1}) }{\partial \nu}(0)
			= \lim_{t \downarrow 0} \frac{ (g_1-g_{m-1})(0, \dots, 0,-t)}{t}
			&= (g_1^*-g_{m-1}^*)(0, \dots, 0,-1) =0.
			\end{align*}
			Hence, $0$ is not a strict maximum point and there must be a point in $\VV$ ($\WW$ respectively) reaching $0$ (i.e. the maximum) as well.
			Then \cite[Theorem 3.5]{gilbarg} implies that $f_1-f_m$, and $g_1-g_{m-1}$ must be constant. In fact, by the vanishing trace, $f_1 -f_m= 0= g_1-g_{m-1}$. Therefore, $(i.)$ must hold.
			Also by the vanishing trace and weak version of the Schwarz reflection principle, there are harmonic functions $f \in \CC^2(\VV \cap \BL)$, $g \in \CC^2(\WW \cup \BL)$ satisfying $(ii.)$ and $(iii.)$.
			
		\end{proof}
	
		\begin{rmk}\label{rmk:DfBounds}
			Let $f$, $g$ denote harmonic blow-ups as in Theorem \ref{thm:blowupsC2}$(ii.)$. Then there are constants $\Cl{53}$, $\Cr{54}$ such that
			\begin{enumerate}[(i.)]
				\item $ \displaystyle |Df(0)|= |Dg(0)| 
				\leq \Cr{53} \min \left\{ \sqrt{ \int_{\VV \cap \UU^n_{1/2}(0)} |f|^2 \dd \Leb^n },  \sqrt{ \int_{\WW \cap \UU^n_{1/2}(0)} |g|^2 \dd \Leb^n }  \right\}
				\leq \Cr{54}$.
				\item For all $y \in \BB^n_{1/4}(0)$ the following holds
				$$ |f(y) - y \cdot Df(0)|
				\leq  \Cr{53} |y|^2 \sqrt{ \int_{\VV \cap \UU^n_{1/2}(0)} |f|^2 \dd \Leb^n }
				\leq \Cr{54}|y|^2.$$
				\item For all $y \in \BB^n_{1/4}(0)$ the following holds
				$$ |g(y) - y \cdot Dg(0)|
				\leq  \Cr{53}|y|^2 \sqrt{ \int_{\WW \cap \UU^n_{1/2}(0)} |g|^2 \dd \Leb^n }
				\leq \Cr{54}|y|^2.$$
			\end{enumerate}
		\end{rmk}
	
		\begin{proof}
			$(i.)$
			By the Schwarz reflection principle, we can extend $f$ to an harmonic function $\tilde{f}$ defined on $\BB^n_{1/2}(0)$. Then by the interior estimates for harmonic functions \cite[Theorem 2.10]{gilbarg}, the mean value property and Hölder's inequality, it follows that
			$$|Df(0)| \leq 8n \sup_{ \BB^n_{1/4}} |\tilde{f}|
			\leq 8n \frac{2^n}{\omm_n} \int_{\BB^n_{1/2}}|\tilde{f}| \dd \Leb^n
			\leq 8n \left(\frac{ 2^n}{\omm_n}\right)^2 \sqrt{ \int_{\BB^n_{1/2}}|f|^2 \dd \Leb^n}.$$
			Moreover, by Lemma \ref{lem:XLessExc}$(ii.)$ (with $\sigma$ replaced by $1/2$) and Definition \ref{def:harmblow}$(iii.)$, this integral is bounded by $2^{n+1}\Cr{22}$.
			The same holds for $g$.
			
			$(ii.)$ By the Taylor formula, $ |f(y) - y \cdot Df(0)| \leq C |D^2f(0)||y|^2$. Also by \cite[Theorem 2.10]{gilbarg}, the following holds
			$$|D^2f(0)| \leq \frac{ n^2}{16} \sup_{ \BB^n_{1/4}} |\tilde{f}|.$$
			The inequalities follow then as in $(i.)$.
			
			$(iii.)$ Similar to $(ii.)$.
		\end{proof}
	
		\section{Excess decay}
		With the $\CC^2$ functions from Theorem \ref{thm:blowupsC2}, we prove the following inequalities of the excess. We will use them to prove Theorem \ref{thm:RotExcessDecay} by constructing inductively a sequence of currents which will show that the excess of the (slightly rotated) original current decays at most proportional to the radius.
		
		\begin{thm}\label{thm:excessThetaMax}
			Let $(T, \M) \in \TT$ and define $\theta := \big(\Cr{29}(1+\Cr{54}) \big)^{-2}$ (see Remarks \ref{rmk:mugammT}$(iii.)$ and \ref{rmk:DfBounds}). There is a constant $\Cr{55} \geq 1$ such that if $T$ fulfils $\max\{\EE_C(T,1), \Cr{55}\kappa_T, \sqrt{\Aa}\} \leq \frac{1}{\Cr{55}}$, then there is a real number $\omega$ satisfying
			$$|\omega|^2 \leq \Cr{54}^2 \max\left\{\EE_C(T,1), \sqrt{\Aa} \right\}
			\quad \text{ and } \quad
			\EE_C(\gamm_{\omega \#}T, \theta) \leq \theta \max\left\{\EE_C(T,1), \Cr{55}\kappa_T, \sqrt{\Aa} \right\} .$$
		\end{thm}
		\begin{proof}
			We argue by contradiction. If the theorem did not hold, then there would be a sequence $\{(T_\nu, \M_\nu)\}_{\nu \geq1} \subset \TT$ such that for all $|\omega| \leq \Cr{54}\mm_\nu$ the following holds
			\begin{align}
			\max\{\epnu^2, \sqrt{\Anu}, \nu \kappa_\nu\} &\leq \frac{1}{\nu},\label{eq:epnuKappaSmall} \\
			\EE_C(\gamm_{\omega \#}T_\nu, \theta) &> \theta \max\{\epnu^2, \sqrt{\Anu}, \nu\kappa_\nu \} \label{eq:ExcGammTheta},
			\end{align}
			where $\epnu:= \sqrt{\EE_C(T_\nu,1)}$, $\kappa_\nu := \kappa_{T_\nu}$ and $\Anu := \Aa_{\M_\nu}$.
			Notice that by the monotonicity of the excess \eqref{eq:excmon}, the condition \eqref{eq:ExcGammTheta} (with $\omega=0$) implies
			$$\theta \nu \kappa_\nu
			\leq \theta \max\{\EE_C(T_\nu,1), \sqrt{\Anu}, \nu\kappa_\nu \} 
			<\EE_C(T_\nu, \theta)
			\leq \frac{\epnu^2}{\theta^n}.$$
			Hence, by \eqref{eq:epnuKappaSmall}, we can assume that 
			\begin{align*}
			\epnu^2 + \frac{\kappa_\nu}{\epnu^2} + \Anu < \frac{2}{\nu} + \frac{1}{\nu \theta^{n+1}}.
			\end{align*}
			Therefore, we notice that as in \eqref{eq:existenceHarmBlowup}, $\{(T_\nu, \M_\nu)\}_{\nu \geq1}$ is, up to subsequence, a blowup sequence with associated harmonic blowups $f_i$, $g_j$. Let $f$, $g$ denote the $\CC^2$-functions as in Theorem \ref{thm:blowupsC2}$(ii.)$.
			As they vanish on $\BL$, for every $0< \sigma <1$ the functions $\epnu^{-1}\vin$, $\epnu^{-1}\wjn$ converge uniformly on $\VV_\sigma$, $\WW_\sigma$. Thus, we derive from Lemma \ref{lem:harmblowbound} that
			\begin{equation}\label{eq:XepnuConvF}
			\begin{split}
			\limsup_{\nu \to \infty} \sup_{\BC_{1/2} \cap \pp^{-1}(\VV) \cap \spt(T_\nu)} \Big| \frac{X_{n+1}}{\mm_\nu} - f \circ \pp \Big| &=0,\\
			\limsup_{\nu \to \infty} \sup_{\BC_{1/2} \cap \pp^{-1}(\WW) \cap \spt(T_\nu)} \Big| \frac{X_{n+1}}{\mm_\nu} - g \circ \pp \Big| &=0.
			\end{split}
			\end{equation}
			From Remark \ref{rmk:DfBounds} and the proof of Theorem \ref{thm:blowupsC2}, we deduce the existence of some $\beta \in [-\Cr{54},\Cr{54}]$ satisfying
			\begin{align*}
			Df(0)=(0, \dots, 0,\beta)=Dg(0).
			\end{align*}
			Therefore, by applying Remark \ref{rmk:DfBounds}$(ii.)$, $(iii.)$, it follows
			\begin{equation}\label{eq:fMinusBetaX}
			\begin{split}
			\big|f(x) - \beta x_n\big| = \big|f(x) - x Df(0)\big| &\leq \Cr{54}|x|^2
			\qquad \text{ for } x \in \VV \cap \BB^n_{1/4}(0),\\
			\big|g(x) - \beta x_n\big| = \big|g(x) - x Dg(0)\big| &\leq \Cr{54}|x|^2
			\qquad \text{ for } x \in \WW \cap \BB^n_{1/4}(0).
			\end{split}
			\end{equation}
			Then we rotate the currents such that the new differential vanishes. Indeed, let $\omega_\nu := \arctan(\beta \mm_\nu)$. Then
			\begin{align}\label{eq:omegaBounded}
			|\omega_\nu| \leq |\beta| \mm_\nu \leq \Cr{54} \mm_\nu.
			\end{align}
			Consider now $S_\nu := (\muu_{1/\theta \#} \gamm_{\omega_\nu \#} T_\nu ) \LL \UU_3$ and $\tilde{\M}_\nu:= \muu_{1/\theta }/\M_\nu)$. By \eqref{eq:epnuKappaSmall}, the assumptions of Remark \ref{rmk:mugammT}$(iii.)$ are fulfilled  for $\nu$ large enough, and hence, $(S_\nu, \tilde{\M}_\nu) \in \TT$ and
			\begin{align}\label{eq:kappaSnu}
			\kappa_{S_\nu} \leq \theta \kappa_\nu, \qquad \Aa_{\tilde{\M}_\nu} \leq \theta \Anu.
			\end{align}
			By \eqref{eq:XepnuConvF}, \eqref{eq:fMinusBetaX} 
			and the Remark \ref{rmk:DfBounds}$(ii.)$, $(iii.)$, it follows
			\begin{align*}
			\limsup_{\nu \to \infty} \sup_{\BC_2 \cap \spt(S_\nu) }\left| \frac{X_{n+1}}{ \mm_\nu} \right|
			&\leq \limsup_{\nu \to \infty}  \sup_{\BC_{3} \cap \spt(\muu_{1/\theta \#} T_\nu)} \Big| \frac{X_{n+1}} { \mm_\nu} - \beta X_n \Big|\\
			&\leq  \frac{1}{\theta} \limsup_{\nu \to \infty}  \sup_{\BC_{3 \theta} \cap \spt(T_\nu)} \Big| \frac{X_{n+1}}{\mm_\nu} - \beta X_n \Big|\\
			&\leq \frac{1}{\theta} \limsup_{\nu \to \infty} \left(  \sup_{\substack{\BC_{3 \theta} \cap \VV\\ \cap \spt(T_\nu)}}
			|f\circ \pp - \beta Y_n| + \sup_{\substack{\BC_{3 \theta} \cap \WW\\ \cap \spt(T_\nu)}} |g \circ \pp - \beta Y_n| \right) \\
			&\leq  \frac{1}{\theta} \Cr{54} \big((3\theta)^2 + (3\theta)^2 \big)\\
			&= 18\Cr{54} \theta.
			\end{align*}
			Together with Lemma \ref{lem:excLessX} (with $\sigma \uparrow 1$ and $T$ replaced by $S_\nu$), \eqref{eq:kappaSnu} and Definition \ref{def:harmblow}$(iii.)$, we yield
			\begin{align*}
			\limsup_{\nu \to \infty} \frac{\EE_C(\gamm_{\omega_\nu \#} T_\nu, \theta)}{\mm_\nu^2}
			&=\limsup_{\nu \to \infty} \frac{\EE_C(S_\nu, 1)}{\mm_\nu^2}\\
			&\leq \limsup_{\nu \to \infty} \Cr{18}\left(\frac{ \Cr{19} \sup_{\BC_2 \cap \spt(S_\nu)} X_{n+1}^2}{\mm_\nu^2} + \frac{\kappa_{S_\nu}+\mathbf{A}_{\tilde{\M}_\nu}}{\mm_\nu^2} \right)\\
			&\leq  \Cr{18} \left( \Cr{19}  \limsup_{\nu \to \infty} \sup_{\BC_2 \cap \spt(S_\nu)} \frac{X_{n+1}^2 }{ \mm_\nu^2} + \theta \limsup_{\nu \to \infty} \frac{\kappa_\nu + \Anu}{\mm_\nu^2} \right)\\
			&\leq (18)^2  \Cr{18} \Cr{19} \Cr{54}^2 \theta^2\\
			&< \theta.
			\end{align*}
			As $\omega_\nu$ is bounded (see \eqref{eq:omegaBounded}), the latter inequality contradicts \eqref{eq:ExcGammTheta} for $\nu$ large enough.
			
		\end{proof}

	\subsection{Proof of Theorem \ref{thm:RotExcessDecay}}
		\begin{proof}
			We construct a sequence of currents $\{(T_\nu, \M_\nu)\}_{\nu \in \N} \subset \TT$ and real numbers $\{\omega_\nu\}_{\nu \geq 1}$ inductively. We start with $(T_0, \M_0):=(T, \M)$. Assume that for some fixed $j \in \N$, we already have $(T_j, \M_j) \in \TT$ and denote by $\Aa_j ;=\Aa_{\M_j}$ and $\mm_j := \max\{\sqrt{\EE_C(T_j,1)},  \Aa_j^{1/4}\}$. By Theorem \ref{thm:excessThetaMax}, there is a real number $ |\omega_{j+1}| \leq \Cr{54} \mm_j$ such that if we define
			$$ T_{j+1} := (\muu_{1/\theta \#} \gamm_{\omega_{j+1} \#} T_j) \LL \UU_3 
			\qquad \textnormal{ and } \qquad
			\M_{j+1}:= \muu_{1/\theta}(\M_j) $$
			then $(T_{j+1}, \M_{j+1}) \in \TT$ and by Remark \ref{rmk:mugammT}$(iii.)$
			$$ \max\big\{ \EE_C(T_{j+1},1), \Aa_{j+1}, \Cr{55}\kappa_{T_{j+1}} \big\} 
			\leq \theta \max\big\{ \EE_C(T_j,1), \Aa_j, \Cr{55}\kappa_{T_j} \big\}.$$
			Using this inequality $j$ times, we deduce
			$$\max\big\{ \EE_C(T_{j+1},1), \Aa_j, \Cr{55}\kappa_{T_{j+1}} \big\} 
			\leq \theta^{j+1} \max\big\{ \EE_C(T,1), \Aa, \Cr{55}\kappa_{T} \big\}
			\leq \frac{ \theta^{j+2}}{\Cr{55}}. $$
			Moreover, the following holds
			\begin{align}
			|\omega_{j+1}| &\leq \Cr{54} \sqrt{ \frac{ \theta^{j+1}}{\Cr{55}}}, \label{eq:omegaJ}\\
			\EE_C(T_j, 1) + \kappa_{T_j} + \Aa_j \leq 3 \max&\big\{\EE_C(T_j, 1),\Aa_j , \kappa_{T_j} \big\} \leq 3\frac{ \theta^{j+1}}{\Cr{55}}. \label{eq:ExcPlusKappa}
			\end{align}
			Then we define $\displaystyle \eta_j := \sum_{k=1}^j \omega_k$ and $\eta := \lim\limits_{j \to \infty} \eta_j$. This is a valid choice for $\eta$ as \eqref{eq:omegaJ} and the fact that $\theta^{1/2} \leq 1/2$ implies
			$$ |\eta| \leq \Cr{54} \sum_{k=1}^\infty  \sqrt{ \frac{ \theta^{k}}{\Cr{55}}}
			= \frac{\Cr{54}}{\sqrt{\Cr{55}}}  \sum_{k=1}^\infty (\theta^{1/2})^k
			= \frac{\Cr{54}}{\sqrt{\Cr{55}}} \frac{\theta^{1/2}}{1-\theta^{1/2}}
			\leq 2 \frac{\Cr{54}}{\sqrt{\Cr{55}}} \theta^{1/2}. $$
			Fix $0<r<\theta/4$ and choose an appropriate $j \in \N$ such that $\theta^{j+1} \leq 4r < \theta^j$. Then we use the inequalities \eqref{eq:omegaJ}, \eqref{eq:ExcPlusKappa} together with \eqref{eq:excRotSmall} from the proof of Remark \ref{rmk:mugammT}$(iii.)$ (with $T$, $\M$, $\omega$ replaced by $T_j$, $\M_j$, $\eta - \eta_j$) and the excess monotonicity \eqref{eq:excmon} to derive	
			\begin{align*}
			\EE_C(\gamm_{\eta \#}T, r)
			&\leq \left( \frac{\theta^j}{4r} \right)^n \EE_C\big(\gamm_{\eta \#}T, \frac{\theta^j}{4} \big)
			\leq \theta^{-n} \EE_C\big(\gamm_{\eta \#}T, \frac{\theta^j}{4} \big)\\
			&= \theta^{-n} \EE_C\big(\muu_{4\#}\gamm_{\eta \#}T, \theta^j \big)\\
			&=  \theta^{-n} \EE_C\big(\gamm_{\eta_j \#}\muu_{4\#}\gamm_{\eta-\eta_j \#}T, \theta^j \big)\\
			&= \theta^{-n} \EE_C\big(\muu_{(1/\theta)^j\#} \gamm_{\eta_j \#}\muu_{4\#}\gamm_{\eta-\eta_j \#}T, 1 \big)\\
			&=\theta^{-n} \EE_C\big(\muu_{4\#} \gamm_{\eta-\eta_j \#}T_j, 1 \big)\\
			&\leq \theta^{-n} \frac{\Cr{29}}{\Cr{28}} \left( \Big( \sum_{k=j+1}^\infty \omega_k \Big)^2 + \EE_C(T_j,1) + \kappa_{T_j} +\Aa_j \right)\\
			&\leq \theta^{-n} \frac{\Cr{29}}{\Cr{28}} \left( \sum_{k=j+1}^\infty \omega_k^2 + 3 \frac{\theta^{j+1}}{\Cr{55}} \right)\\
			&\leq \theta^{-n} \frac{\Cr{29}}{\Cr{28}} \left( \frac{\Cr{54}^2}{\Cr{55}}  \frac{\theta^{j+1}}{1-\theta} + 3 \frac{\theta^{j+1}}{\Cr{55}} \right)\\
			&\leq \theta^{-n} \frac{\Cr{29}}{\Cr{28}} \frac{3(\Cr{54}^2+1)}{\Cr{55}} \theta^{j+1}\\
			&\leq \theta^{-n} \frac{\Cr{29}}{\Cr{28}} \frac{3(\Cr{54}^2+1)}{\Cr{55}} (4r)\\
			& \leq \frac{ \theta^{-n-1}}{\Cr{55}} ~ r.
			\end{align*}
		\end{proof}
	\newpage
\section{The boundary regularity Theorem}\label{lastproof}\label{sec:final}
	\begin{thm}\label{thm:lastTheorem}
		Let $U \subset \R^{n+k}$ be open and $T$ an $n$-dimensional locally rectifiable current in $U$ that is  area minimizing in some smooth $(n+1)$-manifold $\M$ and such that $\partial T$ is an oriented $\CC^{2}$ submanifold of $U$. Then for any point $a \in \spt(\partial T)$, there is a neighborhood $V$ of $a$ in $U$ satisfying that $V \cap \spt(T)$ is an embedded $\CC^{1,\frac14}$ submanifold with boundary.
	\end{thm}
	
Hardt and Simon found out, that it is enough to consider currents whose tangent cones at boundary are in fact a tangent planes. Once we have this tangent plane, we can parametrize the support of the current with graphs over the plane.
	
	\begin{lem}\label{lem:ConesAtCones}
		Let $Q \in \RR_n^{loc}(\R^{n+1})$ be an absolutely area minimizing cone 			with $\partial Q = \EE^{n-1} \times \del_0 \times \del_0$. Then, the support 			of $Q$ is contained in a hyperplane.
	\end{lem}
	
	\begin{proof}
		This can be read in the original paper \cite{hardtSimon}.
	\end{proof}

	\begin{lem}\label{lem:IfConeHyperplane}
		Let $U$, $T$ and $\M$ be as in Theorem \ref{thm:lastTheorem} and assume further that for every $a \in \spt(\partial T)$, there is a tangent cone $C$ at $a$ such that $\spt(C)$ is contained in a hyperplane. Then for any point $a \in \spt(\partial T)$, there is a neighborhood $V$ of $a$ in $U$ satisfying that $V \cap \spt(T)$ is an embedded $\CC^{1,\frac14}$ submanifold with boundary.
	\end{lem}
	\begin{proof}
		After some translation, reflection and rotation, we can assume wlog that $a=0$ and the hyperplane is $\{(y,0): y \in \R^n \} \subset \R^{n+k}$. Hence, for $m = \Theta^n( \mT, 0) +\frac{1}{2} \in \N$,
		$$ \Big( m \big( \EE^n \LL \{y \in \R^n: y_n>0 \} \big) + (m-1) \big( \EE^n \LL \{y \in \R^n: y_n<0 \} \big) \Big) \times \del _0$$
		is an oriented tangent cone of $T$ at $0$ by \cite[4.1.31(2)]{federer}. Therefore, we find a nullsequence $\{r_k\}_{k\geq 1} \subset \R_+$ such that $\muu_{1/r_k \#} T$ converges in $\RR^{loc}_n(\R^{n+k})$ to this cone as $k \to \infty$. Moreover, we assume that for every $k$ we have $3r_k < \dist(0, \partial U)$. Then it follows that
		\begin{align}\label{eq:XDividedRZero}
		\lim_{k \to \infty} \sup_{\BB_{r_k} \cap \spt(T)} \frac{X_{n+1}}{r_k}
		= \lim_{k \to \infty} \sup_{\BB_1 \cap \spt(\muu_{1/r_k \#}T)} X_{n+1}=0.
		\end{align}
		By \cite[5.4.2]{federer}, also the associated measures converge weakly and hence,
		\begin{align*}
		\lim_{k \to \infty} r_k^{-n} \MM \big( T \LL(\UU_{3r_k} \cap \BC_{r_k}) \big)
		&= \lim_{k \to \infty} \MM \big( (\muu_{1/r_k \#} T) \LL(\UU_{3} \cap \BC_{1}) \big)\\
		&= m\Leb^n(\VV) + (m-1)\Leb^n(\WW)
		= (m-\frac{1}{2})\omm_n,
		\end{align*}
		which implies that
		\begin{align*}
		\lim_{k \to \infty} &\left| r_k^{-n} \MM \big( \pp_\# \big( T \LL(\UU_{3r_k} \cap \BC_{r_k}) \big) \big) -(m-\frac{1}{2})\omm_n \right| \\
		&\qquad \leq \lim_{k \to \infty} \left|  \MM \big( \pp_\# \big( (\muu_{1/r_k \#} T) \LL(\UU_{3} \cap \BC_{1}) \big) \big)- \MM \big( (\muu_{1/r_k \#} T) \LL(\UU_{3} \cap \BC_{1}) \big) \right|\\
		&\qquad =0,
		\end{align*}
		where we also have used \eqref{eq:XDividedRZero}.\\
		Thus, if we define $T_k := (\muu_{1/r_k \#} T) \LL \UU_{3}$ and $\M_k:=\muu_{1/r_k}(\M)$, then for $k$ large enough, we have $(T_k, \M_k) \in \TT$ and
		$$\max\big\{ \EE_C(T_k,1), \Cr{55}\kappa_{T_k}, \Aa_k \big\} \leq \frac{\theta}{\Cr{55}}.$$
		Then we can apply Theorem \ref{thm:RotExcessDecay} (with $T$ replaced by $T_k$) and notice that we can choose $\eta$ to be zero, to find the decay 
		$$ \EE_C(T_k, r)\leq  \frac{\theta^{-n-1}}{\Cr{55}} ~ r.$$
		Now, we differ between two cases.\\
		\textit{Case 1: } $m=1$. 
		This is a corollary of Allard's interior regularity theorem. However,
		a self-contained proof could be given from the results of the previous sections.
		Observe first that, by Corollary \ref{cor:graphsOfGammT}, in a sufficiently small neighborhood of
		$x$, the current $T$ is supported in the $\Phii$-graph of $\tilde{v}_1$ and so we can
		assume, wlog, that $\spt (T)\setminus \spt (\partial T)$ is 
		connected. By the Constancy Lemma, it follows that the density $\Theta$ is an 
		an integer constant $k$ at every interior point of such neighborhood. So the
		current is actually $k$ times the one induced by the $\Phii$-graph of $\tilde{v}_1$.
		However, since the boundary of $T$ is a current with multiplicity $1$ we easily
		conclude that $k$ is actually $1$. The current
		$T$ is thus the current induced by the $\Phii$-graph of the $C^{1,\frac14}$ function $\tilde{v}_1$.
		Notice that there is a neighborhood $U$ of $0$ such that $\Theta^n( \mT, y) = \frac{1}{2}$ for all $y \in U \cap \spt(\partial T)$. 
		
		\textit{Case 2: } $m >1$. We fix $k$ and use Corollary \ref{cor:graphsOfGammT} with $\gamm_{\eta \#} T$ replaced by $T_k$. Hence, we get functions $\tilde{v}_i$, $\tilde{w}_j$ whose $\Phii$-graphs around zero form $\spt(T_k)$. Moreover, we know that $D\tilde{v}_i(0)=0= D\tilde{w}_j$. Hence, similar to the proof of Theorem \ref{thm:blowupsC2}, by the E. Hopf boundary point Lemma for quasilinear equations \cite[Theorem 2.7.1]{pucciSerrin}, we deduce that $\tilde{v}_m-\tilde{v}_1 \equiv 0 \equiv \tilde{w}_{m-1} - \tilde{w}_1$. Therefore, they all coincide. 
		
		Notice that the regular points of
		$$ \UU_{r_k} \cap ( \spt(T) \setminus \spt(\partial T))
		= \muu_{r_k} \big(\UU_1 \cap ( \spt(T_k) \setminus \spt(\partial T_k)) \big)
		\supseteq  \muu_{r_k} \big( \graph(\tilde{v}_1 ) \cup ( \graph(\tilde{w}_1 ) \big)$$
		consist of at least two connected components. Let $G$ denote that component of the regular points containing $ \muu_{r_k} \big( \graph(\tilde{v}_1 ) \big)$ and consider 
		$$S:= \frac{1}{m} (T \LL G).$$
		Notice that by \cite[4.1.31(2)]{federer}, the density $\Theta(\mT, x)$ is constantly $m$ for all $x \in G$.
		We will show later that on some open neighborhood $V$ of $0$ in $U$, we have that $\spt(T) = \spt(T-S)$, $T-S$ has no boundary in $W$ and then, we apply interior regularity theory.\\
		First notice that as $T$, $S$ are area minimizing in $\M$ and
		$ \mT = \lVert S \rVert + \lVert T-S \rVert $ holds,
		is follows that $T-S$ is also area minimizing $\M$.\\
		Then, we denote $W:= \UU_{r_k} \cap \BC_{\delta r_k}$, where $\delta$ is as in Corollary \ref{cor:graphsOfGammT}, and aim to show that 
		\begin{align}\label{eq:BoundSInBoundT}
		(\partial S) \LL W = (\partial T) \LL W.
		\end{align} 
		Notice that 
		$$\spt(\partial S) \subset \spt\big((\partial T) \LL G \big)  \cup \spt\big(T \LL (\partial G) \big)$$
		and hence,
		$$\spt\big((\partial S)\LL W \big) 
		\subset \spt\big((\partial T) \LL W \big)  \cup \spt\big(T \LL (\partial G \cap W )\big)
		= \spt\big((\partial T) \LL W \big) .$$
		Moreover, we can use the Constancy Theorem \cite[4.1.7]{federer} to derive
		\begin{align*}
		\pp_\# \big((\partial S) \LL W \big)
		&= \Big( \partial \big( \frac{1}{m} \pp_\#\big(T \LL (G\cap W)\big) \big) \Big) \LL \pp(W)\\
		&= \Big( \partial \big( \EE^n \LL \{r_k y \in \pp(W): y_n > \varphi_{T_k}(y_1, \dots, y_{n-1}) \} 
		\Big) \LL \pp(W)\\
		&=  \big( \partial \big(\pp_\#(T \LL W) \big) \big)  \LL \pp(W)\\
		&= \pp_\# \big((\partial T) \LL W \big).
		\end{align*}
		As the map $\left. \pp \right|_{\spt((\partial T)\LL W)}$ is a $\CC^2$-diffeomorphism, \eqref{eq:BoundSInBoundT} must hold. Then $T-S$ has in $W$ no boundary and by \eqref{eq:XDividedRZero}, a tangent cone of $T-S$ at $0$ is contained in $X_{n+1}^{-1}(0)$. Therefore, we can apply \cite[Theorem 5.3.18]{federer} to $p_\#(T-S)$ and deduce that there is an open neighborhood $V$ of $0$ in $U$ such that
		$$V \cap \spt(T) = V \cap \spt(T-S)$$
		is a smooth embedded submanifold of $\M$.
	\end{proof}
	
Putting the previous two lemmas together, we deduce the boundary regularity theorem:
	
	\begin{proof}[Proof of Theorem \ref{thm:lastTheorem}]
		Let $a \in \spt(\partial T)$. Then by \cite[Theorem 3.6]{brothers}, $T$  has an absolutely area minimizing tangent cone $Q \in \RR^{loc}_n(T_a\M)$ at $a$.  After some rotation, we can assume that $\partial Q = (-1)^n \EE^{n-1} \times \del_0 \times \del_0$. By Lemma \ref{lem:ConesAtCones}, the cone is contained in some hyperplane and by Lemma \ref{lem:IfConeHyperplane}, we conclude that $T$ is regular at $a$.
	\end{proof}	
	
	\newpage
	
\section{Proof of the technical statements}\label{chapter:proofs}
	\subsection{Proof of Corollary \ref{cor:monotonicity}}
		\begin{proof}
			By Lemma \ref{lem:expMassMonoton}, we have for $0<r<2$
			\begin{align*}
			\mT( \BB_r) \leq r^n \omm_n \exp \left( \Cl{0} \left( \Aa_\M + \kappa_T \right)(2-r) \right)  \frac{\mT(\BB_2)}{2^n \omm_n} 
			\leq 2^{-n} e^{4C} \MM(T) r^n
			\end{align*}
			and
			\begin{align*}
			\mT( \BB_r) \geq r^n \omm_n \lim_{s \downarrow 0} \left( \exp \left( \Cr{0} \left( \Aa_\M + \kappa_T \right)(s-r) \right)  \frac{\mT(\BB_s)}{s^n \omm_n} \right)
			\geq \omm_n e^{-4\Cr{0}} m r^n.
			\end{align*}
			Hence, there is a constant $\Cl{3}>0$ such that 
			\begin{equation}\label{eq:massOfBall}
			\frac{1}{\Cr{3}} r^n \leq \mT (\BB_r) \leq \Cr{3} r^n.
			\end{equation}
			Recall that $\Cr{23}$ is such that $|\overset{\rightarrow}{H}| \leq \Cr{23} \Aa_\M$. Then we use Lemma \ref{lem:monFormula} to estimate
			\begin{align*}
			&\left| \frac{ \lVert T \rVert (\BB_s)}{s^n} -\frac{ \lVert T \rVert (\BB_r)}{r^n} -\int_{\BB_s \backslash \BB_r} |X^\perp|^2 |X|^{-n-2} \dT \right|\\
			&\leq \int_r^s \rho^{-n-1} \left(\Cr{23} \rho \Aa_\M \mT(\BB_r) + \rho \omm_{n-1} \alpha \kappa_T \rho^{n} \right) \dd \rho\\
			&\leq \Cr{2} (\Aa_\M + \kappa_T)\left( s - r \right).
			\end{align*}	
		\end{proof}
	
	\subsection{Proof of Lemma \ref{lem:excLessX}}
		The proof of Lemma \ref{lem:excLessX} is based on the rather technical area comparison lemma: if we change slightly the $(n+1)$-component of a current, then its new mass stays close to its original mass.
		
		In the following, we will denote points in $\R^{n+k}$ by $(x,y)$, where $x \in \R^{n+1}$ and $y \in \R^{k-1}$. 
		
		\begin{lem} \label{lem:areacomp}
			Let $ 0< \tau <1$, $\rho >0$ and $A \subset \BC_1$ be a Borel set which is a cylinder (i.e. $A=\pp^{-1}(\pp(A))$). Let $\mu:\R^n \to [0,1]$ be a $\CC^1$-function satisfying $\sup_{\pp(A)} |D\mu| \leq \rho/\tau$ and consider the map
			$$F:\R^{n+k} \to \R^{n+k}, F(x,y)=\big(x_1, \dots, x_n, \mu(x_1,\dots,x_n)x_{n+1}, \Phii(x_1, \dots, x_n, \mu(x_1,\dots,x_n)x_{n+1})\big).$$ Then there is a constant $\Cl{7}>0$ only depending on $n$, $k$ and $m$ such that for any current $T$ with $(T, \M) \in \TT$ the following holds
			\begin{align*}
			\MM \big( F_\# (T \LL A) \big) - \MM (T \LL A) \leq \Cr{7} \left( \frac{1+\rho^2}{\tau^2}\int_A X_{n+1}^2 \dT 
			+\frac{\kappa_T^2}{\tau^2}
			+ \left(2 +\frac{\rho^2}{\tau^2} \right)\Aa \right),
			\end{align*}
			where $A_\tau:=\{x\in \R^{n+1}: \dist(x,A) < \tau\}$ is an enlargement of $A$ by $\tau$.
		\end{lem}

		\begin{proof}
			By \cite[Section 4.1.30]{federer}, we infer that for any $\omega \in \DD^n(\R^{n+1})$
			$$\big(F_\#(T\LL A) \big)(\omega)= \int_A \langle F_\# \overset{\rightarrow}{T}(x), \omega(F(x)) \rangle \dT.$$
			We expand the tangent vector in the following basis for $T_{(x, \Phii(x))}\M$
			\begin{align}\label{eq:basis}
			v_j(x) := (e_j, \partial_j \Phii(x)) \qquad \textnormal{ for } j \in \{1, \dots, n+1\},
			\end{align}
			where $e_j$ denotes the $j$-th standard basis vector in $\R^{n+1}$. Then there are real numbers $T_j$ such that
			\begin{align}\label{eq:expanT}
			\overset{\rightarrow}{T} = \sum_{j=1}^{n+1} T_j ~ v_1 \wedge \cdots \wedge \widehat{v_j} \wedge \cdots \wedge v_{n+1}.
			\end{align}
			We compute
			\begin{align*}
			F_\# \overset{\rightarrow}{T} (x,y)
			&= T_{n+1} ~ v_1(F(x)) \wedge \cdots \wedge v_n(F(x))\\
			&+ \sum_{j=1}^n \big( T_j \mu - T_{n+1} x_{n+1} \partial_j \mu \big) v_1(F(x)) \wedge \cdots \wedge \widehat{v_j(F(x))} \wedge \cdots \wedge v_{n+1}(F(x))
			\end{align*}
			and therefore, we have
			\begin{align*}
			|F_\# \overset{\rightarrow}{T}|^2
			&\leq \left( T_{n+1}^2 + \sum_{j=1}^n \big( T_j \mu - T_{n+1} X_{n+1} \partial_j \mu \big)^2 \right)
			\left( \sum_{j=1}^{n+1} \lvert v_1 \wedge \cdots \wedge \widehat{v_j} \wedge \cdots \wedge v_{n+1} \rvert^2 \right)\\
			&\leq \left( T_{n+1}^2 + \sum_{j=1}^n \big( T_j \mu - T_{n+1} X_{n+1} \partial_j \mu \big)^2 \right) \left(1+\C|D\Phii|^2 \right)\\
			&\leq T_{n+1}^2 + \sum_{j=1}^n \big( T_j \mu - T_{n+1} X_{n+1} \partial_j \mu \big)^2 + \Cl{8}|D\Phii|^2 \left(1 + \frac{\rho^2}{\tau^2} \right).
			\end{align*}
			We argue as in the original paper \cite[Lemma 3.1.1]{hardtSimon} to deduce
			\begin{equation}\label{eq:MassDifference}
			\begin{split}
			\MM \big( F_\# &(T \LL A) \big) - \MM (T \LL A)\\
			&\leq 2 \frac{\rho^2}{\tau^2}\int_A X_{n+1}^2 \dT + \int_A \big( 1-T_{n+1}^2 \big) \dT +  \Cr{8}\Aa^2 \left(1 + \frac{\rho^2}{\tau^2} \right) \MM(T).
			\end{split}
			\end{equation}

			In order to bound the second integral, we compute the first variation of $T$ with respect the following vectorfield
			$$ \Xi: \R^{n+k} \to \R^{n+k}, \quad (x,y) \mapsto \big(x_{n+1}-\psi_T(x_1, \dots, x_{n-1}) \big)\lambda^2(x)e_{n+1},$$
			where $e_{n+1}$ denotes the $(n+1)$-th basis vector of $\R^{n+k}$ and $\lambda: \R^{n+1} \to [0,1]$ is a $\CC^1$ cut-off function with $\spt(\lambda) \subset A_\tau$, $\lambda |_{{}_A} =1$ and $\sup|D\lambda| \leq \Cl{15}/\tau$. Notice that $\Xi$ vanishes on $\spt(\partial T)$ and therefore, by \cite[Theorem 3.2]{DeLellisBoundary}
			\begin{align}\label{eq:firstVariation}
			\int \Div_{\overset{\rightarrow}{T}} \Xi ~\dT = - \int \Xi \cdot \overset{\rightarrow}{H}_T(x) ~\dT (x),
			\end{align}
			where $\overset{\rightarrow}{H}_T$ is the mean curvature vector.
			
			As $\spt(T) \subset \M$, we have $\Div_{\overset{\rightarrow}{T}} \Xi = \Div_\M \Xi - \Div_\nu \Xi$ where $\nu \in T_{(x, \Phii(x))} \M$ is the outer (?) normal vector to $\overset{\rightarrow}{T}$. We  compute $\nu$ by expanding everything in the basis in \eqref{eq:basis}:
			\begin{align*}
			\nu &= \sum_{j=1}^{n+1} \nu_j v_j\\
			\overset{\rightarrow}{T}&= \tau_1 \wedge \cdots \wedge \tau_n \qquad \textnormal{with} \qquad \tau_i= \sum_{j=1}^{n+1} t_{i,j} v_j.
			\end{align*}
			As $\nu$ is normal to $\overset{\rightarrow}{T}$, we can use the expansion \eqref{eq:expanT} to find the following equalities for all $j \in \{1, \dots, n+1\}$ and $t_i:=(t_{i,1}, \dots t_{i, n+1})^\intercal$ with $i \in \{1,\dots, n\}$:
			\begin{align}
			T_j &= {\det}^{1, \dots, \hat{j}, \dots, n+1} \begin{pmatrix}
			t_1 & \cdots & t_n 
			\end{pmatrix}, \\
			0 &= \langle \nu, \tau_i \rangle = \langle \begin{pmatrix} \nu_1\\ \colon \\ \nu_{n+1}	\end{pmatrix} , g \cdot t_i \rangle, \label{eq:SPnuTau}
			\end{align}
			where $g = (\langle v_i, v_j \rangle_{i,j}) =\id_{n+1} + (\langle \partial_i \Phii, \partial_j \Phii \rangle_{i,j}) =: \id_{n+1} +B$ is the metric.\\
			From \eqref{eq:SPnuTau}, we deduce that
			\begin{align*}
			\nu_j= \star \big( (g \cdot t_i) \wedge \cdots \wedge (g \cdot t_n) \big)
			= (-1)^j {\det}^{1, \dots, \hat{j}, \dots, n+1} \begin{pmatrix}
			g \cdot t_1 & \cdots & g \cdot t_n 
			\end{pmatrix}.
			\end{align*} 
			We compute
			\begin{equation}\label{eq:divergenceNu}
			\begin{split}
			\Div_\nu \Xi &= \sum_{j=1}^{n+k} (D_\nu \Xi_j)_j = (D_\nu \Xi_{n+1})_{n+1}\\
			&= \left( \langle D \left( \big(x_{n+1}-\psi_T(x_1, \dots, x_{n-1}) \big)\lambda^2(x) \right), \frac{\nu}{|\nu|} \rangle \frac{\nu}{|\nu|} \right)_{n+1}\\
			&= \frac{1}{|\nu|^2} \left( \lambda^2 \nu_{n+1}^2- \lambda^2\sum_{j=1}^{n-1} \nu_{n+1}\nu_j \partial_j \psi_T +2\lambda(X_{n+1}- \psi_T)\sum_{j=1}^{n+1} \nu_{n+1}\nu_j \partial_j \lambda \right).
			\end{split}
			\end{equation}
			
			On the other hand, we need to compute the divergence with respect to $\M$. To do so, we compute the projection on $\M$: Let $M$ be the matrix with column vectors $v_1, \dots v_{n+1} \in \R^{n+k}$. Then we have
			\begin{align*}
			\Div_\M \Xi &= \sum_{j=1}^{n+k} (D_\M \Xi_j)_j = (D_\M \Xi_{n+1})_{n+1}\\
			&= \left( M \cdot g^{-1} \cdot M^T \cdot D \left( \big(x_{n+1}-\psi_T(x_1, \dots, x_{n-1}) \big)\lambda^2(x) \right) \right)_{n+1}\\
			&= \left( \begin{pmatrix}
			g^{1,1}	& \cdots  & g^{1,n+1}  \\ 
			\colon	&  & \colon \\ 
			g^{1,n+1}& \cdots &  g^{n+1,n+1}\\ 
			\star	& \star & \star
			\end{pmatrix} 
			\begin{pmatrix}
			1	&  &0  \\ 
			& \ddots &  \\ 
			0	&  & 1 \\ 
			\partial_1 \Phii	&  \cdots & \partial_{n+1} \Phii
			\end{pmatrix}^T
			\begin{pmatrix}
			\colon	\\ 
			2\lambda(X_{n+1}-\psi_T) \partial_i \lambda - \lambda^2 \partial_i \psi_T	\\ 
			\colon	\\ 
			2\lambda(X_{n+1}-\psi_T) \partial_n \lambda	\\ 
			2\lambda(X_{n+1}-\psi_T) \partial_{n+1} \lambda - \lambda^2	\\ 
			0
			\end{pmatrix} 
			\right)_{n+1}\\
			&= \left( \begin{pmatrix}
			g^{1,1}	& \cdots  & g^{1,n+1}  \\ 
			\colon	&  & \colon \\ 
			g^{1,n+1}& \cdots &  g^{n+1,n+1}\\ 
			\star	& \star & \star
			\end{pmatrix} 
			\begin{pmatrix}
			\colon	\\ 
			2\lambda(X_{n+1}-\psi_T) \partial_i \lambda - \lambda^2 \partial_i \psi_T	\\ 
			\colon	\\ 
			2\lambda(X_{n+1}-\psi_T) \partial_n \lambda	\\ 
			2\lambda(X_{n+1}-\psi_T) \partial_{n+1} \lambda - \lambda^2
			\end{pmatrix} 
			\right)_{n+1}\\
			&= \lambda^2 g^{n+1,n+1} - \lambda^2 \sum_{j=1}^{n-1} g^{n+1,j} \partial_j \psi_T +2\lambda(X_{n+1}-\psi_T) \sum_{j=1}^{n+1} g^{n+1,j} \partial_j \lambda.
			\end{align*}
			This together with \eqref{eq:divergenceNu} yields
			\begin{equation}\label{eq:divComputed}
			\begin{split}
			\Div_{\overset{\rightarrow}{T}} \Xi
			&= \lambda^2 \left(g^{n+1,n+1} -\frac{\nu_{n+1}^2}{|\nu|^2} \right)- \lambda^2 \sum_{j=1}^{n-1} \left( g^{n+1,j} -\frac{\nu_{n+1}\nu_j}{|\nu|^2} \right) \partial_j \psi_T\\
			&\quad +2\lambda(X_{n+1}-\psi_T) \sum_{j=1}^{n+1} \left( g^{n+1,j} -\frac{\nu_{n+1}\nu_j}{|\nu|^2} \right) \partial_j \lambda.
			\end{split}
			\end{equation}
			Together with \eqref{eq:firstVariation}, we have
			\begin{equation}\label{eq:equaFirstVar}
			\begin{split}
			- \int \Xi \cdot \overset{\rightarrow}{H}_T ~\dT 
			&= \int \lambda^2 \left( \left(g^{n+1,n+1} -\frac{\nu_{n+1}^2}{|\nu|^2} \right)- \sum_{j=1}^{n-1} \left( g^{n+1,j} -\frac{\nu_{n+1}\nu_j}{|\nu|^2} \right) \partial_j \psi_T \right) ~\dT \\
			&\quad +2 \int \lambda(X_{n+1}-\psi_T) \sum_{j=1}^{n+1} \left( g^{n+1,j} -\frac{\nu_{n+1}\nu_j}{|\nu|^2} \right) \partial_j \lambda ~\dT.
			\end{split}
			\end{equation}
			In order to regain the term $1-T_{n+1}^2$, we first estimate $\nu_{n+1}$:
			\begin{align*}
			(-1)^{n+1} \nu_{n+1}&=  {\det}^{1, \dots, n} \begin{pmatrix}	g \cdot t_1& \cdots & g \cdot t_n \end{pmatrix}\\
			&= {\det}^{1, \dots, n} \left( \big(\id +(\langle \partial_i \Phii, \partial_j \Phii \rangle_{i,j}) \big) \cdot \begin{pmatrix} t_1& \cdots &  t_n \end{pmatrix} \right)\\
			&= \sum_{\sigma \in S_n} \sgn(\sigma) \left( t_{1, \sigma(1)} + \sum_{j=1}^{n+1} t_{1,j} \langle \partial_{\sigma(1)} \Phii, \partial_j \Phii \rangle \right)
			\cdots  \left( t_{n, \sigma(n)} + \sum_{j=1}^{n+1} t_{n,j} \langle \partial_{\sigma(n)} \Phii, \partial_j \Phii \rangle \right)\\
			&= \sum_{\sigma \in S_n} \sgn(\sigma)  t_{1, \sigma(1)} 
			\cdots t_{n, \sigma(n)}  + O(|D\Phii|)\\
			&=T_{n+1}^2+ O(|D\Phii|).
			\end{align*}
			Hence,
			\begin{align}\label{eq:normalVector}
			\nu_{n+1}^2 \leq T_{n+1}^2 + \C|D\Phii|^2.
			\end{align}
			
			Now, we compute the norm of $\nu$. We use that the Hodge star is norm-preserving and therefore, we have for $\tilde{\nu}:= (\nu_1, \dots, \nu_{n+1})$
			\begin{align*}
			|\tilde{\nu}|^2 = \det \big(\langle g\cdot t_i, g \cdot t_j \rangle_{i,j}\big)
			= \det \big(\langle t_i, g^2 t_j \rangle_{i,j} \big)
			= \det \Big( \big( \langle t_i, t_j \rangle +2 \langle t_i, Bt_j \rangle + \langle t_i, B^2 t_j \rangle\big)_{i,j} \Big).
			\end{align*}
			Notice that
			\begin{equation}
			\begin{split}
			\langle t_i, t_j \rangle +2 \langle t_i, Bt_j \rangle + \langle t_i, B^2 t_j \rangle
			&\geq \langle t_i, t_j \rangle -2 \|B \|_{op} |t_i||t_j| - \| B \|^2_{op}|t_i||t_j|\\
			&\geq \langle t_i, t_j \rangle -\left(2 \|B \| + \| B \|^2 \right) |t_i||t_j|\\
			&\geq \langle t_i, t_j \rangle -\left(2 \sqrt{n+1}|D\Phii|^2 + (n+1)|D\Phii|^4 \right) |t_i||t_j|\\
			&\geq \langle t_i, t_j \rangle -2(n+1)|D\Phii|^2,
			\end{split}
			\end{equation}
			where we used in the last inequality the fact
			$$ |t_i|^2 = \Bigl\lvert \sum_{j=1}^{n+1} t_{i,j} (e_j, \partial_j \Phii) \Bigr\rvert^2- 
			\Bigl\lvert \sum_{j=1}^{n+1} t_{i,j} \partial_j \Phii \Bigr\rvert^2
			\leq |\tau_i|^2=1.$$
			Therefore, we estimate
			\begin{equation}\label{eq:detExpansion}
			\begin{split}
			|\tilde{\nu}|^2 
			&= \sum_{\sigma \in P_n} \prod_{i=1}^n \sgn( \sigma)\langle g\cdot t_i, g \cdot t_{\sigma(i)} \rangle\\
			&\geq \sum_{\sigma \in P_n} \left( \prod_{i=1}^n \sgn( \sigma)\langle t_i, t_{\sigma(i)} \rangle - 2^n(2(n+1))^n|D\Phii|^2 \right)\\
			&\geq \det\big(\langle t_i, t_j \rangle_{i,j} \big) -2^{2n} n! (n+1)^n |D\Phii|^2.
			\end{split}
			\end{equation}
			Now, we use that $\tau_1, \dots \tau_n$ are orthonormal to deduce that
			\begin{align*}
			\delta_{i,j} &= \langle \tau_i, \tau_j \rangle = \langle \sum_{k=1}^{n+1} t_{i,k} (e_k, \partial_k \Phii), \sum_{l=1}^{n+1} t_{i,l} (e_l, \partial_l \Phii) \rangle\\
			&= \langle \sum_{k=1}^{n+1} t_{i,k} e_k, \sum_{l=1}^{n+1} t_{i,l} e_l \rangle 
			+ \langle \sum_{k=1}^{n+1} t_{i,k} \partial_k \Phii, \sum_{l=1}^{n+1} t_{i,l} \partial_l \Phii \rangle\\
			&= \langle t_i, t_j \rangle + \sum_{k,l=1}^{n+1} t_{i,k}t_{j,l} \langle \partial_k \Phii, \partial_l \Phii \rangle
			\end{align*}
			and hence,
			$$ | \delta_{i,j} - \langle t_i, t_j \rangle | \leq 2(n+1) |D \Phii|^2.$$
			By a similar argument as in \eqref{eq:detExpansion}, it follows that
			$$ \det\big(\langle t_i, t_j \rangle_{i,j} \big) \geq 1-2^n n! (n+1)^n|D \Phii|^2.$$
			Putting this into \eqref{eq:detExpansion}, we yield
			\begin{align*}
			|\nu|^2 = \Bigl\lvert \sum_{j=1}^{n+1} \nu_j v_j \Bigr\rvert^2
			= \nu_1^2 + \cdots +\nu_{n+1}^2 + \Bigl\lvert \sum_{j=1}^{n+1} \nu_j \partial_j \Phii \Bigr\rvert^2
			\geq |\tilde{\nu}|^2 \geq 1- 2^{2n+1}n! (n+1)^n |D\Phii|^2.
			\end{align*}
			Therefore,
			\begin{align}\label{eq:normNu}
			\frac{1}{|\nu|^2} \leq \frac{1}{1- 2^{2n+1}n! (n+1)^n |D\Phii|^2}
			\leq 1 + \Cl{9}|D\Phii|^2.
			\end{align}
			Now, we take care of $g^{-1}$.
			By the geometric series and the fact $g= \id + (\langle \partial_i \Phii, \partial_j \Phii \rangle_{i,j})$, we have
			\begin{align}\label{eq:inverseMetric}
			g^{-1} = \id - (\langle \partial_i \Phii, \partial_j \Phii \rangle_{i,j}) + \sum_{l \geq 2} (-1)^l (\langle \partial_i \Phii, \partial_j \Phii \rangle_{i,j})^l
			\end{align}
			and hence,
			\begin{align}\label{eq:normG}
			|g^{i,j} | \leq \delta_{i,j}- \langle \partial_i \Phii, \partial_j \Phii \rangle + \Cr{10}|D\Phii|^4.
			\end{align}
			Now, we are ready to estimate piece by piece the right hand side of \eqref{eq:equaFirstVar}
			\begin{itemize}
				\item We use \eqref{eq:normalVector}, \eqref{eq:normNu} and \eqref{eq:normG} to deduce
				\begin{align*}
				\int \lambda^2 \left( g^{n+^, n+1}- \frac{\nu_{n+1}^2}{|\nu|^2} \right) \dT
				&\geq \int \lambda^2 \left( 1-|\partial_{n+1} \Phii|^2 - \Cr{10}|D\Phii|^4- T_{n+1}^2 - \C|D\Phii|^2 \right) \dT\\
				&\geq \int \lambda^2 \left( 1- T_{n+1}^2  \right) \dT - \Cl{11}\MM(T)\Aa^2.
				\end{align*}
				\item We use \eqref{eq:normalVector}, \eqref{eq:normNu} and \eqref{eq:normG} to deduce
				\begin{align*}
				\int \lambda^2 \sum_{j=1}^{n-1} &\left( g^{n+1,j} - \frac{\nu_{n+1}\nu_j}{|\nu|^2} \right) \partial_j \psi_T ~\dT\\
				&\leq \int \lambda^2 \left( \C|D\Phii|^2 + \frac{|(\nu_1, \dots, \nu_n)|}{|\nu|} \right) \kappa_T ~ \dT\\
				&\leq \kappa_T \int \lambda^2 \frac{\sqrt{|\tilde{\nu}|^2-\nu_{n+1}^2}}{|\nu|^2}  \dT +\Cl{12}\MM(T)|D\Phii|^2\\
				&\leq \kappa_T \int \lambda^2 \sqrt{ 1-T_{n+1}^2 + \C|D\Phii|^2 }\left(1+ \Cr{9}|D\Phii|^2 \right)  \dT +\Cr{12}\MM(T)\Aa^2\\
				&\leq \kappa_T \int \lambda^2 \sqrt{ 1-T_{n+1}^2 + \Cl{14}|D\Phii|^2 } \dT +\Cr{12}\MM(T)\Aa^2.
				\end{align*}
				\item We use \eqref{eq:normalVector}, \eqref{eq:normNu}, \eqref{eq:normG} and a similar argument as in \eqref{eq:detExpansion} to deduce
				\begin{align*}
				&\int 2\lambda (X_{n+1}-\psi_T) \sum_{j=1}^{n+1} \left( g^{n+1,j} - \frac{\nu_{n+1}\nu_j}{|\nu|^2} \right) \partial_j \lambda ~\dT\\
				&= \int 2\lambda(|X_{n+1}| +\kappa_T) \left( \left( g^{n+1, n+1} - \frac{\nu_{n+1}^2}{|\nu|^2} \right)\partial_{n+1} \lambda + \sum_{j=1}^{n} \left(g^{n+1,j} - \frac{\nu_{n+1}\nu_j}{|\nu|^2} \right) \partial_j \lambda \right) \dT\\
				&\leq 2 \int \lambda |D\lambda| (|X_{n+1}| +\kappa_T) \left( 1-|\partial_{n+1} \Phii|^2 -T_{n+1}^2 
				+ \C|D\Phii|^2 + \frac{|(\nu_1, \dots, \nu_n)|}{|\nu|} \right) \dT\\
				&\leq 2 \frac{\Cr{15}}{\tau} \left( \int \lambda (|X_{n+1}| +\kappa_T) \left( 1-T_{n+1}^2 + \frac{\sqrt{|\tilde{\nu}|^2-\nu_{n+1}^2}}{|\nu|^2} \right) \dT
				+ \Cl{13}\MM(T)\Aa^2 \right)\\
				&\leq 2 \frac{\Cr{15}}{\tau} \left( \int \lambda (|X_{n+1}| +\kappa_T) \left( 1-T_{n+1}^2 +\sqrt{ 1-T_{n+1}^2 + \C|D\Phii|^2 }\left(1+ \Cr{9}|D\Phii|^2 \right) \right) \dT \right.\\
				&\ \hspace*{13cm}+ \Cr{13}\MM(T)\Aa^2 \Big)\\
				&\leq 2 \frac{\Cr{15}}{\tau} \left( \int \lambda  (|X_{n+1}| +\kappa_T) 2 \sqrt{ 1-T_{n+1}^2 + \Cl{16}|D\Phii|^2 } \dT
				+ \C\MM(T)\Aa^2 \right).
				\end{align*}
			\end{itemize}
			Putting all this into \eqref{eq:firstVariation} yields
			\begin{equation}\label{eq:firstVariation2}
			\begin{split}
			\int &\lambda^2 (1-T_{n+1}^2) \dT \\
			&\leq \int \kappa_T \lambda^2 \sqrt{ 1-T_{n+1}^2 + \Cr{14}|D\Phii|^2 } \dT 
			+ \frac{\Cr{15}}{\tau} \int \lambda|X_{n+1}| \sqrt{ 1-T_{n+1}^2 + \Cr{16}|D\Phii|^2 } \dT\\
			&\quad + \frac{\Cr{15}}{\tau} \int \kappa_T\lambda \sqrt{ 1-T_{n+1}^2  + \Cr{16}|D\Phii|^2} \dT + \int \Xi \cdot \overset{\rightarrow}{H} \dT + \Cl{17}\MM(T)\Aa^2.
			\end{split}
			\end{equation}
			Using three times the Cauchy inequality ($2ab \leq a^2 +b^2$), we estimate
			\begin{itemize}
				\item $ \displaystyle \int \kappa_T \lambda^2 \sqrt{ 1-T_{n+1}^2 + \Cr{14}|D\Phii|^2 } \dT $
				$$ \ \leq  \int_{A_\tau} \frac{\lambda^2}{4} \left(1-T_{n+1}^2 + \Cr{14}|D\Phii|^2 \right) \dT + \int_{A_\tau} \kappa_T^2\lambda^2 \ \dT, $$
				\item $ \displaystyle \frac{\Cr{15}}{\tau} \int \lambda|X_{n+1}| \sqrt{ 1-T_{n+1}^2 + \Cr{16}|D\Phii|^2 } \dT$
				$$\ \leq \frac{1}{16} \int_{A_\tau} \lambda^2 \left(1-T_{n+1}^2 + \Cr{16}|D\Phii|^2 \right) \dT + \frac{\Cr{15}}{\tau^2}\int_{A_\tau} X_{n+1}^2 ~\dT, $$
				\item $ \displaystyle \frac{\Cr{15}}{\tau} \int \kappa_T\lambda \sqrt{ 1-T_{n+1}^2  + \Cr{16}|D\Phii|^2} \dT $
				$$\ \leq \frac{1}{16} \int_{A_\tau} \lambda^2 \left(1-T_{n+1}^2 + \Cr{16}|D\Phii|^2 \right) \dT + \frac{\Cr{15}}{\tau^2} \int_{A_\tau} \kappa_T^2 \dT.$$
			\end{itemize}
			Again putting this into \eqref{eq:firstVariation2} yields
			\begin{align*}
			\int_{A_\tau} \lambda^2 \big(1-T_{n+1}^2\big) \dT 
			&\leq \frac{1}{2} \int_{A_\tau} \lambda^2 \left(1-T_{n+1}^2 \right) \dT  +\frac{\Cr{15}}{\tau^2}\int_{A_\tau} X_{n+1}^2 ~\dT\\
			&\quad +\int_{A_\tau} \Xi \cdot \overset{\rightarrow}{H} \dT
			+ \C\MM(T)\left( \Aa^2 +\kappa_T^2 + \frac{\kappa_T^2}{\tau^2} \right)
			\end{align*}
			and hence,
			\begin{align*}
			\int_{A} \big(1-T_{n+1}^2 \big) \dT &\leq \int_{A_\tau} \lambda^2 \big(1-T_{n+1}^2 \big) \dT\\
			&\leq 2\frac{\Cr{15}}{\tau^2}\int_{A_\tau} X_{n+1}^2 ~\dT + \Cl{17} \MM(T) \left(\kappa_T^2+ \frac{\kappa_T^2}{\tau^2} + 2\Aa \right).
			\end{align*}
			Using \eqref{eq:MassDifference}, we deduce the desired inequality
			\begin{equation*}
			\begin{split}
			\MM \big( F_\# &(T \LL A) \big) - \MM (T \LL A)\\
			&\leq \C \frac{1+\rho^2}{\tau^2}\int_A X_{n+1}^2 \dT 
			+\Cr{17} \MM(T) \frac{\kappa_T^2}{\tau^2}
			+ \C\MM(T) \Aa \left(2 + \frac{\rho^2}{\tau^2} \right).
			\end{split}
			\end{equation*}

		\end{proof}
			
		Now we have all the tools to estimate the excess of $T$ with its height.
		
		\begin{proof}[Proof of Lemma \ref{lem:excLessX}]
			The second inequality holds true with $ \Cr{19}\geq 3^n(1+ m\omm_n) \geq \MM(T)$.\\
			For the first inequality, we want to use Lemma \ref{lem:areacomp} for $A:=\BC_{1+\tau} \backslash \BC_1$, $\rho=3$ and $\tau=\sigma/2$. Consider $F$ as in the lemma for some $\CC^1$-function $\mu:\R^n \to [0,1]$ satisfying $\sup\limits_{\pp(A)} |D\mu| \leq \rho/\tau$ and
			\begin{align*}
			\begin{cases}
			\mu(z)=0 &\text{if } |z|\leq 1\\
			\mu(z)>0 &\text{if } 1 < |z| < 1+\tau\\
			\mu(z)=1 &\text{if } |z| \geq1 + \tau.
			\end{cases}
			\end{align*}
			Moreover, we define for $t\in \R$ and $(x,y)\in \R^{n+k}$ the homotopy
			$$H_t(x,y):=\left(\pp(x), ((1-t)\mu\circ \pp(x) +t)x_{n+1}, \Phii \big(\pp(x), ((1-t)\mu\circ \pp(x) +t)x_{n+1} \big) \right).$$
			Notice that $F$ is the identity on $\M \setminus \BC_{1+ \tau}$ and $F=(\pp,0, \Phii(\pp, 0))$ on $\BC_1$.\\
			Then for $R_T:= H_\#([0,1] \times \partial T)$ we have $\spt(R_T) \subset \M $ and
			$$ \hspace*{-2.3cm} \partial \big( T\LL \BC_{1+\tau} - F_\#(T\LL \BC_{1+\tau}) -R_T \big) = \partial (T-F_\#T -R_T)=0. $$
			Hence, by the area minimality of $T$ in $\M$, we have
			$$\MM(T \LL \BC_{1+\tau}) \leq \MM \big( F_\#(T \LL \BC_{1+\tau}) \big) + \MM(R_T).$$
			Moreover, by \cite[Remark 26.21(2)]{simon}, the following holds
			$$\MM(R_T) \leq \sup_{\spt(\partial T)}|\partial_t H| \sup_{\spt(\partial T)} |\partial_x H| \MM\big((\partial T)\LL \BC_2 \big).$$
			Therefore, we compute
			\begin{align*}
			|\partial_t H|² &\leq \left(X_{n+1}- X_{n+1}\mu \circ \pp \right)² + |D\Phii|²\left(X_{n+1}- X_{n+1}\mu \circ \pp \right)²\\
			&\leq \left(1+|D\Phii|^2 \right)|X_{n+1}|²\left(1-\mu\circ \pp \right)²\\
			&\leq \kappa_T²\left(1+|D\Phii|^2 \right),
			\end{align*}
			\begin{align*}
			|\partial_x H|² &\leq n + |D\mu|²X_{n+1}² + |D\Phii|²\left(n+|D\mu|X_{n+1} \right)² + (|\mu|+1)² + |D\Phii|²(|\mu|+1)²\\
			&\leq n +\left( \frac{6}{\sigma} \right)² \kappa_T² + |D\Phii|²\left(4+\left(n+\frac{6\kappa_T}{\sigma} \right)² \right) +4\\
			&\leq \C \left( 1+\frac{\kappa_T}{\sigma} \right)²,
			\end{align*}
			\begin{align*}
			\MM\big((\partial T)\LL \BC_2 \big) \leq \omm_{n-1} 2^{n-1} \sqrt{n+\kappa_T^2 +\Aa^2(1+ \kappa_T²)}
			\leq \C (1+\kappa_T).
			\end{align*}
			Thus, we have
			$$\MM(R_T) \leq \Cl{20} \frac{\kappa_T}{\sigma}(1+\Aa).$$
			Now, we argue as originally in \cite{hardtSimon} and use Lemma \ref{lem:areacomp} to deduce
			\begin{align*}
			\EE_C(T,1) 
			&\leq \MM\big( F_\#(T\LL A) \big) - \MM(T\LL A) + \Cr{20}  \frac{\kappa_T}{\sigma}(1+\Aa)\\
			&\ \leq \frac{\Cr{18}}{\sigma^2} \left( \kappa_T + \int_{\BC_{1+\sigma}} X_{n+1}^2 \dT 
			+\Aa\right).
			\end{align*}	
		\end{proof}
	
	\subsection{Proof of Lemma \ref{lem:XLessExc}}
		\begin{proof}
			We call a function $f$ $T$-subharmonic if
			$$\int \langle D_{\overset{\rightarrow}{T}}f ,  D_{\overset{\rightarrow}{T}} \zeta \rangle \dT \leq 0
			\qquad \textnormal{ for all } \zeta \in \CC^1(\R^{n+k}; \R_{\geq0}) \textnormal{ with } \spt(\zeta) \cap \spt(\partial T) = \emptyset.$$
			The functions 
			$$ h_i: \R^{n+k} \to \R, \quad (x,y) \mapsto (-1)^i x_{n+1} + x_{n+1}^2, \qquad \textnormal{ for } i \in \{1,2\}$$
			are $T$-subharmonic as
			\begin{align*}
			\int \langle D_{\overset{\rightarrow}{T}}h_i ,  D_{\overset{\rightarrow}{T}} \zeta \rangle \dT
			&= \int \langle \pi \cdot Dh_i, \pi \cdot D\zeta \rangle \dT
			= \int \langle Dh_i, \pi \cdot D\zeta \rangle \dT\\
			&=\int \langle (-1)^i\ee_{n+1} +2 X_{n+1} \ee_{n+1}, \pi \cdot D\zeta \rangle \dT\\
			&= \int \Big( \Div_{\overset{\rightarrow}{T}} \left( \zeta \left((-1)^i+2 X_{n+1} \right) \ee_{n+1} \right)
			-2  \zeta \pi_{n+1, n+1} \Big) \dT\\
			&=  \int \left( -\zeta \left((-1)^i+2 X_{n+1} \right) \ee_{n+1} \cdot \overset{\rightarrow}{H}
			-2 \zeta  g^{n+1,n+1} \right) \dT,\\
			&\leq \int \zeta\left( 7\Cr{23} \left|D^2 \Phii \right| - 2 \left( 1 - |\partial_{n+1} \Phii|^2- \Cr{10}|D\Phii|^4 \right)
			\right) \dT	\\
			&\leq \int \zeta\left( 7\Cr{23} \left|D^2 \Phii \right| - 2 \left( 1 -(1+ \Cr{10}) |D \Phii|^2 \right)
			\right) \dT	\\
			&\leq 0,
			\end{align*}
			where $\pi(x)$ denotes the orthogonal projection to the tangent plane of $T$ at $x$ and we used \eqref{eq:normG}, \cite[Theorem 3.2]{DeLellisBoundary} and the fact $\big( \spt(\zeta \ee_{n+1}) \cap \spt(\partial T) \big) \subset \big( \spt(\zeta) \cap \spt(\partial T) \big) = \emptyset$.

			Consider the nonnegative, convex function 
			\begin{align*}
			f: \R \to \R,\ t \mapsto 
			\begin{cases} t-2\kappa_T,& \textnormal{if}\ t \geq 2\kappa_T\\
			-t-2\kappa_T,& \textnormal{if}\ t \leq -2\kappa_T\\
			0, &\textnormal{else}
			\end{cases}.
			\end{align*}
			Notice that $f((-1)^i X_{n+1} + X_{n+1}^2)$ vanishes on $\spt(\partial T)$. If $f$ were additionally smooth, than by \cite[Lemma 7.5(3)]{allard2} $f((-1)^i X_{n+1} + X_{n+1}^2)$ would be $T$-subharmonic. Therefore, we take a smooth nonnegative mollifier $\eta$ satisfying $\spt(\eta) \subset (-1,1)$ and $\int_\R \eta(x) dx=1$. Define $\eta_\epsilon (x):=\frac{1}{\epsilon} \eta(x/\epsilon)$ and $f_\epsilon := f \ast \eta_\epsilon$. $f_\epsilon$ is smooth, convex and converges uniformly to $f$ when $\epsilon \downarrow 0$. Therefore $f_\epsilon \circ ((-1)^i X_{n+1} + X_{n+1}^2)$ is $T$-subharmonic and by \cite[Theorem 7.5(6)]{allard2}
			\begin{equation}\label{eq:subharmonic}
			\begin{split}
			\sup_{\BC_{1-\sigma}\cap \spt(T)} &f\big((-1)^i X_{n+1} + X_{n+1}^2\big)^2\\
			&= \sup_{a\in \pp^{-1}(0)} \sup_{\tauu_a (\BB_{1-\sigma}) \cap \spt(T)}f\big((-1)^i X_{n+1} + X_{n+1}^2\big)^2\\
			&=\sup_{a\in \pp^{-1}(0)} \lim_{\epsilon \downarrow 0} \Big(\sup_{\tauu_a (\BB_{1-\sigma}) \cap \spt(T)}  f_\epsilon \circ \big((-1)^i X_{n+1} + X_{n+1}^2\big)\Big)^2\\
			&\leq \sup_{a\in \pp^{-1}(0)} \lim_{\epsilon \downarrow 0} \left( \frac{\Cl{25}}{\sigma^n} \int_{\tauu_a(\BB_{1-\sigma/2})} (f_\epsilon \circ \big((-1)^i X_{n+1} + X_{n+1}^2\big))^2 \dT \right)\\
			&\leq \frac{\Cr{25}}{\sigma^n} \int_{\BC_{1-\sigma/2}} f^2\big((-1)^i X_{n+1} + X_{n+1}^2\big) \dT.
			\end{split}
			\end{equation}
			We deduce further that in $\BB_2$ the following holds
			\begin{equation}\label{eq:xLessf}
			\begin{split}
			X_{n+1}^2- 40 \kappa_T &\leq \big(|X_{n+1}|+X_{n+1}^2 \big)^2- 40 \kappa_T\\
			&\leq \begin{cases}
			\big( X_{n+1} + X_{n+1}^2 \big)^2-20\kappa_T, &\text{ if } |X_{n+1} + X_{n+1}^2| \geq 2\kappa_T\\
			0, &\text{else}
			\end{cases}\\
			&\ +
			\begin{cases}
			\big( -X_{n+1} + X_{n+1}^2 \big)^2-20\kappa_T, &\text{ if } |X_{n+1} - X_{n+1}^2| \geq 2\kappa_T\\
			0, &\text{else}
			\end{cases}\\
			&\leq f^2\big( X_{n+1} + X_{n+1}^2 \big) + f^2\big( -X_{n+1} + X_{n+1}^2 \big)
			\end{split}
			\end{equation}			
			and
			\begin{equation}\label{eq:fLessX}
			\begin{split}
			f^2\big( X_{n+1} + X_{n+1}^2 \big) + f^2\big( -X_{n+1} + X_{n+1}^2 \big)
			&\leq 2 \left( \big( X_{n+1} + X_{n+1}^2 \big)^2 +\big(-X_{n+1} + X_{n+1}^2 \big)^2 + 8\kappa_T^2 \right)\\
			&\leq 4 \big( |X_{n+1}| + X_{n+1}^2 \big)^2 + 16\kappa_T^2\\
			&\leq 36\big(X_{n+1}^2+ \kappa_T^2 \big).
			\end{split}
			\end{equation}
			Putting \eqref{eq:subharmonic}, \eqref{eq:xLessf} and \eqref{eq:fLessX}, we conclude
			\begin{align*}
			\sup_{\BC_{1-\sigma}\cap \spt(T)} X_{n+1}^2 
			&\leq \frac{\Cr{25}}{\sigma^n} \int_{\BC_{1-\sigma/2}} \left( f^2\big(X_{n+1} + X_{n+1}^2\big) +f^2\big(-X_{n+1} + X_{n+1}^2\big) \right) \dT +40\kappa_T\\
			&\leq \frac{36 \Cr{25}}{\sigma^n} \int_{\BC_{1-\sigma/2}} \left( X_{n+1}^2 +\kappa_T^2 \right)\dT +40\kappa_T\\
			&\leq \frac{\Cr{21}}{\sigma^n} \left( \int_{\BC_{1-\sigma/2}} X_{n+1}^2 \dT + \kappa_T \right).
			\end{align*}
			For $(ii.)$, we specify $\Cl{270}$ later and let
			$$\tilde{C} :=12 \cdot 3^{3n+2}\left(7+2m +2\Cr{6}+ \Cr{270} \right) \Cr{21}\big(1+m\omm_n\big).$$
			\textit{Case 1: }
			$ \displaystyle \EE_C(T,1) + \kappa_T + \Aa \geq 3^{n+2}\big(1+m \omm_n \big) \frac{\sigma^{n+1}}{\tilde{C}}$.\\
			In this case, as $\spt(T) \subset \BB_3$, we can bound
			\begin{align*}
			\int X_{n+1}^2 \dT 
			\leq 3^{n+2}\big(1+m\omm_n \big) 
			\leq \frac{\tilde{C}}{\sigma^{n+1}} \big(\EE_C(T,1) + \kappa_T + \Aa \big).
			\end{align*}
			\textit{Case 2: }
			$ \displaystyle \EE_C(T,1) + \kappa_T + \Aa < 3^{n+2}\big(1+m \omm_n \big) \frac{\sigma^{n+1}}{\tilde{C}}
			\hfill \refstepcounter{equation}(\theequation)\label{eq:case2}$.\\
			Here, we aim to show that $\BC_{1-\sigma/2} \cap \spt(T) \subset \BB_1$. If this were true, the following would conclude the lemma. Namely, recall the normal vector $\nu$ from the proof of Lemma \ref{lem:areacomp}. Then, by Cauchy's inequality, we can deduce
			\begin{equation}\label{eq:firstPartX2}
			\begin{split}
			\int_{\BB_1} X_{n+1}^2 \dT 
			&= \int_{\BB_1} \big( \langle X, \frac{\nu}{|\nu|} \rangle + \langle X, \ee_{n+1}-­\frac{\nu}{|\nu|} \rangle \big)^2 \dT\\
			&\leq 2 \int_{\BB_1} \left( |X^\perp|^2 + |X|^2 \left|\ee_{n+1}-­\frac{\nu}{|\nu|} \right|^2 \right) \dT\\
			&\leq 2 \int_{\BB_1} \left( |X^\perp|^2 |X|^{-n-2} + \left\| \ee_{n+1} \cdot \ee_{n+1}^\top -­\frac{1}{|\nu|^2} \nu \cdot \nu^\top \right\|^2 \right) \dT\\
			\end{split}
			\end{equation}
			Now, we recall that the cylindrical excess can also be expressed by
			$$\frac{1}{r^n} \int_{\BC_r} \|\pi - \pp \|^2 \dT,$$
			where $\pi(x)$ still denotes the orthogonal projection to the tangent plane of $T$ at $x$
			We compute for $(x,y) \in \BB_1$
			\begin{align*}
			\left(\pi - \pp \right)(x,y)
			&= \left(  M \cdot g^{-1} \cdot M^T(x,y)^T - \langle (x,y), \frac{\nu}{|\nu|} \rangle \frac{\nu}{|\nu|} \right)- \sum_{j=1}^n x_j \ee_j\\
			&= B(x,y) + x_{n+1} \ee_{n+1} - \langle (x,y), \frac{\nu}{|\nu|} \rangle \frac{\nu}{|\nu|},
			\end{align*}
			where
			$$B(x,y) :=  M \cdot g^{-1} \cdot M^T (x,y)^T-(x,0)^T.$$
			Using \eqref{eq:inverseMetric} we estimate
			$$|B(x,y)| \leq \Cl{27} |D \Phi|.$$ 
			Hence, by Corollary \ref{cor:sphericalExc} and the inequality \eqref{eq:spherCylindExc}, we can continue the estimate of \eqref{eq:firstPartX2} in the following way:
			\begin{equation}\label{eq:integralX2}
			\begin{split}
			\int_{\BB_1} X_{n+1}^2 \dT 
			&\leq 2 \left( \EE_S(T,1) + \Cr{6}(\Aa + \kappa_T) + \int_{\BB_1} 2 \left( \|\pi - \pp \|^2 + \|B\|^2 \right) \dT \right)\\
			&\leq 2 \EE_S(T,1) + 2\Cr{6}(\Aa + \kappa_T) + 4\EE_C(T,1) + \Cr{27}\Aa^2\\
			&\leq \left(6+2m+2\Cr{6} \right)(\EE_C(T,1) + \kappa_T) + \left(2\Cr{6} + \Cr{270} \right)\Aa.
			\end{split}
			\end{equation}
			As $(6+2m+2\Cr{6} +\Cr{270}) \leq \tilde{C} \leq \tilde{C} \sigma^{-n-1}$, we are left with proving that $$\BC_{1-\sigma/2} \cap \spt(T) \subset \BB_1.$$
			First, we notice that due to a similar reasoning as we did for $(i.)$ and using \eqref{eq:integralX2}, we have
			\begin{align} \label{eq:sigma12}
			\sup_{\BB_{1-\sigma/6} \cap \spt(T)} X_{n+1}^2
			&\leq \frac{6^n}{\sigma^n} \Cr{21} \left( \int_{\BB_1} X_{n+1}^2 \dT + \kappa_T \right) \notag \\ \notag
			&\leq \frac{6^n \Cr{21}}{\sigma^n} \left( \left(7+2m+2\Cr{6} \right)(\EE_C(T,1) + \kappa_T) + \left(2\Cr{6} + \Cr{270} \right) \Aa \right)\\
			&\leq \frac{\sigma}{12}.
			\end{align}
			As a next step, we show that $\spt\big((\partial T) \LL \BC_{1-\sigma/3 } \big) \subset \UU_{1-\sigma/6}.\hfill \refstepcounter{equation}(\theequation)\label{eq:subsetdT}$\\
			We argue by continuity: Assume by contradiction that this is not the case. Then we would find a $ z\in \R^{n-1}$ such that $\big(z, \varphi_T(z), \psi_T(z), \Phii(z, \varphi_T(z), \psi_T(z)) \big) \in \BC_{1-\sigma/3 } \backslash \UU_{1-\sigma/6}$, hence, $|(z, \varphi_T(z))| < 1-\sigma/3$ but $\big|\big(z, \varphi_T(z), \psi_T(z), \Phii(z, \varphi_T(z), \psi_T(z)) \big) \big| \geq 1-\sigma/6$. Then it must hold that
			\begin{equation}\label{eq:phiT}
			\psi_T(z)^2 + \left| \Phii(z, \varphi_T(z), \psi_T(z)) \right|²
			\geq \left( 1-\frac{\sigma}{6} \right)^2 - \left( 1-\frac{\sigma}{3} \right)^2
			= \frac{\sigma}{3}-\frac{\sigma^2}{12} .
			\end{equation}
			Consider now for $t\in [0,1]$ the curve $\gamma(t):=\big(tz,\varphi_T(tz),\psi_T(tz), \Phii(tz, \varphi_T(tz), \psi_T(tz)) \big) \in \R^{n+k}$. As $\gamma(0)=0$ and $\gamma(1) \notin \UU_{1-\sigma/6}$, there is by the mean value Theorem a $t \in [0,1]$ such that $|\gamma(t)|=1-\sigma/6.$ Let $\tilde{s}:= \min\{t \in [0,1]: |\gamma(t)|=1-\sigma/6\} >0.$ Then for all $0<s<\tilde{s}$, we have $\gamma(s) \in \UU_{1-\sigma/6}$ and by \eqref{eq:sigma12}, $\psi_T(sz)^2 < \sigma/12$. But then we get by \eqref{eq:phiT}
			\begin{align*}
			|\gamma(\tilde{s})-\gamma(s)|
			&\geq \left|\psi_T(\tilde{s}z) -\psi_T(sz) \right|\\
			&\geq \sqrt{\frac{\sigma}{3}-\frac{\sigma^2}{12}-\left| \Phii(\tilde{s}z, \varphi_T(\tilde{s}z), \psi_T(\tilde{s}z)) \right|^2 }- \sqrt{\frac{\sigma}{12}}\\
			&\geq \sqrt{ \frac{\sigma}{4} - |D\Phii|^2 \left(1-\frac{\sigma}{3} \right)^2 }- \sqrt{\frac{\sigma}{12}}\\
			&\geq \frac{\sqrt{\sigma}}{24},
			\end{align*}
			where we used the assumption of the lemma in the last inequality. As $0<s<\tilde{s}$ was arbitrary, this contradicts the continuity of $\gamma$. Hence, \eqref{eq:subsetdT} holds true.
			
			And then $\spt(T) \LL \BC_{1-\sigma/2}$ stays in the unit ball: We denote by $p$ to projection to $\R^{n+1}$. Then as $T$ is minimizing in $\M$, $p_\#T$ is minimizing a parametric integrand described Lemma \ref{lem:minSurfEqu}. Then we can use \cite[Corollary 4.2]{hardt} to deduce that $\spt(p_\#T)$ is contained in the convex hull of $\spt(\partial(p_\# T))$. Hence, $\spt(p_\#T \LL \BC_{1-\sigma/2}) \subset \BB_{1- \sigma/6}$. Using the fact that
			$T= (\id, \Phii)_\# p_\# T$ and $|D \Phii| \leq \sigma /6$, we conclude that $\spt(T) \LL \BC_{1-\sigma/2} \subset \BB_1$.
		\end{proof}
	\subsection{Proof of Remark \ref{rmk:scale3}}
		\begin{proof}
			\begin{enumerate}[(i.)]
				\item we choose $\sigma= 1/4$ in Lemma \ref{lem:XLessExc} and get that
				$$ \sup_{\BC_{3/4} \cap \spt(T)} X_{n+1}^2 
				\leq 4^{2n+1} \Cr{21} \Cr{22} \big( \EE_C(T,1) +\kappa_T+\Aa \big)
				\leq \left( \frac{1}{8} \right)^2.$$
				\item We first check, whether we created additional boundary while taking the intersection with $\UU_3$. If this were the case, then for $|\omega| \leq \frac{1}{8}$, there is a point 
				$$(u, v) \in \{ x\in \gamm_\omega(\M): |(x_1,\dots,x_{n-1})| \textstyle \leq \frac{1}{2}, |x_n| < \frac{1}{2} \} 
				\cap \gamm_\omega \left( X_{n+1}^{-1} \left( \left[-\frac{1}{8}, \frac{1}{8} \right]\right) \cap \partial \UU_{3/4} \cap \M \right)$$
				with 
				\begin{itemize}
					\item	$u= \big(x_1,\dots, x_{n-1}, x_n\cos(\omega)-x_{n+1}\sin(\omega), x_n\sin(\omega)+x_{n+1}\cos(\omega) \big)$
					\item $v= \Phii \big(x_1,\dots, x_{n-1}, x_n\cos(\omega)-x_{n+1}\sin(\omega), x_n\sin(\omega)+x_{n+1}\cos(\omega) \big)$
					\item $ \displaystyle |x_{n+1}| \leq \frac{1}{8}$
					\item $ \displaystyle x_1^2 + \cdots + x_{n+1}^2+ |\Phii(x_1, \dots x_{n+1})|^2 = \frac{9}{16}$
					\item $ \displaystyle x_1^2+\cdots + x_{n-1}^2 \leq \frac{1}{4}$
					\item $\displaystyle |x_n\cos(\omega)-x_{n+1}\sin(\omega)| < \frac{1}{2}$.
				\end{itemize}
				This implies that $x_n^2 \geq \frac{19}{64}- |\Phii(x_1, \dots, x_{n+1})|^2 \geq \frac{9}{32}$ and hence,
				\begin{align*}
				\frac{1}{2} &> |x_n\cos(\omega)-x_{n+1}\sin(\omega)|\\
				&\geq \sqrt{\frac{9}{32}} \cos(\omega) + \frac{1}{8}(\cos(\omega)-\sin(\omega))\\
				&\geq \frac{\sqrt{19}-1}{8} \cos \left(\frac{1}{8} \right) + \frac{1}{8}\left(\cos \left(\frac{1}{8} \right)-\sin \left(\frac{1}{8} \right) \right)\\
				& > \frac{1}{2}.
				\end{align*}
				Hence, there is no such $x$ and the intersection is trivial, thus we have
				$$\partial \big( (\muu_{4\#} \gamm_{\omega\#} T) \LL \UU_3 \big) =
				\big( \partial (\muu_{4\#} \gamm_{\omega\#} T) \big) \LL \UU_3.$$
				The remaining conditions for $(\muu_{4\#} \gamm_{\omega\#} T) \LL \UU_3$ to belong to $\TT$ follow like in the original paper \cite{hardtSimon}.
				\item We write $(\muu_{r\#} \gamm_{\omega\#} T) \LL \UU_3 = (\muu_{r/4\#} \muu_{4\#} \gamm_{\omega\#} T) \LL \UU_3$ in order to use Remark \ref{rmk:scale3}. As in the original paper \cite{hardtSimon}, we deduce
				\begin{equation*}
				\sup \Big\{ x_{n+1}^2: x \in \spt\big((\gamm_{\omega\#}T) \LL \BC_{1/2}\big) \Big\} 
				\leq 4 \bigg(\omega^2 + \sup_{\BC_{3/4} \cap \spt(T)} |X_{n+1}| \bigg).
				\end{equation*}
				Hence, by using Lemma \ref{lem:excLessX} (with $\sigma \uparrow 0$ and Lemma \ref{lem:XLessExc}, we have 
				\begin{align}
				\EE_C\big( (\muu_{4\#} \gamm_{\omega\#} T) &\LL \UU_3 ,1 \big) 
				\leq \Cr{18} \left( \Cr{19} \sup_{\BC_2 \cap \spt(\muu_{4\#} \gamm_{\omega\#} T)} X_{n+1}^2 + \frac{\kappa_T+ \Aa}{4} \right)\notag \\
				&\leq \Cr{18} \left( 16 \Cr{19} \sup_{\BC_{1/2} \cap \spt(\gamm_{\omega\#} T)} X_{n+1}^2  +\kappa_T+ \Aa  \right)\notag\\
				&\leq \Cr{18} \left( 4^3 \Cr{19} \omega^2+ 4^3 \Cr{19} \sup_{\BC_{3/4} \cap \spt(T)} X_{n+1}^2  + \kappa_T + \Aa \right)\\
				&\leq \frac{\Cr{29}}{\Cr{28}} \left(\omega^2 + \EE_C(T,1) + \kappa_T  + \Aa \right) \label{eq:excRotSmall}\\
				&\leq \frac{1}{\Cr{28}}.\notag
				\end{align}
				Thus, we can use Remark \ref{rmk:scale3} and conclude.
			\end{enumerate}
		\end{proof}
	
	\subsection{Proof of Lemma \ref{lem:sphercylindr}}
		\begin{proof}
			The plan to prove this lemma is as follows: First, we bound the excess with $\int X_{n+1}^2 \dT$ by Lemma \ref{lem:excLessX}. Then, we construct a vectorfield and compute the associated first variation. By minimality of $T$ this can be expressed by the mean curvature vector. Moreover, by the choice of the vectorfield, we can bound $\int X_{n+1}^2 \dT$ with $\int |X^\perp |^2|X|^{-2} \dT$. By Corollary \ref{cor:sphericalExc} this carries over to the spherical excess.\\
			Let $T$ be as in the lemma and $\Cr{18}$ as in Lemma \ref{lem:excLessX}. Moreover, we define
			\begin{align*}
			\Cr{43}&= 2^{2n+2}\Cr{21}\Cr{22},\\
			\Cr{46}&=3^{2n+8}\Cr{18} (1+m\omm_n).
			\end{align*}
			We apply Lemma \ref{lem:XLessExc} with $\sigma=1/2$ to deduce
			$$\sup_{\BC_{1/2}\cap \spt(T)} X_{n+1}^2 
			\leq 2^{2n+1}\Cr{21}\Cr{22}\big( \EE_C(T,1)+\kappa_T + \Aa \big) \leq \frac{1}{2}.$$
			Hence, for all $x=(\tilde{x}, \tilde{y}) \in \BC_{1/2}\cap \spt(T)$ the following holds
			\begin{align}\label{eq:sptTU1}
			|x|^2 \leq (1+|D\Phii(\tilde{x})|^2)( |\pp(x)|^2 + x_{n+1}^2) \leq \frac{4}{3} \left(\frac{1}{4} + \frac{1}{2} \right) = 1.
			\end{align}
			For $x=(\tilde{x},\tilde{y}) \in \R^{n+k}$ the projection to the tangent space of $\M$ at $(\tilde{x}, \Phii(\tilde{x}))$ is given by 
			\begin{align*}
			P= P_{\tilde{x}} := M g^{-1} M^T 
			=	\begin{pmatrix}
			\id \\
			D\Phii
			\end{pmatrix}
			g^{-1}
			\begin{pmatrix}
			\id & D\Phii
			\end{pmatrix}
			=\begin{pmatrix}
			g^{-1} &  g^{-1} \cdot D\Phii\\
			(g^{-1} \cdot D\Phii)^T & D\Phii^T \cdot  g^{-1}  \cdot D\Phii
			\end{pmatrix}.
			\end{align*}
			Therefore
			\begin{align}\label{eq:traceEstimate}
			\textnormal{tr}_{n+1}(P):= \sum_{i=1}^{n+1} P_{ii} \leq n+1 + \C |D\Phii|^2
			\end{align}
			and
			\begin{align}\label{eq:ProjectionEstimate}
			\left|(P-\id) \begin{pmatrix} \tilde{x} \\ 0 \end{pmatrix} \right|
			= \left| \begin{pmatrix} g^{-1}\tilde{x}-\tilde{x} \\ D\Phii (g^{-1}\tilde{x}) \end{pmatrix} \right|
			\leq \Cl{42}|D\Phii(\tilde{x})|,
			\end{align}
			where we used \eqref{eq:inverseMetric}.
			
			Denote by $\nuu$ the outer unit normal vector being tangent to $\M$ and normal to the approximate tangent space of $T$. As $\nuu =(\nuu_1, \dots, \nuu_{n+k}) \in \text{span}\{(\ee_i, \partial_i \Phii): i \leq n+1 \}$, we have
			$$ \nuu_{n+1+j} = \sum_{i=1}^{n+1} \nuu_i \partial_i \Phii^j \qquad \text{ for all } j \leq k-1.$$
			Denote by $\tilde{\nuu}= (\nuu_1, \dots, \nuu_{k+1})$. Then the following holds
			\begin{align}\label{eq:normNormal}
			| \nuu | \leq (1+|D\Phii|) \left| \tilde{\nuu} \right|.
			\end{align}
			Moreover, define $A:=\UU_1 \setminus B_{1/4}$ where $B_{1/4} =\BB^{n+1}_{1/4} \times \R^{k-1}$. Denote $\kappa:=\kappa_T$, $\varepsilon:= \sqrt{\EE_C(T,1/3)}$, $\beta:= 4\Cr{46}^{-1/2}$ and for all $x \in \R^{n+k}$ let
			$$\lambda(x):= \max \left\{0, \frac{x_{n+1}}{|\tilde{x}|}- \beta \varepsilon -\kappa\right\}.$$
			Then in $A$ we have
			\begin{equation}\label{eq:derivLambda}
			\begin{split}
			\left| \langle (\tilde{X},0), D_{\overset{\to}{T}} \lambda \rangle \right| &\leq
			\left| \langle (\tilde{X},0), D_{\overset{\to}{T}} \left( \frac{X_{n+1}}{|\tilde{X}|} \right) \rangle \right|\\
			&=\left| \langle (\tilde{X},0), \left( (P - \nuu \otimes \nuu) \left( \frac{\ee_{n+1}}{|\tilde{X}|} - \frac{X_{n+1}}{|\tilde{X}|^3}(\tilde{X},0) \right) \right) \rangle \right|\\
			&\leq \left| \frac{\nuu_{n+1}}{|\tilde{X}|} \langle \tilde{X}, \tilde{\nuu} \rangle -  \frac{X_{n+1}}{|\tilde{X}|^3} \langle \tilde{X}, \tilde{\nuu} \rangle^2 \right| + 8\Cr{42} |D\Phii|\\
			&\leq  2\left| \frac{\langle \tilde{X}, \tilde{\nuu} \rangle}{|\tilde{X}|} \right| + 8\Cr{42}|D\Phii|.
			\end{split}
			\end{equation}
			Let $k \in \N$ with $k\geq1$ and choose a $\CC^1$ function $\mu_k:\R \to \R$ such that for $t\geq 1/4$ we have
			$$\mu_k(t) = \max\{0, t^{-n}-1\}^{1+1/k}.$$
			Moreover, let $h_k: \R^{n+k} \to \R^{n+k}$ be a $\CC^1$ vectorfield satisfying $ \left. h_k \right|_{B_{1/4} \cap \spt(T)} \equiv 0$ and 
			$$ h_k(x)= \lambda^2(x) \mu_k \big( |\tilde{x}| \big)(\tilde{x},0) \qquad \text{ for } x \notin B_{1/4}.$$
			Notice that for $x \in \big( \spt(\partial T) \cap \UU_2 \big) \subset \big\{x \in \R^{n+k}: x_{n+1} \leq |\tilde{x}|(\beta \varepsilon + \kappa)\big\}$ we have $\lambda(x)=0$, and when $|\tilde{x}| \geq 1$, $\mu_k(|\tilde{x}|)=0$. Hence, $h_k$ vanishes on 
			$$\spt(\partial T) \cup \big(B_{1/4} \cap \spt(T) \big) \cup \big\{x \in \R^{n+k}: x_{n+1} \leq |\tilde{x}|(\beta \varepsilon + \kappa)\big\}$$
			and by \cite[Thereom 3.2]{DeLellisBoundary},  $\displaystyle \int_{\UU_3} \Div_{\overset{\to}{T}} h_k\ \dT=- \int h_k \cdot \overset{\to}{H}_{{T}}\ \dT$.\hfill \refstepcounter{equation}(\theequation)\label{eq:firstVar}
			
			We compute
			\begin{align*}
			\Div_{\overset{\to}{T}} h_k
			&= \sum_{j=1}^{n+1} \left((P- \nuu \otimes \nuu)(2X_j \lambda \mu_k D\lambda +X_j \lambda^2 \frac{\mu_k'}{|\tilde{X}|} (\tilde{X},0) + \ee_j \lambda^2\mu_k) \right)_j\\
			&= 2 \lambda \mu_k \langle (\tilde{X},0), D_{\overset{\to}{T}} \lambda \rangle + \lambda^2 \mu_k' \langle (\tilde{X},0), (P- \nuu \otimes \nuu) \frac{(\tilde{X},0)}{|\tilde{X}|} \rangle + \textnormal{tr}_{n+1}(P- \nuu \otimes \nuu)\lambda^2 \mu_k.
			\end{align*}
			Using \eqref{eq:firstVar}, \eqref{eq:traceEstimate}, \eqref{eq:ProjectionEstimate}, \eqref{eq:normNormal} and \eqref{eq:derivLambda} we find 
			\begin{align*}
			\lim_{k \to \infty} &\int_A  h_k \cdot \overset{\to}{H}_{{T}}\ \dT\\
			&\leq \lim_{k \to \infty} \int_A 4 \lambda \mu_k \left| \frac{\langle \tilde{X}, \tilde{\nuu} \rangle}{|\tilde{X}|} \right| +  \lambda^2 \mu_k' \langle \tilde{X}, (\id- \tilde{\nuu} \otimes \tilde{\nuu}) \frac{\tilde{X}}{|\tilde{X}|} \rangle  + n\lambda^2 \mu_k \dT + \Cl{44} \Aa\\
			&= \int_A 4 \lambda (|\tilde{X}|^{-n}-1) \left| \frac{\langle \tilde{X}, \tilde{\nuu} \rangle}{|\tilde{X}|} \right| +  \lambda^2 n |\tilde{X}|^{-n}-\lambda^2 n |\tilde{X}|^{-n-2} \langle \tilde{X}, \tilde{\nuu} \rangle^2 + n\lambda^2 (|\tilde{X}|^{-n}-1) \dT\\ &
			\quad + \Cr{44} \Aa\\
			&= \int_A \left( 4\lambda (|\tilde{X}|^{-n}-1) \left| \frac{\langle \tilde{X}, \tilde{\nuu} \rangle}{|\tilde{X}|} \right| -\lambda^2 n |\tilde{X}|^{-n-2}\langle \tilde{\nuu}, \tilde{X} \rangle^2 - n\lambda^2 \right) \dT + \Cr{44} \Aa
			\end{align*}
			and hence,
			\begin{align*}
			n \int_A \lambda^2 \ \dT
			&\leq \int_A \left( 4\lambda (|\tilde{X}|^{-n}-1) \left| \frac{\langle \tilde{X}, \tilde{\nuu} \rangle}{|\tilde{X}|} \right| -\lambda^2 n |\tilde{X}|^{-n-2}\langle \tilde{\nuu}, \tilde{X} \rangle^2\right) \dT + \C \Aa\\
			&\leq \C \left( \int_A \lambda  \left| \frac{\langle \tilde{X}, \tilde{\nuu} \rangle}{|\tilde{X}|} \right| \dT +  \Aa \right)\\
			&\leq \frac{n}{2} \int_A \lambda^2 \ \dT + \frac{\Cl{45}}{2} \left( \int_A  \left| \frac{\langle \tilde{X}, \tilde{\nuu} \rangle}{|\tilde{X}|} \right|^2 \dT + \Aa \right).
			\end{align*}
			We conclude
			\begin{align*}
			\int_A \lambda^2 \ \dT \leq \Cr{45} \left( \int_A  \left| \frac{\langle \tilde{X}, \tilde{\nuu} \rangle}{|\tilde{X}|} \right|^2 \dT + \Aa \right).
			\end{align*}
			We argue in the same way to prove the same inequality for 
			$$\tilde{\lambda} := \min \left\{0, \frac{X_{n+1}}{|\tilde{X}|}+\beta \varepsilon +\kappa\right\}.$$
			As the $\spt(\lambda) = \big\{x \in \R^{n+k}: x_{n+1} \geq |\tilde{x}|(\beta \varepsilon + \kappa)\big\}$ and $\spt(\tilde{\lambda}) = \big\{x \in \R^{n+k}: x_{n+1} \leq - |\tilde{x}|(\beta \varepsilon + \kappa)\big\}$, we see that
			$\spt(\lambda^2 +\tilde{\lambda}^2) = \big\{x \in \R^{n+k}: |x_{n+1}| \geq |\tilde{x}|(\beta \varepsilon + \kappa)\big\}$ and hence
			\begin{align*}
			\int_A X_{n+1}^2 &\dT
			\leq \int_A \frac{X_{n+1}^2}{|\tilde{X}|^2} \dT\\
			&= \int_A \left( \frac{X_{n+1}}{|\tilde{X}|} -(\beta \varepsilon+\kappa) \right)\left( \frac{X_{n+1}}{|\tilde{X}|} +(\beta \varepsilon+\kappa) \right) \dT + (\beta \varepsilon+\kappa)^2 \lVert T \rVert(A)\\
			&\leq \int_A \left| \frac{X_{n+1}}{|\tilde{X}|} -(\beta \varepsilon+\kappa) \right| \left| \frac{X_{n+1}}{|\tilde{X}|} +(\beta \varepsilon+\kappa) \right|  \1_{\spt(\lambda^2+\tilde{\lambda}^2)} \dT + (\beta \varepsilon+\kappa)^2 \lVert T \rVert(A)\\
			&\leq \frac{1}{2}\int_A \left( \lambda^2 + \tilde{\lambda}^2 \right) \dT + 2( \beta^2 \varepsilon^2 + \kappa^2)\lVert T \rVert(A)\\
			&\leq  \C \left( \int_A  \left| \frac{\langle \tilde{X}, \nuu \rangle}{|\tilde{X}|} \right|^2 \dT + \Aa \right) + 2( \beta^2 \varepsilon^2 + \kappa^2)\lVert T \rVert(A)\\
			&\leq \Cr{45} \left( \int_A  |X^\perp|^2 |X|^{-n-2} \dT + \Aa \right) + 2( \beta^2 \varepsilon^2 + \kappa^2)\lVert T \rVert(A).
			\end{align*}
			Notice that by the assumption of the lemma
			$$ \int_{B_{1/4}} X_{n+1}^2 \dT
			\leq \frac{\EE_C(T,1)}{\Cr{46}} \lVert T \rVert(B_{1/4})
			= \frac{\EE_C(T,1)}{16} \beta^2 \lVert T \rVert(B_{1/4})
			{\leq \epsilon^2 \beta^2 \lVert T \rVert(B_{1/4}) }.$$
			We use Lemma \ref{lem:excLessX} (with $T$, $\sigma$ replaced by $(\muu_{3\#}T) \LL \UU_3$, $1/2$), \eqref{eq:sptTU1} and Corollary \ref{cor:sphericalExc} (with $s=1$) to deduce
			\begin{align*}
			\varepsilon^2 
			&= \EE_C \big( ({\muu_3}_\# T)\LL \UU_3, 1 \big)\\
			&\leq 4 \Cr{18} \left( \kappa_{({\muu_3}_\# T)\LL \UU_3} + \int_{\BC_{3/2}} X_{n+1}^2 \dd \lVert \muu_{3\#} T \rVert + \Aa_{\muu_3(\M)} \right) \\
			&\leq 4 \cdot 3^n \Cr{18} \left( \kappa + \int_{\BC_{1/2}} X_{n+1}^2 \dT + \Aa \right)\\
			&\leq 3^{n+2} \Cr{18} \bigg( \kappa + \int_{\UU_1} X_{n+1}^2 \dT + \Aa \bigg)\\
			&\leq  3^{n+2} \Cr{18} \left( \Cr{45} \left(\int_A  |X^\perp|^2 |X|^{-n-2} \dT +2\Aa \right)  + 2\MM(T)(\beta^2 \varepsilon^2 + \kappa) \right)\\
			&\leq  3^{n+2} \Cr{18} \left( \Cr{45} \left(\EE_S(T,1) +\Cr{6} \kappa +(2+ \Cr{6})\Aa \right)  + 2\MM(T)(\beta^2 \varepsilon^2 + \kappa) \right)\\
			&\leq \frac{ 3^{2n+3} \Cr{18} \big(1+m\al(n)\big) 16 }{3^{2n+8} \big(1+m\al(n)\big)\Cr{18}} \varepsilon^2
			+ \frac{\Cr{47}}{2} \big(\EE_S(T,1) +\kappa + \Aa \big)\\
			&\leq \frac{\varepsilon^2}{2} + \frac{\Cr{47}}{2}\big(\EE_S(T,1) + \kappa + \Aa \big).
			\end{align*}
		\end{proof}

\bibliographystyle{unsrt}
\bibliography{BoundaryRegularityPaper}

\end{document}